\documentclass{jag-l}

\usepackage[english]{babel}
\usepackage{amsmath,amsthm,amssymb,bm,enumitem,mathtools,thmtools}
\setlist[enumerate]{label=\textup{(\arabic*)}}

\newtheorem{theorem}{Theorem}[section]
\newtheorem{lemma}[theorem]{Lemma}
\newtheorem{proposition}[theorem]{Proposition}
\newtheorem{corollary}[theorem]{Corollary}
\newtheorem{example}[theorem]{Example}

\theoremstyle{definition}
\newtheorem{definition}[theorem]{Definition}

\theoremstyle{remark}
\newtheorem{remark}[theorem]{Remark}
\newtheorem{remarks}[theorem]{Remarks}

\numberwithin{equation}{section}

\usepackage{tikz-cd}
\tikzset{
    labl/.style={anchor=south, rotate=90, inner sep=.5mm}
}
\tikzset{
  symbol/.style={
    draw=none,
    every to/.append style={
      edge node={node [sloped, allow upside down, auto=false]{$#1$}}}
  }
}

\input{macros.sty}

\begin{document}

\title[The Tate conjecture for Gushel--Mukai varieties]{The Tate Conjecture for even dimensional Gushel--Mukai varieties in characteristic $p\geq 5$}

\author{Lie Fu}
\address{Universit\'e de Strasbourg, IRMA $\&$  USIAS, Strasbourg, France}
\email{lie.fu@math.unistra.fr}
\thanks{L.F.\ is partially supported by the Radboud Excellence Initiative and by the Agence Nationale de la Recherche (ANR) under projects ANR-20-CE40-0023 and ANR-16-CE40-0011.}

\author{Ben Moonen}
\address{Radboud University Nijmegen, IMAPP, Nijmegen, The Netherlands}
\email{b.moonen@science.ru.nl}

\begin{abstract} 
We study Gushel--Mukai (GM) varieties of dimension $4$ or~$6$ in characteristic~$p$. Our main result is the Tate conjecture for all such varieties over finitely generated fields of characteristic $p\geq 5$. In the case of GM sixfolds, we follow the method used by Madapusi Pera in his proof of the Tate conjecture for K3 surfaces. As input for this, we prove a number of basic results about GM sixfolds, such as the fact that there are no nonzero global vector fields. For GM fourfolds, we prove the Tate conjecture by reducing it to the case of GM sixfolds by making use of the notion of generalised partners plus the fact that generalised partners in characteristic~$0$ have isomorphic Chow motives in middle degree. Several steps in the proofs rely on results in characteristic~$0$ that are proven in our paper~\cite{FuMoonen-GM1}.  
\end{abstract}

\maketitle

\section{Introduction}

\subsection{}
Gushel--Mukai (GM) varieties form a class of Fano varieties that are interesting for various reasons, not the least of which is their link to hyperk\"ahler geometry. Over fields of characteristic~$0$ they have been extensively studied in recent years; see for instance the work of Iliev--Manivel~\cite{IlievManivel11}, Debarre--Iliev--Manivel \cite{DIM1}, \cite{DIM2}, a series of papers \cite{DK-GMClassif}, \cite{DK-Kyoto}, \cite{DK-Moduli}, \cite{DK-Intermediate} by Debarre and Kuznetsov, and the work of Kuznetsov--Perry~\cite{KuzPerry} and Perry--Pertusi--Zhao~\cite{PPZ}, as well as the many references contained in these papers. By contrast, it seems the study of GM varieties over fields of positive characteristic is largely unexplored terrain. Let us note here that the definition of GM varieties as the smooth projective varieties that can be realised (at least \'etale locally on the base) as intersections
\[
X = \CGr(2,V_5) \cap \PP(W) \cap Q
\]
of the cone over the Grassmannian~$\Grass(2,V_5)$ with a linear space and a quadric, is meaningful over an arbitrary base scheme. We refer to Section~\ref{sec:Basics} for further discussion of some basic properties of GM varieties.

The main goal of the present paper is to explore GM varieties over fields of characteristic~$p$, and to prove, at least when $p \geq 5$, the Tate Conjecture. Our main result is as follows.

\begin{theorem}
\label{thm:TCGMcharp}
Let $k$ be a field of characteristic $p\geq 5$ which is finitely generated over its prime field. Let $Y$ be a Gushel--Mukai variety over~$k$ of dimension~$4$ or~$6$.
\begin{enumerate}
\item\label{TateGMl} For $\ell \neq p$ the Galois representation on $H^\bullet_{\et}(Y_{\kbar},\QQ_\ell)$ is completely reducible, and the cycle class maps
\[
\CH^i(Y_{\kbar}) \otimes \QQ_\ell \to \Tate\bigl(H^{2i}_{\et}(Y_{\kbar},\QQ_\ell(i))\bigr)
\]
are surjective. Here $\Tate\bigl(H^{2i}_{\et}(Y_{\kbar},\QQ_\ell(i))\bigr)$ denotes the space of Tate classes in $H^{2i}_{\et}\bigl(Y_{\kbar},\QQ_\ell(i)\bigr)$; see~\emph{Section~\ref{subsec:TateDef}}.

\item\label{TateGMp} Assume, moreover, that $k$ is a finite field. Then the cycle class maps
\[
\CH^i(Y_{\kbar}) \otimes \QQ_p \to \Tate\bigl(H^{2i}_{\crys}(Y/K_0)(i)\bigr)
\]
are surjective. Here $K_0$ denotes the fraction field of~$W(k)$ and we denote by $\Tate\bigl(H^{2i}_{\crys}(Y/K_0)(i)\bigr)$ the space of crystalline Tate classes in $H^{2i}_{\crys}(Y/K_0)\bigl(i\bigr)$; see~\emph{Section~\ref{subsec:TateDef}}.
\end{enumerate}
\end{theorem}

\subsection{}
Almost the entire paper is devoted to the proof of the main theorem for GM sixfolds; see \ref{subsec:GM4proof} below for a sketch of how GM fourfolds are handled. Motivically, a GM sixfold decomposes as the direct sum of a purely algebraic part (i.e., a direct sum of Tate motives) and a ``variable part'' in degree~$6$ that is of K3 type, in the sense that over the complex numbers its Hodge realisation is of K3 type. Recall that the Tate Conjecture for K3 surfaces over finite fields is now known in all cases, due to the work of Nygaard~\cite{Nygaard} and Ogus--Nygaard~\cite{NygOgus}, Maulik~\cite{Maulik}, Charles~\cite{Charles}, \cite{Charles2}, Madapusi Pera~\cite{MP-K3}, and Madapusi Pera--Kim~\cite{KimMP}.

Our proof of Theorem~\ref{thm:TCGMcharp} for GM sixfolds closely follows the approach used by Madapusi Pera in his proof of the Tate conjecture for K3 surfaces \cite{MP-K3}, which is based on a careful study of families of motives (in the sense of compatible systems of cohomological realizations) that live over integral canonical models of certain Shimura varieties. This takes up Sections \ref{sec:FamMotivesChar0}--\ref{sec:TCGM6} of the paper. 

In a nutshell, the idea of the proof is to show that, for $Y/k$ as in Theorem~\ref{thm:TCGMcharp} with $\dim(Y) = 6$, the space of Tate classes in $H^6(Y)\bigl(3\bigr)$ is spanned by Tate classes~$\xi$ that can be lifted to characteristic~$0$, i.e., for which there exists a lifting of~$Y$ to a GM sixfold~$\scrY$ over a dvr of mixed characteristic $(0,p)$ such that $\xi$ lifts to a Tate class on the generic fibre of~$\scrY$. (Note that for different Tate classes~$\xi$ different liftings of~$Y$ are used; in general there will not exist a single~$\scrY$ to which all Tate classes can be lifted.) This is enough to deduce the Tate conjecture, as the Tate conjecture for GM varieties in characteristic~$0$ is known; see our paper~\cite{FuMoonen-GM1}.

While we globally follow the path laid out by Madapusi Pera, there are some parts that we approach differently. In Section~\ref{sec:MatchCryst}, for instance, we use some results from integral crystalline cohomology as a substitute for an argument of Madapusi Pera that is based on the density of the ordinary locus. (Though we think it is true that for GM sixfolds the ordinary locus in characteristic~$p$ is again dense in the moduli space, such a result does not yet seem to be available in the literature.) In addition, there are some steps in the argument for which we provide details that in \cite{MP-K3} and~\cite{MP-IntCanMod} are omitted or merely sketched.

\subsection{}
Though the work of Madapusi Pera provides us with a clear strategy for the proof of Theorem~\ref{thm:TCGMcharp}, the method requires some basic cohomological results about GM sixfolds that turn out to be not entirely straightforward to prove. We need, for instance, that GM sixfolds in characteristic~$p$ have no nonzero global vector fields. In characteristic~$0$ the proof of this result is not hard (see \cite{KuzPerry}, Appendix~A); but the argument uses in an essential way that group schemes in characteristic~$0$ are reduced, and therefore does not carry over to characteristic~$p$. 

Our proof of the cohomological results that we need, which takes up Sections \ref{sec:CohomGr}--\ref{sec:GlobalVectFields}, exploits that a GM sixfold~$X$ can be described as a double cover of the Grassmannian variety $\Grass(2,V_5)$, branched along a smooth quadric section $Y \subset \Grass(2,V_5)$ (which is a GM fivefold). Via a number of standard short exact sequences and their associated cohomology sequences, we are able to reduce all calculations to questions about cohomology groups of homogeneous sheaves on the Grassmannian. We have been able to carry out these calculations only when $p$ is not too small with respect to the weights that appear in the homogeneous sheaves that we need to consider, and it is in this part of the work that the restriction to characteristic $p\geq 5$ in our main theorem arises. Moreover, our proof that $H^0(X,\scrT_X)$ vanishes is based on a brute force approach (see Section~\ref{sec:GlobalVectFields}) that requires some computer algebra. While in comparison with an earlier version of this paper we have been able to reduce the dependency on Magma calculations to a minimum, it would be of interest to find a proof by pure thought.

\subsection{}
\label{subsec:GM4proof}
As will be clear from the above summary, it would be cumbersome to use the same arguments to prove Theorem~\ref{thm:TCGMcharp} for GM fourfolds. In principle, all arguments in Sections~\ref{sec:FamMotivesChar0}--\ref{sec:TCGM6} still work for GM fourfolds, so this is not where the problem lies. However, geometrically GM fourfolds are less accessible than sixfolds, and this makes it a challenge to prove all required cohomological results. (For instance, it is not known to the authors if $H^0(X,\scrT_X)$ vanishes for GM fourfolds in characteristic~$p$.) 

Fortunately, there is a way to reduce the Tate conjecture for GM fourfolds to the case of GM sixfolds. This relies on the idea, which plays an important role in the work of Debarre and Kuznetsov, that GM varieties can be described in terms of suitable (multi-)linear algebra data, known as ``GM data sets'', which in turn are in correspondence with other data $(V_6,V_5,A)$, known as ``Lagrangian data sets''. The idea, then, is that for a given GM fourfold~$X$ over a field~$k$ of characteristic $p>2$, we can, at least after passing to a finite extension of~$k$, find a lifting of~$X$ to a GM fourfold~$\scrX$ over a dvr~$R$ of characteristic~$(0,p)$ with residue field~$k$, as well as a GM sixfold~$\scrX^\prime$ over~$R$, such that the generic fibres $\scrX_K$ and~$\scrX^\prime_K$ are generalised partners, which means that $\bigl(V_6(\scrX_K),A(\scrX_K)\bigr) \cong \bigl(V_6(\scrX^\prime_K),A(\scrX^\prime_K)\bigr)$. (We refer to Section~\ref{sec:TCGM4} for details.) As proven in our paper~\cite{FuMoonen-GM1}, the fact that $\scrX_K$ and~$\scrX^\prime_K$ are generalised partners implies that (possibly after extending the base field) we have an isomorphism of Chow motives $\frh^4(\scrX_K)\bigl(2\bigr) \cong \frh^6(\scrX^\prime_K)\bigl(3\bigr)$. By specialisation this gives us an isomorphism of Chow motives $\frh^4(X)\bigl(2\bigr) \cong \frh^6(X^\prime)\bigl(3\bigr)$, with $X^\prime$ the special fibre of~$\scrX^\prime$, and the Tate conjecture for~$X$ then follows from the (now proven) Tate conjecture for the GM sixfold~$X^\prime$.

\subsection{}
To conclude this introduction, let us remark that for sufficiently general GM varieties (of dimension~$4$ or~$6$), an alternative approach to the Tate Conjecture in characteristic $p \geq 5$ is to establish a motivic connection to a suitable hyperk\"ahler variety (a double EPW sextic) and then use a result of Charles (\cite{Charles}, Theorem~4). It is not clear to us if this approach can work for all even-dimensional GM varieties, as in general the corresponding double EPW sextic is singular.

\subsection{Acknowledgement.} We thank the referee for his or her pertinent remarks on the paper and for several corrections and useful suggestions.

\section{Basics on Gushel--Mukai varieties}
\label{sec:Basics}

\subsection{}
In what follows, we use the geometric interpretation of projective spaces and projective bundles; so if $W$ is a vector space, $\PP(W)$ is the variety of lines in~$W$. (Likewise for projective bundles.)

If $S$ is a scheme and $V$ is a vector bundle of rank $r\geq 3$ on~$S$, we denote by $\Grass(2,V) \subset \PP(\wedge^2 V)$ the Grassmannian of $2$-planes in~$V$ (in its Pl\"ucker embedding). Abbreviating $\Grass = \Grass(2,V)$ and writing $\pi \colon \Grass \to S$ for the structural morphism, we have the universal subbundle and quotient bundle
\[
0 \tto \scrS \tto \pi^* V \tto \scrQ \tto 0\, ,
\]
and we define $\scrO_{\Grass}(1) = \det(\scrQ)$. The Pl\"ucker embedding is given by the complete linear system~$|\scrO_{\Grass}(1)|$.

We write $\CGr(2,V) \subset \PP(\scrO_S \oplus \wedge^2 V)$ for the cone over~$\Grass(2,V)$. If it is clear from the context which~$V$ we mean, we abbreviate $\CGr(2,V)$ to~$\CGr$.

\begin{definition}\label{def:GMVariety}
Let $S$ be a scheme on which $2$ is invertible. By a \emph{Gushel--Mukai (GM) variety of dimension $n \in \{3,4,5,6\}$ over~$S$} we mean a smooth proper $S$-scheme $f \colon X \to S$ of relative dimension~$n$ with geometrically integral fibres such that \'etale locally on~$S$, writing $V_5 = \scrO_S^5$, there exist:
\begin{itemize}
\item a submodule $W \subset  \scrO_S \oplus \wedge^2 V_5$ which is locally on~$S$ a direct summand of rank~$n+5$,
\item a relative quadric $Q \subset \PP(\scrO_S \oplus \wedge^2 V_5)$,
\end{itemize}
such that
\begin{equation}
\label{eq:CGrPWQ}
X \cong \CGr(2,V_5) \cap \PP(W) \cap Q
\end{equation}
as $S$-schemes.
\end{definition}

In everything that follows, whenever we consider Gushel--Mukai varieties, it is assumed that $2$ is invertible on the base scheme.

Gushel--Mukai varieties over a field come in two flavours: if $S$ is the spectrum of a field and $X$ is given as in~\eqref{eq:CGrPWQ} we say that $X$ is:
\begin{itemize}
\item \emph{of Mukai type} if $\PP(W)$ does not pass through the vertex $O \in \CGr$ of the cone, i.e., if $k \oplus (0) \subset k \oplus \wedge^2 V_5$ is not contained in~$W$; 

\item \emph{of Gushel type} if $\PP(W)$ passes through the vertex~$O$. 
\end{itemize}
In the work of Debarre--Kuznetsov (see for instance \cite{DK-GMClassif}, \cite{DK-Kyoto}), GM varieties of Mukai type (resp.\ of Gushel type) are called ``ordinary'' (resp.\ ``special''). We prefer to avoid this terminology, as ``ordinary'' already has a well-established meaning in algebraic geometry over fields of characteristic~$p$. 

As GM sixfolds play a main role in this paper, it is perhaps good to point out that if $k$ is a field and $X/k$ is a GM sixfold which is given in the above form, we have $W = k \oplus \wedge^2 V_5$, so that $X$ is just a quadric section of~$\CGr$. In particular, GM sixfolds are necessarily of Gushel type.

\subsection{}
Let $X$ be a GM variety over a field~$k$ which is given in the form \eqref{eq:CGrPWQ}. The assumption that $X$ is smooth over~$k$ implies that $X$ does not pass through the vertex of~$\CGr$. Further, $\CGr \cap \PP(W)$ and $Q \cap \PP(W)$ are nonsingular at all points of~$X$, of dimensions~$n+1$ and~$n+3$ respectively. A resolution of~$\scrO_X$ by a complex of locally free $\scrO_{\PP(W)}$-modules is obtained by taking the tensor product of resolutions of $Q \cap \PP(W)$ and $\CGr \cap \PP(W)$; this yields a resolution
\begin{multline}
\label{eq:ResolOX}
\qquad 0 \tto \scrO(-7) \tto \bigl(V_5^\vee \oplus k\bigr) \otimes \scrO(-5) \tto \bigl(V_5 \otimes \scrO(-4)\bigr) \oplus  \bigl(V_5^\vee \otimes \scrO(-3)\bigr)\\
\tto \bigl(V_5 \oplus k\bigr) \otimes \scrO(-2) \tto \scrO \tto \scrO_X \tto 0\qquad
\end{multline}
(See also formula~(2.4) in~\cite{DK-GMClassif}.)

\begin{lemma}
\label{lem:LiftingX}
Let $X$ be a Gushel--Mukai variety over an algebraically closed field~$k$ of characteristic~$p$. Then $X$ can be lifted to a Gushel--Mukai variety over~$W(k)$.
\end{lemma}

\begin{proof}
This is clear from the description~\eqref{eq:CGrPWQ}.
\end{proof}

\begin{lemma}
\label{lem:Hi(OX)}
Let $f\colon X \to S$ be a Gushel--Mukai variety. Then $R^i f_*\scrO_X = 0$ for all $i \geq 1$.
\end{lemma}

\begin{proof}
It suffices to prove this when $S$ is the spectrum of a field. Now use \eqref{eq:ResolOX} and the fact that $\omega_{\PP(W)} \cong \scrO(m)$ with $m < -7$ (in fact $m = -(n+5)$, where $n = \dim(X) \in \{3,4,5,6\}$).
\end{proof}

\begin{proposition}
\label{prop:PicX}
Let $X/S$ be a Gushel--Mukai variety of relative dimension~$n$. Then we have an isomorphism $\Pic_{X/S} \cong \ZZ_S$ (the constant group scheme~$\ZZ$ over~$S$) such that $\omega_{X/S}$ corresponds to the section $-(n-2)$. If $X$ is given as in~\eqref{eq:CGrPWQ}, the pullback of~$\scrO_{\PP(W)}(1)$ corresponds to the section~$1$.
\end{proposition}

\begin{proof}
Over~$\CC$ this is well-known; see \cite{DK-GMClassif}, Lemma~2.29. This implies the result over an arbitrary field~$k$ of characteristic~$0$, because $\omega_{X/k}$ gives a non-zero $k$-valued point of~$\Pic_{X/k}$, so that the latter cannot be a twisted form of~$\ZZ$. Next assume $\charact(k) = p>2$. Then $H^1(X,\scrO_X) = 0$ gives that $\Pic_{X/k}$ is an \'etale group scheme. Lift~$X_{\kbar}$ to a Gushel--Mukai variety~$\scrX$ over~$W(\kbar)$; then $H^2(X,\scrO_X) = 0$ implies that every line bundle on~$X_{\kbar}$ can be lifted to a line bundle on~$\scrX$. With the previous argument this again gives $\Pic_{X/k} \cong \ZZ$. The general case now follows because all fibres of~$\Pic_{X/S}$ are isomorphic to~$\ZZ$, \'etale locally on the basis it has a section that fibrewise gives~$1$, and there is a global section (given by~$\omega_{X/S}$) that is fibrewise non-zero.
\end{proof}

\begin{remark}
If the section $1 \in \Pic_{X/S}(S)$ is represented by a line bundle, we will denote that line bundle by~$\scrO_X(1)$. We will use this notation only in situations (for instance when working over an algebraically closed field) where it is clear that such a line bundle exists. The case $\dim(X) = 3$ is special in this regard, because then $\omega_{X/S} = \scrO_X(-1)$; but for $\dim(X) \in \{4,5,6\}$ there may in general be a Brauer obstruction for representing $1 \in \Pic_{X/S}(S)$ by an actual line bundle.
\end{remark}

\section{Cohomology of sheaves on the Grassmannian}
\label{sec:CohomGr}

\subsection{}
We work over a field~$k$ of characteristic~$p$. We will later assume $p$ is not too small.

Let $G = \SL_5$ over~$k$, let $B$, $B_- \subset G$ be the upper and lower triangular Borel subgroups and $T = B \cap B_-$ the diagonal maximal torus. Then $X^*(T) = \ZZ^5/\ZZ\cdot (1,\ldots,1)$. We write vectors in $X^*(T)$ in the form $[a_1,\ldots,a_5]$. Let $e_i \in X^*(T)$ be the class of the $i$th standard base vector in~$\ZZ^5$.

The roots of~$G$ are the vectors $e_i - e_j$ for $i \neq j$; the positive roots are those with $i<j$. The simple roots are the $\alpha_i = e_i - e_{i+1}$ for $i=1,\ldots,4$, and the half-sum of the positive roots is the vector
\[
\rho = [2,1,0,-1,-2]\, .
\]

Let $\langle\ ,\ \rangle \colon \bigl(X^*(T) \otimes \RR\bigr) \times \bigl(X_*(T) \otimes \RR\bigr) \to \RR$ denote the natural pairing. If $\alpha = e_i - e_j$ is a root and $\alpha^\vee$ is the corresponding coroot, we have $\bigl\langle[a_1,\ldots,a_5],\alpha^\vee\bigr\rangle = a_i - a_j$.

We identify the Weyl group~$W$ of~$G$ with the symmetric group~$\frS_5$ in the usual way, and for $w \in W$ we denote by~$\ell(w)$ the length of~$w$ (with respect to the set of simple reflections $s_i = (i\quad i+1)$).

Let $X^*(T)_+ \subset X^*(T)$ be the set of dominant weights. Concretely, a vector $[a_1,\ldots,a_5] \in X^*(T)$ is in $X^*(T)_+$ if and only if $a_1 \geq a_2 \geq \cdots \geq a_5$. (Note that $X^*(T)$ is the weight lattice, as $G$ is simply connected.)

\subsection{}
Consider the flag variety~$G/B_-$. For $\lambda \in X^*(T)$, we denote by $H^*(\lambda)$ the cohomology of the homogeneous line bundle~$\scrL(\lambda)$ on~$G/B_-$ given by~$\lambda$. We shall make use of some basic results about these cohomology groups, which we now recall.

Inside $X^*(T) \otimes \RR$, consider the shifted closed Weyl chamber
\[
\overline{\frC} = \bigl\{ \lambda \in X^*(T) \otimes \RR\bigm| 0 \leq \langle \lambda + \rho,\alpha^\vee\rangle\quad \text{for all $\alpha \in \Phi^+$} \bigr\}\, ,
\]
and the basic $p$-alcove
\[
\overline{\frA} = \bigl\{ \lambda \in X^*(T) \otimes \RR\bigm| 0 \leq \langle \lambda + \rho,\alpha^\vee\rangle \leq p \quad \text{for all $\alpha \in \Phi^+$} \bigr\}\, .
\]

Recall that we have a ``dot action'' of $W = \frS_5$ on $X^*(T) \otimes \RR$, given by the rule $w \bullet \lambda = w(\lambda + \rho) - \rho$.

The following result can be found in \cite{Jantzen}, Section~II.5.5.

\begin{proposition}
\label{prop:pAlcove}
Let $\mu \in X^*(T)$.
\begin{enumerate}
\item\label{pAlc1} If $\mu \in \overline{\frA}$ but $\mu \notin X^*(T)_+$ then $H^*(w\bullet \mu) = 0$ for all $w \in W$.
\item\label{pAlc2} If $\mu \in \overline{\frA}\cap X^*(T)_+$ then for all $w \in W$ and $i\geq 0$ we have
\[
H^i(w \bullet \mu) \cong \begin{cases} H^0(\mu) & \text{if $i = \ell(w)$;}\\ 0 & \text{otherwise.}\end{cases}
\]
\item\label{Kempf} (Kempf's vanishing theorem) If $\lambda$ is a dominant weight then $H^i(\lambda) = 0$ for all $i>0$.
\end{enumerate}
\end{proposition}

\begin{remark}
Given $\lambda \in X^*(T)$, there exists a $w \in W$ such that $w\bullet \lambda \in \overline{\frC}$; for this write $\lambda + \rho = [a_1,\ldots,a_5]$ and choose $w \in W = \frS_5$ such that $a_{w^{-1}(1)} \geq a_{w^{-1}(2)} \geq \cdots \geq a_{w^{-1}(5)}$. We then have $w\bullet \lambda \in  \overline{\frA}$ if and only if $a_{\max} - a_{\min} \leq p$, where $a_{\max}$ and~$a_{\min}$ are the maximum and minimum of $\{a_1,\ldots,a_5\}$.
\end{remark}

\subsection{}
Let $P \supset B_-$ be the standard parabolic subgroup (with respect to~$B_-$) corresponding to $\{\alpha_1,\alpha_2,\alpha_4\}$. Concretely, $P$ is the subgroup of matrices $A = (a_{ij})$ in~$G$ with the property that $a_{ij} = 0$ for all $1 \leq i\leq 3$ and $4\leq j\leq 5$. The homogeneous variety $G/P$ is the Grassmannian $\Grass = \Grass(2,5)$ of $2$-planes in a fixed $5$-dimensional space.

Let $\pi \colon G/B_- \to \Grass$ be the projection. If $V$ is a quasi-coherent sheaf on~$\Grass$ then $H^*(\Grass,V) \cong H^*(G/B_-,\pi^*V)$; see \cite{Demazure}, Corollary~1 to Theorem~3. 

Note that $\pi^*\scrO_{\Grass}(1) = \scrL(\lambda)$ with $\lambda = [1,1,1,0,0] = [0,0,0,-1,-1]$.

\begin{proposition}
\label{prop:CohomGr}
As above, let $\Grass = \Grass(2,5)$ over a field $k$ of characteristic~$p$.
\begin{enumerate}
\item\label{HTGr} We have $M_5(k)/k\cdot \id \isomarrow H^0(\Grass,\scrT_{\Grass})$ and $H^i(\Grass,\scrT_{\Grass}) = 0$ for all $i\geq 1$.
\item\label{HTGr-1-2} We have $H^*\bigl(\Grass,\scrT_{\Grass}(-1)\bigr) =  H^*\bigl(\Grass,\scrT_{\Grass}(-2)\bigr) = 0$.
\item\label{hiO(m)Grass} If $H^i\bigl(\Grass,\scrO_{\Grass}(m)\bigr) \neq 0$ then either $i=0$ and $m \geq 0$ or $i=6$ and $m\leq -5$.
\end{enumerate}
For the next assertions, assume $p\geq 5$.
\begin{enumerate}[resume]
\item\label{hijGrass} All Hodge numbers $h^{i,j}(\Grass) = h^j(\Grass,\Omega^i_{\Grass})$ are the same as in characteristic zero; concretely: $h^{i,j}(\Grass) = 0$ unless $0 \leq i=j \leq 6$, and the Hodge numbers $h^{i,i}(\Grass)$ are
\[
1\, ,\quad 1\, ,\quad  2\, ,\quad 2\, ,\quad 2\, ,\quad 1\, ,\quad 1\, .
\]

\item\label{Omega1-1} We have
\begin{align*}
& H^*\bigl(\Grass,\Omega^1_{\Grass}(-1)\bigr) = H^*\bigl(\Grass,\Omega^1_{\Grass}(-2)\bigr) = 0\, ,\\
& H^*\bigl(\Grass,\Omega^2_{\Grass}(-1)\bigr) = H^*\bigl(\Grass,\Omega^2_{\Grass}(-2)\bigr) = 0\, ,\\
& H^*\bigl(\Grass,\Omega^3_{\Grass}(-1)\bigr) = 0\, .
\end{align*}

\item\label{Om2(-3)}
We have $H^i\bigl(\Grass,\Omega^2_{\Grass}(-3)\bigr) = 0$ for $i \neq 5$, and $h^5\bigl(\Grass,\Omega^2_{\Grass}(-3)\bigr) = 5$.
\end{enumerate}
\end{proposition}

\begin{proof}
Part \ref{HTGr} is \cite{Demazure}, Proposition~2. For \ref{HTGr-1-2} we use that $\pi^* \scrT_{\Grass}(-1)$ admits a filtration
\[
\pi^* \scrT_{\Grass}(-1) = \scrE_0 \supset \scrE_1 \supset \cdots \supset \scrE_5 \supset 0
\]
with graded pieces isomorphic to $\scrL(\lambda)$ for
\[
\lambda = e_1 + e_4\, ,\quad e_2 + e_4\, ,\quad e_3 + e_4\, ,\quad e_1+e_5\, ,\quad e_2+e_5\, ,\quad e_3+e_5\, .
\]
In all cases there exists a simple root~$\alpha$ such that $\langle \alpha^\vee,\lambda\rangle = -1$. (Take $\alpha = e_3 -e_4$ in the first two cases, $\alpha = e_2 -e_3$ in the third, and $\alpha = e_4 - e_5$ in the remaining three cases.) By \cite{Demazure}, Lemma~1(a), it follows that in each case $H^*\bigl(G/B,\scrL(\lambda)\bigr) = 0$, and this gives that the cohomology of~$\scrT_{\Grass}(-1)$ is zero. (In~\cite{Demazure} the assumption is made that $G$ is an adjoint group, so that only weights~$\lambda$ in the root lattice are allowed. However, the results are valid for all reductive groups~$G$.)

For $\scrT_{\Grass}(-2)$ we can use almost the same argument. This time the weights that occur are
\begin{align*}
&e_1 + 2e_4 + e_5\, , &&e_2 + 2e_4 + e_5\, , &&e_3 + 2e_4 + e_5\, ,\\
&e_1 + e_4 + 2e_5\, , &&e_2 + e_4 + 2e_5\, , &&e_3 + e_4 + 2e_5\, .
\end{align*}
For the last five of these, one again readily finds simple roots~$\alpha$ such that $\langle \alpha^\vee,\lambda\rangle = -1$. For $\lambda = e_1 + 2e_4 + e_5$ this does not work. However, in this case the vanishing of $H^*\bigl(G/B,\scrL(\lambda)\bigr)$ follows from \cite{Demazure}, Lemma~1(d), taking $\alpha = e_3-e_4$ and $\beta = e_2-e_3$.

\ref{hiO(m)Grass} We use that $\pi^*\scrO_{\Grass}(m) \cong \scrL(\lambda_m)$ with $\lambda_m = [m,m,m,0,0]$. For $m\geq 0$ this is a dominant weight, so Kempf's vanishing theorem gives the assertion. Next we observe that $\langle \lambda_{-1}, \alpha_3^\vee\rangle = -1$ and
\[
\langle \lambda_{-2}, \alpha_3^\vee\rangle = -2\, ,\quad
\langle \lambda_{-2}, \alpha_4^\vee\rangle = 0\, ,\quad
\langle \alpha_3, \alpha_4^\vee\rangle = -1\, .
\]
By \cite{Demazure}, Lemma~1, these imply that $H^*\bigl(\Grass,\scrO_{\Grass}(-1)\bigr) = H^*\bigl(\Grass,\scrO_{\Grass}(-2)\bigr) = 0$. As $\omega_{\Grass} = \scrO_{\Grass}(-5)$, the remaining cases follow by Serre duality.

For \ref{hijGrass}, it suffices to show that $h^{i,j}(\Grass)$ has the expected value for $i+j< \dim(\Grass) = 6$; Serre duality then gives the same for $i+j>6$, and because we know that each $\Omega^i_{\Grass}$ has the same Euler characteristic as in characteristic~$0$ this gives the stated result.

The strategy of proof is again to use that $\pi^*\Omega^i_{\Grass}$ has a filtration by homogeneous subbundles whose graded quotients are of the form~$\scrL(\lambda)$. The assumption that $p\geq 5$ will guarantee that for all weights~$\lambda$ that occur there is a $w \in W$ such that $w\bullet \lambda \in \overline{\frA}$, and we can apply Proposition~\ref{prop:pAlcove}. For instance, the sequence of weights that occurs in $\pi^*\Omega^1_{\Grass}$ is
\begin{align*}
& \lambda_1 = [-1,0,0,0,1]\, ,\quad \lambda_2 = [-1,0,0,1,0]\, ,\quad \lambda_3 = [0,-1,0,0,1]\, ,\\
& \lambda_4 = [0,-1,0,1,0]\, ,\quad \lambda_5 = [0,0,-1,0,1]\, ,\quad \lambda_6 = [0,0,-1,1,0]\, .
\end{align*}
The corresponding vectors $\lambda_j + \rho$ are
\begin{align*}
& [1,1,0,-1,-1]\, ,\quad [1,1,0,0,-2]\, ,\quad [2,0,0,-1,-1]\, ,\\
& [2,0,0,0,-2]\, ,\quad [2,1,-1,-1,-1]\, ,\quad [2,1,-1,0,-2]\, .
\end{align*}
For $j \neq 6$ we see that $\lambda_j \in \overline{\frA}$ but $\lambda_j \notin X^*(T)$; hence $H^*(\lambda_j) = 0$. For $j=6$, taking $w = (3\; 4)$ gives $w\bullet \lambda_6 = 0$ and Proposition~\ref{prop:pAlcove}\ref{pAlc2} gives that $h^1(\lambda_6) = 1$ and $h^i(\lambda_6) = 0$ for $i\neq 1$. In total this gives that $h^i(\Grass,\Omega^1_{\Grass}) = 0$ for $i\neq 1$ and $h^1(\Grass,\Omega^1_{\Grass}) = 1$, as desired.

For the cohomology of $\Omega^2_{\Grass}$ the argument is similar. There are twelve weights~$\lambda$ that occur (three of which have multiplicity~$2$), namely all combinations $\lambda_i + \lambda_j$ with $i\neq j$. In all but two cases, taking $w \in W$ such that $w\bullet \lambda \in \overline{\frA}$ results in a vector that is not a dominant weight, in which case $H^*(\lambda) = 0$ by Proposition~\ref{prop:pAlcove}\ref{pAlc1}. The two weights for which the behaviour is different are $\lambda = \lambda_4 + \lambda_6 = [0,-1,-1,2,0]$ and $\lambda = \lambda_5 + \lambda_6 = [0,0,-2,1,1]$; for these, taking $w = (2\; 3\; 4)$ (resp.\  $w = (3\; 5\; 4)$) gives $w\bullet \lambda = 0$, and in both cases this gives a $1$-dimensional contribution to $H^2(\Grass,\Omega^2_{\Grass})$, as $\ell(w) =2$.

For the cohomology of $\Omega^3_{\Grass}$ the calculation is even simpler. Again we write down all weights that occur; there are sixteen of them, two of which occur with multiplicity~$3$. We know we can only get a contribution to $H^j(\Grass,\Omega^3_{\Grass})$ with $j\leq 2$ if there exists an element $w \in W$ of length at most~$2$ for which $w\bullet \lambda \in \overline{\frA}$. This happens for eight of the weights but in none of these cases $w \bullet \lambda$ is a dominant weight. Hence $H^j(\Grass,\Omega^3_{\Grass}) = 0$ for $j\leq 2$.

Finally, for $\Omega^4_{\Grass}$ it never happens that there is an element $w \in W$ of length $\ell(w)$ at most~$1$ such that $w\bullet \lambda \in \overline{\frA}$, and for $\Omega^5_{\Grass}$ none of the six weights that occur is dominant. This concludes the proof of~\ref{hijGrass}.

For \ref{Omega1-1} the argument is the same: in the case of $\Omega^1_{\Grass}(-1)$ and $\Omega^1_{\Grass}(-2)$, the weights that occur are the $\lambda_i^\prime = \lambda_i + [0,0,0,1,1]$, respectively the $\lambda_i^\prime = \lambda_i + [0,0,0,2,2]$. In each case one again finds a Weyl group element~$w$ such that $w \bullet \lambda_i^\prime \in \overline{\frA}$ but $w \bullet \lambda_i^\prime$ is not dominant. To give at least some details, here is what happens for~$\Omega^2_{\Grass}(-2)$. As before, $\pi^*\Omega^2_{\Grass}(-2)$ has a filtration by homogeneous subbundles whose graded quotients are of the form~$\scrL(\lambda)$. In the next table we list all weights $\lambda$ that occur, we calculate $\lambda + \rho$, give an element $w \in W$ such that $w(\lambda+\rho)$ is dominant, and we give $w \bullet \lambda = w(\lambda+\rho) - \rho$.
\[
\begin{array}{ccccc}
\lambda & \lambda +\rho & w & w(\lambda+\rho) & w\bullet \lambda \\[6pt]
{}[-1,-1,0,4,2] & [1,0,0,3,0] & (1\ 2\ 3\ 4) & [3,1,0,0,0] & [1,0,0,1,2] \\
{}[-1,0,-1,4,2] & [1,1,-1,3,0] & (1\ 2\ 3\ 4) & [3,1,1,0,-1] & [1,0,1,1,1]\\
{}[-2,0,0,3,3] & [0,1,0,2,1] & (1\ 4)(3\ 5) & [2,1,1,0,0] & [0,0,1,1,2]\\
{}[-1,-1,0,3,3] & [1,0,0,2,1] & (1\ 2\ 4)(3\ 5) & [2,1,1,0,0] & [0,0,1,1,2]\\
{}[-1,0,-1,3,3] & [1,1,-1,2,1] & (1\ 2\ 3\ 5\ 4) & [2,1,1,1,-1] & [0,0,1,2,1]\\
{}[0,-1,-1,4,2] & [2,0,-1,3,0] & (1\ 2\ 3\ 5\ 4) & [3,2,0,0,-1] & [1,1,0,1,1]\\
{}[0,-2,0,3,3] & [2,-1,0,2,1] & (2\ 5\ 3\ 4) & [2,2,1,0,-2] & [0,1,1,1,1]\\
{}[0,-1,-1,3,3] & [2,0,-1,2,1] & (2\ 4)(3\ 5) & [2,2,1,0,-2] & [0,1,1,1,1]\\
{}[0,0,-2,3,3] & [2,1,-2,2,1] & (2\ 3\ 5\ 4) & [2,2,1,1,-2] & [0,1,1,2,0]\\
{}[-1,-1,0,2,4] & [1,0,0,1,2] & (1\ 2\ 4\ 3\ 5) & [2,1,1,0,0] & [0,0,1,1,2]\\
{}[-1,0,-1,2,4] & [1,1,-1,1,2] & (1\ 2\ 3\ 5) & [2,1,1,1,-1] & [0,0,1,2,1]\\
{}[0,-1,-1,2,4] & [2,0,-1,1,2] & (2\ 4\ 3\ 5) & [2,2,1,0,-1] & [0,0,1,1,1]
\end{array}
\]
Once again we see that in all cases $w\bullet \lambda \in \overline{\frA}$ but $w \bullet \lambda$ is not dominant. The remaining (but very similar) cases of $\Omega^2_{\Grass}(-1)$ and $\Omega^3_{\Grass}(-1)$ are left to the reader.

For \ref{Om2(-3)}, let us again give the weights that occur.
\[
\begin{array}{ccccc}
\lambda & \lambda +\rho & w & w(\lambda+\rho) & w\bullet \lambda \\[6pt]
{}[-1,-1,0,5,3] & [1,0,0,4,1] && [4,1,0,0,1] & [2,0,0,1,3] \\
{}[-1,0,-1,5,3] & [1,1,-1,4,1] && [4,1,1,1,-1] & [2,0,1,2,1]\\
{}[-2,0,0,4,4] & [0,1,0,3,2] && [3,2,1,0,0] & [1,1,1,1,2]\\
{}[-1,-1,0,4,4] & [1,0,0,3,2] && [3,2,1,0,0] & [1,1,1,1,2]\\
{}[-1,0,-1,4,4] & [1,1,-1,3,2] && [3,2,1,1,-1] & [1,1,1,2,1]\\
{}[0,-1,-1,5,3] & [2,0,-1,4,1] & (1\ 2\ 4)(3\ 5) & [4,2,1,0,-1] & [2,1,1,1,1]\\
{}[0,-2,0,4,4] & [2,-1,0,3,2] && [3,2,2,0,-2] & [1,1,2,1,1]\\
{}[0,-1,-1,4,4] & [2,0,-1,3,2] && [3,2,2,0,-2] & [1,1,2,1,1]\\
{}[0,0,-2,4,4] & [2,1,-2,3,2] && [3,2,2,1,-2] & [1,1,2,2,0]\\
{}[-1,-1,0,3,5] & [1,0,0,2,3] && [3,2,1,0,0] & [1,1,1,1,2]\\
{}[-1,0,-1,3,5] & [1,1,-1,2,3] && [3,2,1,1,-1] & [1,1,1,2,1]\\
{}[0,-1,-1,3,5] & [2,0,-1,2,3] && [3,2,2,0,-1] & [1,1,1,1,1]
\end{array}
\]
In most cases we find an element $w \bullet \lambda \in \overline{\frA}$ that is not dominant (in these cases we do not list~$w$). The only exception to this is found on the sixth line; in this case $w$ is an element of length $\ell(w) = 5$, and the line bundle on $G/B_-$ corresponding to $w\bullet \lambda = [2,1,1,1,1] = [1,0,0,0,0]$ has a $5$-dimensional space of global sections.
\end{proof}

\begin{remark}
It is not known to us if the results in this section (and the next) are still valid over fields of characteristic $2$ or~$3$. In principle this could be checked using Macaulay2 or similar computer algebra packages. While we have been able to do some of these verifications, to get complete results the required calculations seem to require more computing power than we have had at our disposal.
\end{remark}

\section{Cohomology of $6$-dimensional Gushel--Mukai varieties}
\label{sec:CohomGM6}

\subsection{}
\label{subsec:XYGrSetup}
Let $k$ be an algebraically closed field of characteristic~$p\geq 5$. We abbreviate $\Grass = \Grass(2,V_5)$, where $V_5$ is a $5$-dimensional $k$-vector space. Consider a $6$-dimensional GM variety~$X$ over~$k$, which we assume is given in the form $X = \CGr(2,V_5) \cap Q$ (cf.\ the comments made after Definition~\ref{def:GMVariety}). Projection from the vertex of the cone over~$\Grass$ gives a morphism $\gamma \colon X \to \Grass$ of degree~$2$. Let $\iota \colon Y \hookrightarrow X$ be the reduced ramification divisor and $\iota^\prime \colon Y^\prime \hookrightarrow \Grass$ the branch divisor, so that we have a commutative diagram
\[
\begin{tikzcd}
Y \ar[r,hook,"\iota"] \ar[d,"\wr"]& X \ar[d,"\gamma"]  \\
Y^\prime \ar[r,hook,"\iota^\prime"] & \Grass
\end{tikzcd}
\]
We shall later often drop the notational distinction between $Y$ and~$Y^\prime$. Of course, $Y \cong Y^\prime$ is a $5$-dimensional GM variety of Mukai type, and $Y^\prime \subset \Grass$ is cut out by a single quadratic equation; see Section~2.1 in~\cite{DK-Kyoto}.

Recall that on $\Grass$ we have a bundle~$\scrO_{\Grass}(1)$, namely the pull-back of the $\scrO(1)$ on $\PP(\wedge^2 V_5)$, or what is the same: the determinant of the universal quotient bundle. On $X$ we also have a line bundle~$\scrO_X(1)$: the unique ample generator of $\Pic(X) \cong \ZZ$. Similarly we have $\scrO_Y(1)$, and $\iota^*\scrO_X(1) = \scrO_Y(1)$.

We have $\gamma^* \scrO_{\Grass}(1) = \scrO_X(1)$. Further, $\gamma_*\scrO_X \cong \scrO_{\Grass} \oplus \scrO_{\Grass}(-1)$, with algebra structure given by a global section of $\scrO_{\Grass}(2)$ that has $Y^\prime$ as zero scheme. So $\scrO_{\Grass}(Y^\prime) = \scrO_{\Grass}(2)$ and $\scrO_X(Y) = \scrO_X(1)$.

\subsection{}
In the proof of the next results we make repeated use of the following exact sequences
\begin{align}
& 0 \tto \Omega^{i-1}_Y(m-2) \tto \Omega^i_{\Grass}(m)|_Y \tto \Omega^i_Y(m) \tto 0 \label{ses:OmYOmGr}\, ,\\
& 0 \tto \Omega^i_{\Grass}(m-2) \tto \Omega^i_{\Grass}(m) \tto \Omega^i_{\Grass}(m)|_Y \tto 0\, , \label{ses:OmGrOmGr}
\end{align}
where we drop the notational distinction between $Y$ and~$Y^\prime$. (We trust this will not lead to confusion.)

\begin{proposition}
Assume $p\geq 5$. All Hodge numbers $h^{i,j}(Y) = h^j(Y,\Omega^i_Y)$ of $Y$ are the same as in characteristic~$0$, i.e., the Hodge diamond of~$Y$ is given by
\[
\begin{smallmatrix}
&&&&& 1 &&&&&\\
&&&& 0 && 0 &&&&\\
&&& 0 && 1 && 0 &&&\\
&& 0 && 0 && 0 && 0 &&\\
& 0 && 0 && 2 && 0 && 0 &\\
0 && 0 && 10 && 10 && 0 && 0\\
& 0 && 0 && 2 && 0 && 0 &\\
&& 0 && 0 && 0 && 0 &&\\
&&& 0 && 1 && 0 &&&\\
&&&& 0 && 0 &&&&\\
&&&&& 1 &&&&&\\
\end{smallmatrix}
\]
\end{proposition}

\begin{proof}
By Serre duality it suffices to calculate $h^{i,j}(Y)$ for $i\leq 2$. For $i=0$ we can use Lemma~\ref{lem:Hi(OX)}. The short exact sequence
\[
0 \tto \scrO_{\Grass}(-4) \tto \scrO_{\Grass}(-2) \tto \iota^\prime_*\scrO_Y(-2) \tto 0
\] 
together with Proposition~\ref{prop:CohomGr}\ref{hiO(m)Grass} gives $H^*\bigl(Y,\scrO_Y(-2)\bigr) = 0$. Sequence \eqref{ses:OmYOmGr} with $i=1$ and $m=0$ therefore gives $H^i(Y,\Omega^1_Y) \cong H^i(\Grass,\Omega^1_{\Grass}|_Y)$. Using sequence \eqref{ses:OmGrOmGr} with $i=1$ and $m=0$ together with Proposition~\ref{prop:CohomGr}\ref{Omega1-1} we then find that $h^{1,j}(Y) = 1$ if $j=1$, and $=0$ otherwise.

By Serre duality and Proposition~\ref{prop:CohomGr}\ref{HTGr-1-2} we have $H^*\bigl(\Grass,\Omega^1_{\Grass}(-4)\bigr) = 0$. Combined with part~\ref{Omega1-1} of that proposition we obtain $H^*\bigl(Y,\Omega^1_{\Grass}(-2)|_Y \bigr) = 0$. The sequence \eqref{ses:OmYOmGr} with $i=1$ and $m=-2$ then gives
\[
h^i\bigl(Y,\Omega^1_Y(-2)\bigr) = h^{i+1}\bigl(Y,\scrO_Y(-4)\bigr) = h^{4-i}\bigl(Y,\scrO_Y(1)\bigr) = \begin{cases} 10 & \text{for $i=4$,}\\ 0 & \text{otherwise.} \end{cases}
\]
(To see that $10$ is the correct value we can use the Euler characteristic.) Next use sequence \eqref{ses:OmGrOmGr} with $i=2$ and $m=0$. By Proposition~\ref{prop:CohomGr} we get
\[
h^i(Y,\Omega^2_{\Grass}|_Y) = h^i(\Grass,\Omega^2_{\Grass}) = \begin{cases} 2 & \text{for $i=2$,}\\ 0 & \text{otherwise,} \end{cases}
\]
and \eqref{ses:OmYOmGr} with $i=2$ and $m=0$ then gives the Hodge numbers~$h^{2,j}(Y)$.
\end{proof}

\subsection{}
Our next goal is to calculate the Hodge numbers of~$X$, using the exact sequences
\begin{align}
& 0 \tto \gamma^* \Omega^i_{\Grass} \tto \Omega^i_X \tto \iota_*\Omega^{i-1}_Y(-1) \tto 0\, . \label{ses:OmGrOmX}
\end{align}
As $\gamma_*\scrO_X = \scrO_{\Grass} \oplus \scrO_{\Grass}(-1)$, pushing down to~$\Grass$ gives an exact sequence
\begin{align}
& 0 \tto \Omega^i_{\Grass} \oplus \Omega^i_{\Grass}(-1) \tto \gamma_*\Omega^i_X \tto \iota^\prime_*\Omega^{i-1}_Y(-1) \tto 0\, . \label{ses:g*OmGrOmX}
\end{align}

We shall also use Raynaud's version of Kodaira--Akizuki--Nakano vanishing; this says that if $Z/k$ is a smooth projective variety which lifts to~$W_2(k)$ and $L$ is an ample line bundle on~$Z$, then $H^j(Z,\Omega^i_{Z/k}\otimes L^{-1}) = 0$ for $i+j < \min\{p,\dim(Z)\}$. See \cite{DeligneIllusie}, Corollary~2.8.

\begin{proposition}
Assume $p\geq 5$. The Hodge numbers $h^{i,j}(X)$ are the same as in characteristic~$0$, i.e., the Hodge diamond of~$X$ is given by
\[
\begin{smallmatrix}
&&&&&& 1 &&&&&&\\
&&&&& 0 && 0 &&&&&\\
&&&& 0 && 1 && 0&&&&\\
&&& 0 && 0 && 0 && 0&&&\\
&&0 && 0 && 2 && 0 && 0&&\\
& 0 && 0 && 0 && 0 && 0&& 0&\\
0 &&0 && 1 && 22 && 1 && 0&& 0\\
& 0 && 0 && 0 && 0 && 0&& 0&\\
&&0 && 0 && 2 && 0 && 0&&\\
&&& 0 && 0 && 0 && 0&&&\\
&&&& 0 && 1 && 0&&&&\\
&&&&& 0 && 0 &&&&&\\
&&&&&& 1 &&&&&&\\
\end{smallmatrix}
\]
\end{proposition}

\begin{proof}
For $i=0$ we use Lemma~\ref{lem:Hi(OX)}. For $i=1$, use sequence~\eqref{ses:g*OmGrOmX} with $i=1$. We have $H^*\bigl(Y,\scrO_Y(-1)\bigr) = 0$ because $H^*\bigl(\Grass,\scrO_{\Grass}(-1)\bigr) = H^*\bigl(\Grass,\scrO_{\Grass}(-3)\bigr) = 0$ by Proposition~\ref{prop:CohomGr}\ref{hiO(m)Grass}. Further, $H^*\bigl(\Grass,\Omega^1_{\Grass}(-1)\bigr) = 0$ by part~\ref{Omega1-1} of that proposition. Hence
\[
h^{1,j}(X) = h^{1,j}(\Grass) = \begin{cases} 1 & \text{for $j=1$,}\\ 0 & \text{otherwise,} \end{cases}
\]

Next consider $\Omega^2_X$. Take sequence \eqref{ses:g*OmGrOmX} with $i=2$. If $j\leq 3$ Raynaud's vanishing theorem gives $H^j\bigl(Y,\Omega^1_Y(-1)\bigr) = 0$, and because $H^*\bigl(\Grass,\Omega^2_{\Grass}(-1)\bigr) = 0$ (see Prop.~\ref{prop:CohomGr}\ref{Omega1-1}) we get $h^{2,j}(X) = 0$ for $j=0,1,3$ and $h^{2,2}(X) = h^{2,2}(\Grass) = 2$.

For $j \leq 2$ we find, using sequence \eqref{ses:g*OmGrOmX} with $i=3$, that $H^j(X,\Omega^3_X) = 0$ because we have $H^j(\Grass,\Omega^3_{\Grass}) = H^j\bigl(\Grass,\Omega^3_{\Grass}(-1)\bigr) = 0$ by Proposition~\ref{prop:CohomGr}, and $H^j\bigl(Y,\Omega^2_Y(-1)\bigr) = 0$ by Raynaud's vanishing theorem.

The above cases, together with Serre duality, give all Hodge numbers except $h^{3,3}(X)$; it then follows that $h^{3,3}(X) = 22$ because the Euler characteristic of~$\Omega^3_X$ is the same as in characteristic~$0$.
\end{proof}

\begin{corollary}
\label{cor:dRCohGM}
Let $S$ be a scheme in which $6$ is invertible. Let $f\colon X \to S$ be a Gushel--Mukai variety of dimension~$6$.
\begin{enumerate}
\item All sheaves $R^jf_*\Omega^i_{X/S}$ are locally free, and their formation commutes with base change.
\item The relative de Rham cohomology sheaves $H^m_{\dR}(X/S) = R^mf_* \Omega^\bullet_{X/S}$ on $S$ are locally free, and their formation commutes with base change.
\item The Hodge--de Rham spectral sequence
\[
E_1^{i,j} =  R^jf_*\Omega^i_{X/S} \quad \Rightarrow\quad H^{i+j}_{\dR}(X/S)
\]
degenerates at the $E_1$-page.
\end{enumerate}
\end{corollary}

\subsection{}
\label{subsec:vfieldxiA}
Next we want to prove some results about the cohomology of the tangent bundle of~$X$. As in the above, we start from results on the Grassmannian, then obtain results about the Gushel--Mukai fivefold~$Y$, and then deduce results about~$X$.

The group $\PGL(V_5)$ acts on $\Grass = \Grass(2,V_5)$. The isomorphism $\frpgl(V_5) = \End(V_5)/k\cdot \id \isomarrow H^0(\Grass,\scrT_{\Grass})$ of Proposition~\ref{prop:CohomGr}\ref{HTGr} is obtained by taking derivatives. For $A \in \End(V_5)$, let $\xi_A$ be the global vector field on~$\Grass$ given by the class of~$A$.

For the following results, the notation is as in Section~\ref{subsec:XYGrSetup}.

\begin{lemma}
\label{lem:CohomOY}
We have $h^0\bigl(Y,\scrO_Y(1)\bigr) = 10$ and $h^i\bigl(Y,\scrO_Y(1)\bigr) = 0$ for all $i\geq 1$. Similarly, $h^0\bigl(Y,\scrO_Y(2)\bigr) = 49$ and $h^i\bigl(Y,\scrO_Y(2)\bigr) = 0$ for all $i\geq 1$.
\end{lemma}

\begin{proof}
For the assertions about~$\scrO_Y(2)$, use the short exact sequence $0 \tto \scrO_{\Grass} \tto \scrO_{\Grass}(2) \tto \scrO_Y(2) \tto 0$ plus Proposition~\ref{prop:CohomGr}\ref{hiO(m)Grass}, and note that $h^0\bigl(\Grass,\scrO_{\Grass}(2)\bigr) = 50$. The same method works for~$\scrO_Y(1)$.
\end{proof}

\begin{proposition}
\label{prop:CohomTY}
Assume the base field~$k$ has characteristic $p \geq 5$ (or $\charact(k) =0$). Then $h^1(Y,\scrT_Y) = 25$ and $H^i(Y,\scrT_Y) = 0$ for all $i \neq 1$.
\end{proposition}

The proof of this result will be given in the next section.

\begin{remark}
\label{rem:p=5}
We do not know if the results in this section are still valid in characteristic $2$ and~$3$. As we shall explain, the only remaining obstruction is the verification that certain explicitly given varieties are singular. In principle this can be done using computer algebra, but the running time of these calculations exceeds our patience. This proposition is the only missing ingredient for the proof of Theorem~\ref{thm:TCGMcharp} in characteristic~$<5$.
\end{remark}

\begin{corollary}
With $Y/k$ as in \emph{Proposition~\ref{prop:CohomTY}}, the automorphism group scheme $\Aut_Y$ is finite \'etale.
\end{corollary}

\begin{proposition}
\label{prop:CohomTX}
Let $X$ be a Gushel--Mukai variety of dimension~$6$ over a field~$k$ of characteristic $p \geq 5$ (or $p=0$). Then $h^1(X,\scrT_X) = 25$ and $H^0(X,\scrT_X) = 0$ for all $i \neq 1$.
\end{proposition}

\begin{proof}
In characteristic~$0$ the assertion can be found in \cite{KuzPerry}, Appendix~A; so we may assume $k=\kbar$ and $p \geq 5$. Then $X$ is given as in~\eqref{eq:CGrPWQ} and the notation of Section~\ref{subsec:XYGrSetup} applies.

Pushing down the short exact sequence
\[
0 \tto \scrT_X \tto \gamma^*\scrT_{\Grass} \tto \iota_*\scrO_Y(2) \tto 0
\] 
we get
\[
0 \tto \gamma_* \scrT_X \tto \scrT_{\Grass} \oplus \scrT_{\Grass}(-1) \tto \iota^\prime_* \scrO_{Y^\prime}(2) \tto 0\, .
\]
By Proposition~\ref{prop:CohomGr} and Lemma~\ref{lem:CohomOY} it follows that $H^i(X,\scrT_X) = 0$ for $i \geq 2$, and that we have an exact sequence
\[
0 \tto H^0(X,\scrT_X) \tto H^0(\Grass,\scrT_{\Grass})\xrightarrow{~\beta~} H^0\bigl(Y,\scrO_Y(2)\bigr) \tto H^1(X,\scrT_X) \tto 0\, .
\]
By Proposition~\ref{prop:CohomGr}\ref{HTGr} and Lemma~\ref{lem:CohomOY} it suffices to show that $\beta$ is injective.

The map~$\beta$ sends the class of a vector field~$\xi$ to the global section of the normal bundle $\scrN_{Y \subset \Grass} \cong \scrO_Y(2)$ that is obtained by projecting~$\xi|_Y$ to~$\scrN_{Y \subset \Grass}$. Because $H^0(Y,\scrT_Y) = 0$ by Proposition~\ref{prop:CohomTY}, it follows from the exact sequence $0 \tto \scrT_{\Grass}(-2) \tto \scrT_{\Grass} \tto \scrT_{\Grass}|_{Y} \tto 0$ together with Proposition~\ref{prop:CohomGr}\ref{HTGr-1-2} that $\beta$ is injective. This gives the proposition.
\end{proof}

\begin{corollary}
\label{cor:AutX}
With $X/k$ as in the proposition, the automorphism group scheme $\Aut_X$ is finite \'etale.
\end{corollary}

\begin{corollary}
\label{cor:DefoX}
Let $X/k$ be as in the proposition, with $k$ a perfect field of characteristic $p \geq 5$. Then the formal deformation functor $\Defo(X)$ is pro-representable by a formal scheme which is (non-canonically) isomorphic to the formal spectrum of $W[\![t_1,\ldots,t_{25}]\!]$.
\end{corollary}

The last result in this section concerns the calculation of a Kodaira--Spencer map; this will play a crucial role in our proof of the Tate conjecture. We retain the notation of~\ref{subsec:XYGrSetup}, but we drop the notational distinction between $Y$ and~$Y^\prime$ (and $\iota$ and~$\iota^\prime$).

It will be convenient to use the notation $H^{i,j}(X) = H^j(X,\Omega^i_X)$, and similarly for the other varieties involved. On $H^{3,3}$ we consider the pairing $H^{3,3} \times H^{3,3} \to H^{6,6} = k$ given by  cup-product. Define $H^{3,3}(X)_{00} \subset H^{3,3}(X)$ to be the orthogonal complement of the image of the map $\gamma^* \colon H^{3,3}(\Grass) \to H^{3,3}(X)$ (which, as we will see, is injective).

\begin{lemma}
\label{lem:HijX}
Assume $p \geq 5$.
\begin{enumerate}
\item\label{H1TXTY} The natural maps
\[
H^1(X,\scrT_X) \to H^1(Y,\scrT_X|_Y) \leftarrow H^1(Y,\scrT_Y)
\]
are isomorphisms.

\item\label{H33X00} The map $\gamma^* \colon H^{3,3}(\Grass) \to H^{3,3}(X)$ is injective and the sequence \eqref{ses:OmGrOmX} with $i=3$ gives rise to an isomorphism
\[
H^{3,3}(X)_{00} \isomarrow H^3\bigl(Y,\Omega^2_Y(-1)\bigr)\, .
\]

\item\label{H24X} The sequence \eqref{ses:OmGrOmX} with $i=2$ gives rise to an isomorphism
\[
H^{2,4}(X) \isomarrow H^4\bigl(Y,\Omega^1_Y(-1)\bigr)\, .
\]

\item The above isomorphisms fit into a commutative diagram
\[
\begin{tikzcd}
H^1(X,\scrT_X) \times H^{3,3}(X)_{00} \ar[r,"\cup"] \ar[d,"\wr"] & H^{2,4}(X) \ar[d,"\wr"]\\
H^1(Y,\scrT_Y) \times H^3\bigl(Y,\Omega^2_Y(-1)\bigr) \ar[r,"\cup"] & H^4\bigl(Y,\Omega^1_Y(-1)\bigr)
\end{tikzcd}
\]
\end{enumerate}
\end{lemma}

\begin{proof}
\ref{H1TXTY} First we show that the short exact sequence $0 \tto \scrO_Y(-1) \tto  \Omega^1_X|_Y \tto  \Omega^1_Y \tto 0$ splits: restriction of 
\[
0 \tto \gamma^*\Omega^1_{\Grass} \tto \Omega^1_X \tto \iota_*\scrO_Y(-1) \tto 0
\]
to~$Y$ gives a homomorphism $\Omega^1_X|_Y \to \scrO_Y(-1)$ which splits the sequence. Because $H^1\bigl(Y,\scrO_Y(1)\bigr) = 0$ by Lemma~\ref{lem:CohomOY}, it follows that $H^1(Y,\scrT_Y) \to H^1(Y,\scrT_X|_Y)$ is an isomorphism.

The sequence
\[
0 \tto \gamma_* \scrT_X(-1) \tto \scrT_{\Grass}(-1) \oplus \scrT_{\Grass}(-2) \tto \iota^\prime_* \scrO_{Y^\prime}(1) \tto 0
\]
together with Proposition~\ref{prop:CohomGr}\ref{HTGr-1-2} and Lemma~\ref{lem:CohomOY} gives that $H^2\bigl(X,\scrT_X(-1)\bigr) \cong H^1\bigl(Y,\scrO_Y(1)\bigr) = 0$. Using the sequence $0 \tto \scrT_X(-1) \tto \scrT_X \tto \scrT_X|_Y \tto 0$ and the fact that $h^1(X,\scrT_X) = 25 = h^1(Y,\scrT_X|_Y)$, it follows that $H^1(X,\scrT_X) \to H^1(Y,\scrT_X|_Y)$ is an isomorphism.

For \ref{H33X00} and \ref{H24X}, apply~$\gamma_*$ to \eqref{ses:OmGrOmX}, use that $H^j\bigl(Y,\Omega^i_Y(-1)\bigr) = 0$ for $i+j < 5$ (Raynaud's vanishing theorem), and use Proposition~\ref{prop:CohomGr} parts \ref{hijGrass} and~\ref{Omega1-1}. The last assertion is just a diagram chase.
\end{proof}

\begin{proposition}
\label{prop:KodSpencX}
Let $X$ be a $6$-dimensional Gushel--Mukai variety over a field~$k$ of characteristic $p \geq 5$ (or $p=0$). Then the map
\[
H^1(X,\scrT_X) \to \Hom\bigl(H^{3,3}(X)_{00},H^{2,4}(X)\bigr)
\]
given by cup-product is surjective.
\end{proposition}

\begin{proof}
The sequence $0 \tto \Omega^2_{\Grass}(-3) \tto \Omega^2_{\Grass}(-1) \tto \Omega^2_{\Grass}(-1)|_Y \tto 0$ gives an exact sequence
\[
H^3\bigl(\Omega^2_{\Grass}(-1)\bigr) \tto H^3\bigl(\Omega^2_{\Grass}(-1)|_Y \bigr) \tto H^4\bigl(\Grass,\Omega^2_{\Grass}(-3)\bigr)\, .
\]
As the outer terms vanish by Proposition~\ref{prop:CohomGr} it follows that $H^3\bigl(\Omega^2_{\Grass}(-1)|_Y \bigr) = 0$. By \eqref{ses:OmYOmGr} with $i=2$ and $m=-1$ we find that
\[
H^3\bigl(Y,\Omega^2_Y(-1)\bigr) \to H^4\bigl(Y,\Omega^1_Y(-3)\bigr)
\]
is injective. Now consider the commutative diagram
\[
\begin{tikzcd}
H^1(Y,\scrT_Y) \times H^3\bigl(Y,\Omega^2_Y(-1)\bigr) \ar[r] \ar[d,hook] & H^4\bigl(Y,\Omega^1_Y(-1)\bigr) \ar[d]\\
H^1(Y,\scrT_Y) \times H^4\bigl(Y,\Omega^1_Y(-3)\bigr) \ar[r] & H^5(Y,\omega_Y)
\end{tikzcd}
\]
(with $\omega_Y = \scrO_Y(-3)$), where the vertical maps come from the sequences~\eqref{ses:OmYOmGr}. By Serre duality the bottom row is a perfect pairing. Further, we have $h^4\bigl(Y,\Omega^1_Y(-1)\bigr) = h^{2,4}(X) = 1$. It now follows that the rightmost vertical map is an isomorphism (otherwise it would be zero, but this gives a contradiction), and it follows that
\begin{multline*}
H^1(Y,\scrT_Y) \isomarrow \Hom\bigl(H^4\bigl(Y,\Omega^1_Y(-3)\bigr), H^5(Y,\omega_Y) \bigr)\\ \longtwoheadrightarrow \Hom\bigl(H^3\bigl(Y,\Omega^2_Y(-1)\bigr), H^4\bigl(Y,\Omega^1_Y(-1)\bigr) \bigr)\, .
\end{multline*}
Combining this with the last part of Lemma~\ref{lem:HijX}, this gives what we want.
\end{proof}

\section{Global vector fields}
\label{sec:GlobalVectFields}

In this section we give a proof of Proposition~\ref{prop:CohomTY}. The proof relies on computer algebra calculations. It would be desirable to find a proof by pure thought.

The result in characteristic~$0$ is known, see \cite{KuzPerry}, Appendix~A. It therefore suffices to prove the proposition over an algebraically closed field of characteristic $p \geq 5$.

\subsection{}
\label{ssec:VFieldsGr}
Let $k$ be a field whose characteristic is not in $\{2,3\}$. Let $V_5$ be a $5$-dimensional $k$-vector space, and let
\[
\Grass = \Grass(2,V_5) \hookrightarrow \PP = \PP(\wedge^2 V_5)
\]
be the Grassmannian of $2$-planes in~$V_5$. The group $\GL(V_5)$ acts on $\Grass$, and by Proposition~\ref{prop:CohomGr}\ref{HTGr} this action induces an isomorphism $\frpgl(V_5) \isomarrow H^0(\Grass,\scrT_{\Grass})$. We write $\xi_A$ for the vector field given by a class $[A] \in \frpgl(V_5)$.

Let $X$ and~$Y$ be as in Section~\ref{subsec:XYGrSetup}. We use the short exact sequences
\begin{align*}
&0 \tto \scrT_Y \tto \scrT_{\Grass}|_{Y} \tto \scrO_Y(2) \tto 0\, ,\\
&0 \tto \scrT_{\Grass}(-2) \tto \scrT_{\Grass} \tto \scrT_{\Grass}|_{Y} \tto 0\, .
\end{align*}
Using the second sequence and the first two parts of Propositon~\ref{prop:CohomGr}, we find that $H^0(\Grass,\scrT_{\Grass}) \isomarrow H^0(Y,\scrT_{\Grass}|_{Y})$, which has dimension~$24$, and that $H^i(Y,\scrT_{\Grass}|_{Y}) = 0$ for $i>0$. Together with Lemma~\ref{lem:CohomOY} we find that $H^i(Y,\scrT_Y) = 0$ for $i>1$, and we are left with an exact sequence
\[
0 \tto H^0(Y,\scrT_Y) \tto H^0(Y,\scrT_{\Grass}|_{Y}) \xrightarrow{~\beta~} H^0\bigl(Y,\scrO_Y(2)\bigr) \tto H^1(Y,\scrT_Y) \tto 0\, .
\]
The map~$\beta$ is induced by the projection from $\scrT_{\Grass}|_{Y}$ to the normal bundle $\scrN_{Y \subset \Grass} \cong \scrO_Y(2)$. We therefore need to show that a nonzero global vector field~$\xi_A$ on~$\Grass$ cannot be tangent to~$Y$ at every point of~$Y$.

\subsection{}
\label{subsec:Kermu}
The algebraic group $\GL(V_5)$ acts on $\Sym^2(\wedge^2 V_5^\vee)$. Let
\[
I(\Grass) = H^0\bigl(\PP,\scrI_{\Grass}(2)\bigr) \subset \Sym^2(\wedge^2 V_5^\vee)
\]
be the subspace of quadrics that vanish on~$\Grass$, which is a $5$-dimensional subrepresentation. If $e_1,\ldots,e_5$ is an ordered basis of~$V_5$, the elements $x_{ij} = \check{e}_i \wedge \check{e}_j$ with $1 \leq i < j \leq 5$ form a basis of $\wedge^2 V_5^\vee$, and then $I(\Grass)$ is spanned by the Pl\"ucker quadrics $q_1,\ldots,q_5$ with
\[
q_1 = x_{23} x_{45} - x_{24} x_{35} + x_{25} x_{34}\, ,\quad \ldots\quad ,\quad q_5 = x_{12} x_{34} - x_{13} x_{24} + x_{14} x_{23}\, .
\]

The multiplication map $\wedge^2 V_5^\vee \otimes \wedge^2 V_5^\vee \to \wedge^4 V_5^\vee$ is symmetrical and therefore induces a morphism of $\GL(V_5)$-modules
\[
\mu \colon \Sym^2(\wedge^2 V_5^\vee) \to \wedge^4 V_5^\vee\, .
\]
As $\mu(q_i) = 3 \cdot \bigl(\check{e}_1 \wedge \cdots \wedge \widehat{\check{e}_i} \wedge \cdots \wedge \check{e}_5\bigr)$ and $3 \in k^*$ by assumption, it follows that we have a decomposition
\begin{equation}
\label{eq:Sym2Wedge2Dec}
\Sym^2(\wedge^2 V_5^\vee) \cong I(\Grass) \oplus \Ker(\mu)\, .
\end{equation}

\subsection{}
Let the $5$-dimensional Gushel--Mukai variety~$Y$ be given as $Y = \Grass \cap Q$ where $Q$ is a quadric in~$\PP$ that does not contain~$\Grass$. The line $k\cdot Q \subset \Sym^2(\wedge^2 V_5^\vee)/I(\Grass)$ is uniquely determined by~$Y$. (Here, and in what follows, we use the same letter~$Q$ for a quadric in~$\PP$ and a defining equation in $\Sym^2(\wedge^2 V_5^\vee)$. We trust this will not lead to confusion.) Identifying $\Sym^2(\wedge^2 V_5^\vee)/I(\Grass)$ with $\Ker(\mu) \subset \Sym^2(\wedge^2 V_5^\vee)$ using the decomposition~\eqref{eq:Sym2Wedge2Dec}, we can view $k \cdot Q$ as a line in~$\Ker(\mu)$.

Let $A \in \frgl(V_5)$, and let $\xi_A$ be the corresponding vector field on~$\Grass$. (See~\ref{ssec:VFieldsGr}.) Then we see that $\xi_A$ is everywhere tangent to~$Y$ if and only if the action of~$A$ on $\Sym^2(\wedge^2 V_5^\vee)/I(\Grass)$ preserves the line~$k\cdot Q$. In what follows we denote the action of~$\frgl(V_5)$ on $\Sym^2(\wedge^2 V_5^\vee)$ by $(A,Q) \mapsto A \circ Q$.

\begin{proposition}
\label{prop:AQ=0}
Let $Q \in \Sym^2(\wedge^2 V_5^\vee)$ be such that $Y = \Grass \cap Q$ is non-singular of dimension~$5$. Consider the following assertions:
\begin{enumerate}
\item\label{H0TYneq0} There exist nonzero global vector fields on~$Y$, i.e., $H^0(Y,\scrT_Y) \neq 0$.
\item\label{AQ=lambdaQ} There exists a non-scalar $A \in \frgl(V_5)$ and a constant $\lambda \in k^*$ such that $A \circ Q = \lambda \cdot Q$.
\item\label{AQ=0} There exists a non-scalar $A \in \frgl(V_5)$ such that $A \circ Q = 0$.
\end{enumerate}
Then \ref{H0TYneq0} $\Leftarrow$ \ref{AQ=lambdaQ} $\Leftrightarrow$ \ref{AQ=0}, and if $Q \in \Ker(\mu) \subset \Sym^2(\wedge^2 V_5^\vee)$ then also \ref{H0TYneq0} $\Rightarrow$~\ref{AQ=lambdaQ}.
\end{proposition}

\begin{proof}
It follows from Section~\ref{ssec:VFieldsGr}, combined with what was said before the proposition that \ref{AQ=lambdaQ} implies~\ref{H0TYneq0}. If $Q \in \Ker(\mu) \subset \Sym^2(\wedge^2 V_5^\vee)$ then also the converse is true, because $\Ker(\mu) \cong \Sym^2(\wedge^2 V_5^\vee)/I(\Grass)$ as representations of~$\frgl(V_5)$. Because we work in characteristic $\neq 2$, the equivalence of \ref{AQ=lambdaQ} and~\ref{AQ=0} follows from the remark that $(A + a\cdot \mathrm{Id}) \circ Q = A\circ Q - (4a \cdot Q)$ for all $a \in k$ and $Q \in \Sym^2(\wedge^2 V_5^\vee)$.
\end{proof}

\subsection{}
Let $e_1,\ldots,e_5$ be the standard basis of $\VV = \ZZ^5$. As before, write $x_{ij} = \check{e}_i \wedge \check{e}_j$ with $1 \leq i < j \leq 5$; these elements (ordered lexicographically) give an ordered basis of~$\wedge^2 \VV^\vee$, and the monomials $x_{ij}x_{lm}$ with $(i,j) \leq_{\mathrm{lex}} (l,m)$ form a basis of $\Sym^2(\wedge^2 \VV^\vee)$. Let $\scrM$ be the set of these monomials.

Consider the standard action of the Lie algebra $\frgl_5$ (over~$\ZZ$) on~$\VV$, and its induced action on $\Sym^2(\wedge^2 \VV^\vee)$, which we again denote by $(A,Q) \mapsto A \circ Q$. For $m$, $m^\prime \in \scrM$, let $C_{m,m^\prime}$ denote the corresponding matrix coefficient, which is a polynomial with integral coefficients in the entries~$A_{i,j}$ of the matrix~$A$.

As above, let $Y$ be a non-singular $5$-dimensional Gushel--Mukai variety over~$k$ which is obtained as $Y = \Grass \cap Q$. Via the choice of a basis of~$V_5$ we identify $V_5 = \VV \otimes k$. For $A \in \frgl_{5,k}$ the equation $A \circ Q = 0$ can be written as a matrix equation over~$k$:
\begin{equation}
\label{eq:MatrixEq}
\Bigl(C_{m,m^\prime}(A_{i,j})\Bigr)_{m,m^\prime \in \scrM} \circ \Bigl(Q_m\Bigr)_{m\in \scrM} = 0\, ,
\end{equation}
where $(Q_m)_{m \in \scrM}$ denotes the expression of~$Q$ as a column vector with respect to the basis~$\scrM$. If $Q$ is given, this can be viewed as an system of equations for the matrix coefficients~$A_{i,j}$.

The following simple observation is crucial for what follows.

\begin{proposition}
\label{prop:NoLiftAQ}
Let $k = \kbar$ be an algebraically closed field of characteristic $p\geq 5$. Let $V_5 = \VV \otimes k$. Let $Q \in \Sym^2(\wedge^2 V_5^\vee)$ be such that $Y = \Grass \cap Q$ is non-singular of dimension~$5$. If $A \in \frgl_{5,k}$ is a non-scalar matrix such that the pair $(A,Q)$ satisfies~\eqref{eq:MatrixEq}, this solution does not admit a lift to characteristic~$0$, i.e., there does not exist a domain~$R$ with fraction field of characteristic~$0$ and a morphism $R \to k$ plus lifts $\tilde{A} \in \frgl_{5,R}$ of~$A$ and $\tilde{Q} \in \Sym^2(\wedge^2 (\VV\otimes R)^\vee)$ of~$Q$ such that the pair $(\tilde{A},\tilde{Q})$ again satisfies equation~\eqref{eq:MatrixEq}.
\end{proposition}

\begin{proof}
Suppose such a lift $(\tilde{A},\tilde{Q})$ exists. Let $K$ be the fraction field of~$R$. Then $\tilde{Y} = \Grass \cap \tilde{Q}$ defines a $5$-dimensional Gushel--Mukai variety over~$K$ (non-singular of dimension~$5$ because it specializes to~$Y$, which is non-singular of dimension~$5$), and $\tilde{A}$ gives a non-zero vector field on~$\tilde{Y}$. This is impossible, because by the results in \cite{KuzPerry}, Appendix~A, Gushel--Mukai varieties over~$K$ (of characteristic~$0$) have no nonzero global vector fields.
\end{proof}

\subsection{}
{}From now on, $k = \kbar$ denotes an algebraically closed field of characteristic $p\geq 5$. As before, let $V_5$ be a $5$-dimensional $k$-vector space, and consider a quadric $Q \in \Sym^2(\wedge^2 V_5^\vee)$ such that $Y = \Grass \cap Q$ is non-singular of dimension~$5$. (We continue to use the same symbol~$Q$ for both the homogeneous polynomial of degree~$2$ and the quadric in $\PP = \PP(\wedge^2 V_5)$ that it defines.)

\begin{lemma}
\label{lem:DiagA}
With $Q$ as above, if there exists a non-scalar $A \in \frgl(V_5)$ with $A \circ Q = 0$, there also exists a diagonalizable non-scalar $A \in \frgl(V_5)$ with $A \circ Q = 0$.
\end{lemma}

\begin{proof}
Choose a basis of~$V_5$ such that $A$ is in Jordan normal form, say
\[
A = \left(\begin{smallmatrix} B_1 &&&\\ &B_2&& \\ &&\ddots& \\ &&&B_r\end{smallmatrix}\right)\, ,
\]
with $B_i$ a Jordan block with eigenvalue~$\lambda_i$. (Here of course $1\leq r\leq 5$, and if $r=5$ then $A$ is a diagonal matrix and there is nothing left to prove.) Let $V_5 = E_1 \oplus \cdots \oplus E_r$ be the corresponding decomposition of~$V_5$ into generalised eigenspaces, and let $A^\prime$ be the transformation of~$V_5$ given by $\lambda_i \cdot \mathrm{Id}$ on~$E_i$.

We have an induced decomposition
\[
\wedge^2 V_5^\vee = U_1 \oplus \cdots \oplus U_s\, ,\qquad s = r + \binom{r}{2}\, ,
\]
where $U_1,\ldots,U_s$ are the spaces $\wedge^2 E_i^\vee$ and the $E_i^\vee \otimes E_j^\vee$ for $i<j$. This induces a decomposition
\begin{equation}
\label{eq:Sym2Dec}
\Sym^2\bigl(\wedge^2 V_5^\vee\bigr) = \bigl(\mathop{\oplus}\limits_i\; \Sym^2(U_i)\bigr)\, \bigoplus\; \bigl(\mathop{\oplus}\limits_{l<m}\; U_l \otimes U_m\bigr)\, .
\end{equation}
In the Lie algebra representation of $\frgl_5$ on~$\Sym^2\bigl(\wedge^2 V_5^\vee\bigr)$, the actions of $A$ and of~$A^\prime$ both respect the decomposition~\eqref{eq:Sym2Dec}. Moreover, on each summand $A^\prime$ acts by a scalar multiplication and $A - A^\prime$ is nilpotent. It follows that~$Q$, which by assumption is a solution of the equation $A \circ Q = 0$, is also a solution of the equation $A^\prime \circ Q = 0$.

The only thing left to show is that $A^\prime$ is not a scalar multiple of the identity. So suppose $A$ is of the form
\[
A = \left(\begin{smallmatrix} \lambda & * &&&\\ & \lambda & * && \\ &&\lambda&*&\\ &&&\lambda&* \\ &&&&\lambda \end{smallmatrix}\right)\, ,
\]
where the four off-diagonal entries marked as~$*$ are either $0$ or~$1$, and all remaining coefficients are zero. In this case the matrix of~$A$ in the Lie algebra action on $\Sym^2\bigl(\wedge^2 V_5^\vee\bigr)$ is upper triangular, with all diagonal coefficients equal to~$-4\lambda$. The assumption that there is a nonzero~$Q$ with $A \circ Q = 0$ therefore implies that $\lambda = 0$. In the remaining cases (choices of $* \in \{0,1\}$), one checks that for any solution~$Q$ of $A \circ Q = 0$ the pair $(A,Q)$ admits a lift to characteristic~$0$, violating Proposition~\ref{prop:NoLiftAQ}. (Of course, this last calculation can be simplified by using the decomposition~\eqref{eq:Sym2Dec}, and one readily sees that trivial Jordan blocks can be ignored, so that only a couple of cases have to be checked.)
\end{proof}

\subsection{}
By combining Proposition~\ref{prop:AQ=0} and Lemma~\ref{lem:DiagA}, to prove Proposition~\ref{prop:CohomTY} it suffices to show that if $Y = \Grass \cap Q$ is non-singular of dimension~$5$, there does not exist a nonzero diagonal matrix $A = \diag(-a_1,\ldots,-a_5)$ such that $A \circ Q = 0$. With respect to the basis $x_{ij}x_{lm}$ of $\Sym^2\bigl(\wedge^2 V_5^\vee\bigr)$, the action of~$A$ is diagonal:
\[
A \circ x_{ij}x_{lm} = (a_i+a_j+a_l+a_m) \cdot x_{ij}x_{lm}\, .
\]

As above, let $\scrM$ be the set of monomials $x_{ij}x_{lm}$ (with $(i,j) \leq_{\mathrm{lex}} (l,m)$). For $Q \in \Sym^2\bigl(\wedge^2 V_5^\vee\bigr)$, let $\scrM(Q) \subset \scrM$ be the subset of all monomials that occur in~$Q$ with nonzero coefficient. For $A = \diag(-a_1,\ldots,-a_5)$, the equation $A \circ Q = 0$ then becomes the system of linear equations
\begin{equation}
\label{eq:SystEqs}
a_i+a_j+a_l+a_m = 0 \qquad \text{for all $x_{ij}x_{lm} \in \scrM(Q)$.}
\end{equation}
Let $M_Q$ be the matrix of size $\# \scrM(Q) \times 5$ with integral coefficients such that \eqref{eq:SystEqs} is the equation
\begin{equation}
\label{eq:MQA=0}
M_Q \left(\begin{smallmatrix} a_1\\ a_2 \\a_3 \\ a_4\\ a_5 \end{smallmatrix}\right) = 0\, .
\end{equation}
Concretely: if $x_{12}x_{34} \in \scrM(Q)$ we get a row $(1,1,1,1,0)$ in~$M_Q$; if $x_{12}x_{13} \in \scrM(Q)$ we get a row $(2,1,1,0,0)$; if $x_{12}^2 \in \scrM(Q)$ we  get a row $(1,1,0,0,0)$; etc. In this last case, note that $2 \in k^*$ by assumption, so we can use the row $(1,1,0,0,0)$ rather than $(2,2,0,0,0)$. Observe that all possible rows are of one of the following three types:
\begin{itemize}
\item rows of type $(1,1,0,0,0)$, where the two coefficients~$1$ can be placed arbitrarily (ten rows of this type);
\item rows of type $(2,1,1,0,0)$, again with all permuted versions (thirty rows of this type);
\item rows of type $(1,1,1,1,0)$, again with all permuted versions (five rows of this type).
\end{itemize}
We denote by $E$ the matrix of size $45 \times 5$ that has all these rows. By construction, $M_Q$ is a sub-matrix of~$E$.

\begin{lemma}
\label{lem:rk/kge4}
Let $N$ be a sub-matrix of~$E$ of size $5\times 5$, and let $p$ be a prime number with $p\geq 5$. If $\rk_\QQ(N) = 5$ then $\rk_{\FF_p}(N) \geq 4$.
\end{lemma}

\begin{proof}
This is verified using Magma. The code we use is available on the second author's webpage, or upon request.
\end{proof}

\begin{proposition}
\label{prop:MQProps}
Let $Q \in \Sym^2\bigl(\wedge^2 V_5^\vee\bigr)$ be such that $Y = \Grass \cap Q$ is non-singular of dimension~$5$ and suppose there exists a nonzero diagonal matrix $A = \diag(-a_1,\ldots,-a_5)$ such that $A \circ Q = 0$. Then the following properties hold.
\begin{enumerate}
\item\label{rk/Q} The matrix $M_Q$ has rank~$5$ over~$\QQ$.

\item\label{5x5sub} The matrix $M_Q$ contains a sub-matrix~$N$ of size $5\times 5$ such that $N$ has rank~$5$ over~$\QQ$ and has rank~$4$ over~$\FF_p$.

\item\label{SingPts} If $M_Q$ does not have a row $(1,1,0,0,0)$ then $M_Q$ has a row of the form $(2,1,*,*,*)$ or $(1,2,*,*,*)$; likewise for all rows of type $(1,1,0,0,0)$ with coefficients~$1$ in different positions.

\item\label{Eigenval} Not all eigenvalues of~$A$ are distinct, i.e., there exist $1\leq i < j \leq 5$ with $a_i = a_j$.
\end{enumerate}
\end{proposition}

\begin{proof}
\ref{rk/Q} If $\rk_\QQ(M_Q) < 5$ then by looking at the Hermite normal form of~$M_Q$ we see that \eqref{eq:MQA=0} has a non-trivial solution for $A = \diag(-a_1,\ldots,-a_5)$ over~$k$ such that $(A,Q)$ lifts to a solution in characteristic~$0$; this is impossible by Proposition~\ref{prop:NoLiftAQ}. 

For~\ref{5x5sub}, choose any sub-matrix~$N$ of~$M_Q$ of size $5\times 5$ with $\rk_\QQ(N) = 5$. By Lemma~\ref{lem:rk/kge4} together with the fact that $N \cdot (a_1,\ldots,a_5)^{\mathsf{t}} = 0$ we have $\rk_{\FF_p}(N) = 4$.

For~\ref{SingPts}, suppose the row $(1,1,0,0,0)$ does not occur in~$M_Q$, which means that the monomial $x_{1,2}^2$ has zero coefficient in~$Q$. The point with Pl\"ucker coordinates $(1:0:\cdots:0)$ is therefore a point of~$Y$. The tangent space to~$\Grass$ at this point is given by the equations $x_{34} = x_{35} = x_{45} = 0$. As $Y$ is non-singular, it follows that at least one of the monomials $x_{12}x_{1i}$ or $x_{12}x_{2i}$ with $i \in \{3,4,5\}$ has to occur in~$Q$ with nonzero coefficient. This means that $M_Q$ has a row of the form $(2,1,*,*,*)$ or $(1,2,*,*,*)$. The same argument applies to rows of type $(1,1,0,0,0)$ with coefficients~$1$ in different positions.

For \ref{Eigenval}, let $P_{ij}$ ($1\leq i<j \leq 5$) be the ten points of~$\Grass$ for which only one Pl\"ucker coordinate is non-zero. (So $P_{12} = (1:0:\ldots:0)$, etc.) If $A = \diag(-a_1,\ldots,-a_5)$ and all~$a_i$ are distinct, the zero scheme of the vector field $\xi = \xi_A$ on~$\Grass$ is the union of the points~$P_{ij}$. (Note that the topological Euler characteristic of~$\Grass$ is~$10$.) Therefore, the zero scheme~$Z(\xi|_Y)$ of~$\xi|_Y$ is the sum of the points~$P_{ij}$ that lie on~$Y$. (Note that $Z(\xi|_Y)$ has length~$1$ at each of these points because $Z(\xi|_Y)$ is a subscheme of~$Z(\xi)$.) On the other hand, our assumption that $A \circ Q = 0$ gives that $\xi|_Y$ is a vector field on~$Y$, and of course the zero scheme of~$\xi|_Y$ as a section of~$\scrT_Y$ is the same as the zero scheme of~$\xi|_Y$ as a section of~$\scrT_{\Grass}|_Y$. Using the results in \cite{Fulton}, Section~14.1, we find that $0 \leq \deg\bigl(c_5(Y)\bigr) \leq 10$. Because the topological Euler characteristic of~$Y$ is~$-12$ this gives a contradiction. (Cf.\ \cite{Fulton}, Example~18.3.7(c).)
\end{proof}

\subsection{}
The strategy of proof is now as follows. Using Magma, we inspect all $5\times 5$ sub-matrices~$N$ of~$E$ with $\rk_\QQ(N) = 5$. For each such matrix~$N$, we store all prime numbers $p\geq 5$ such that $\rk_{\FF_p}(N) < 5$; these are found as the prime numbers that divide the leading coefficients in the Hermite normal form of~$N$. For $(N,p)$ obtained in this way there is, up to scalars, a unique $A = \diag(-a_1,\ldots,-a_5)$ in~$\frgl_{5,k}$ for which $N \cdot (a_1,\ldots,a_5)^{\mathsf{t}} = 0$. (Use Lemma~\ref{lem:rk/kge4}.) Once~$A$ has been determined, we can form the sub-matrix $E_A$ of~$E$ consisting of all rows $X = (x_1,\ldots,x_5)$ of~$E$ for which $X \cdot (a_1,\ldots,a_5)^{\mathsf{t}} = 0$. We store all data $(N,p,A,E_A)$ that are found in this process.

Each row of the matrix~$E$ represents one of the monomials $x_{ij}x_{lm}$. If $(N,p,A,E_A)$ is as just described, let $\scrM_A \subset \scrM$ be the set of monomials that correspond to the rows in~$E_A$, so that $E_A$ is a matrix of size $\# \scrM_A \times 5$. Consider all quadratic homogeneous polynomials~$Q$ that can be written as a linear combination $Q = \sum_{X \in \scrM_A}\; c_X \cdot X$. These are parametrized by an affine space~$\AA$ of dimension~$\# \scrM_A$. In this way, each four-tuple $(N,p,A,E_A)$ gives rise to a family of quadrics $\{Q_t\}_{t\in \AA}$ such that the vector field~$\xi_A$ on~$\Grass$ restricts to a vector field on $Y_t = \Grass \cap Q_t$, for every~$t$.

Our aim is to prove that for every non-singular $5$-dimensional GM variety $Y = \Grass \cap Q$ over~$k$ as in Section~\ref{subsec:XYGrSetup} (still with $\mathrm{char}(k) = p \geq 5$) we have $H^0(Y,\scrT_Y) = 0$. Suppose this fails, i.e., there exists such a $Y = \Grass \cap Q$ with $H^0(Y,\scrT_Y) \neq 0$. We may assume $Q \in \Ker(\mu)$. By Proposition~\ref{prop:AQ=0}, Lemma~\ref{lem:DiagA} and Proposition~\ref{prop:MQProps}, there then exists an entry $(N,p,A,E_A)$ in our list such that the matrix~$M_Q$ corresponding to~$Q$ is a sub-matrix of~$E_A$. In particular, $E_A$ satisfies condition~\ref{SingPts} in Proposition~\ref{prop:MQProps} (with $M_Q$ replaced by~$E_A$), and by part~\ref{Eigenval} of that proposition, not all~$a_i$ are distinct. In the above process we may therefore discard all $(N,p,A,E_A)$ for which these conditions are not satisfied. For each of the remaining cases, let $K_A = \FF_p(t_X)_{X \in \scrM_A}$ be the function field of~$\AA$, and let $\scrQ = \sum_{X \in \scrM_A}\; t_X \cdot X$ be the generic member of the family of polynomials parametrized by~$\AA$. We are done if we can verify that in all cases $\Grass \cap \scrQ$ over~$K_A$ is singular. Indeed, our~$Y$ is isomorphic to a fibre in one of these families $\{\Grass \cap Q_t\}_{t \in \AA}$, and if $Y$ is non-singular, so is the generic fibre of that family.

\subsection{}
We carry out a calculation in Magma according to the strategy just described. (For the code, see the second author's webpage.) For $p \geq 11$, there are no solutions $(N,p,A,E_A)$ such that conditions~\ref{SingPts} and~\ref{Eigenval} in Proposition~\ref{prop:MQProps} are satisfied. This therefore already settles Proposition~\ref{prop:CohomTY} in case $p\geq 11$.

For $p=5$ there are, up to a renumbering or rescaling of the coordinates, four families of quadrics that require further checking, and for $p=7$ there is one family that requires inspection. These families are given below. In each case, we shall exhibit a singular point of the generic fibre $\scrQ \cap \Grass$ over an algebraic extension of the function field~$K_A$. (Notation as above.) We use Pl\"ucker coordinates in lexicographical order, so $(x_{12}:x_{13}:x_{14}:x_{15}:x_{23}:x_{24}:x_{25}:x_{34}:x_{35}:x_{45})$.

The families that are found for $p=5$ are the following ones.
\begin{enumerate}
\item\label{fam1} The family of quadrics given by
\begin{multline*}
\qquad t_1\, x_{14}x_{24} + t_2\, x_{23} x_{34} + t_3\, x_{25}^2 + t_4\, x_{15} x_{35} + t_5\, x_{14} x_{45} + t_6\, x_{34} x_{35}\\
+ t_7\, x_{12} x_{23} + t_8\, x_{12} x_{35} + t_9\, x_{13} x_{25} + t_{10}\, x_{15} x_{23} + t_{11}\, x_{13}^2 = 0 \, ,\qquad
\end{multline*}
with $A = \diag(-2,0,-3,-4,0)$. (Note that in this family we may get rid of one parameter by considering only quadric polynomials in $\Ker(\mu) \subset \Sym^2(\wedge^2 V_5^\vee)$; see Section~\ref{subsec:Kermu}.) A singular point of $\Grass \cap \scrQ$ is the point $(0:0:0:0:0:t_5:0:0:0:-t_1)$. (Non-zero Pl\"ucker coordinates in positions $24$ and~$45$. Note that $\scrQ$ is singular in this point.)

\item\label{fam3} The family of quadrics given by
\begin{multline*}
\qquad t_1\, x_{14}^2 + t_2\, x_{15}^2 + t_3\, x_{24} x_{45} + t_4\, x_{25} x_{45} + t_5\, x_{14} x_{15} + t_6\, x_{12} x_{23}\\
+ t_7\, x_{34} x_{45} + t_8\, x_{35} x_{45} + t_9 x_{13} x_{23} = 0\, ,
\qquad
\end{multline*}
with $A = \diag(-1,-3,-3,-4,-4)$. A singular point of $\Grass \cap \scrQ$ is the point $(t_9:-t_6:0:0:0:0:0:0:0:0)$. (Non-zero Pl\"ucker coordinates in positions $12$ and~$13$, and again $\scrQ$ is singular in this point.)

\item\label{fam4} The family of quadrics given by
\begin{multline*}
\qquad t_1\, x_{23} x_{24} + t_2\, x_{23} x_{25} + t_3\, x_{24} x_{45} + t_4\, x_{25} x_{45} + t_5\, x_{13} x_{23} + t_6\, x_{13} x_{45}\\
+t_7\, x_{14} x_{35} + t_8\,  x_{15} x_{34} + t_9\, x_{12}^2 + t_{10}\, x_{14} x_{34} + t_{11}\, x_{15} x_{35} = 0\, ,
\qquad
\end{multline*}
with $A = \diag(-2,-3,0,-4,-4)$. (In this family we could again get rid of one parameter.) Choose an element $\lambda$ in an algebraic extension of the function field $K_A = k(t_1,\ldots,t_{11})$ such that $t_{10}\, \lambda^2 + (t_7 + t_8)\, \lambda + t_{11} = 0$; then $(0:0:\lambda:1:0:0:0:0:0:0)$ is a singular point of $\Grass \cap \scrQ$. To see this, note that tangent equation at this point that comes from the first Pl\"ucker equation is $-\lambda X_{35} + X_{34} = 0$ (using $X_{ij}$ as coordinates on the tangent bundle), whereas the quadric~$\scrQ$ gives the tangent equation $(t_7\lambda + t_{11}) X_{35} + (t_{10}\lambda + t_8) X_{34} = 0$. By our choice of~$\lambda$ these equations are linearly dependent.

\item\label{fam5} The family of quadrics given by
\begin{multline*}
\qquad t_1\, x_{14} x_{45} + t_2\, x_{24}^2 + t_3\, x_{23} x_{25} + t_4\, x_{23} x_{34} + t_5\, x_{14} x_{15} + t_6\, x_{12} x_{24}\\
+ t_7\, x_{13} x_{23} + t_8\, x_{35} x_{45} + t_9\, x_{12}^2 + t_{10}\, x_{15} x_{35} = 0\, ,
\qquad
\end{multline*}
with $A = \diag(-2,-3,0,-2,-4)$. A singular point of $\Grass \cap \scrQ$ is the point $(0:t_4:0:0:0:0:0:-t_7:0:0)$. (Non-zero Pl\"ucker coordinates in positions $13$ and~$34$, and again $\scrQ$ is singular in this point.)

\end{enumerate}

\medskip

For $p=7$, there is, again up to a renumbering or rescaling of the coordinates, one family of quadrics that could give rise to a counterexample:
\begin{enumerate}[resume*]
\item\label{fam7} The family of quadrics given by
\begin{multline*}
t_1\, x_{24} x_{34} + t_2\, x_{12} x_{15} + t_3\, x_{13} x_{35} + t_4\, x_{25}^2 \\+ t_5\, x_{25} x_{45} + t_6\, x_{14} x_{15} + t_7\, x_{45}^2 + t_8\, x_{23} x_{24} = 0\, ,
\end{multline*}
with $A = \diag(0,-1,-4,-1,-6)$. A singular point of $\Grass \cap \scrQ$ is the point $(t_6:0:-t_2:0:0:0:0:0:0:0)$. (Non-zero Pl\"ucker coordinates in positions $12$ and~$14$, and again $\scrQ$ is singular in this point.)
\end{enumerate}
\medskip

As explained, the fact that in each of these cases $\Grass \cap \scrQ$ is singular implies that no member $\Grass \cap Q_t$ in the family is non-singular, so that the families we have found do not produce any counterexamples to Proposition~\ref{prop:CohomTY}. This therefore completes the proof of that proposition. \qed

\section{Families of motives and Shimura varieties (characteristic $0$)}
\label{sec:FamMotivesChar0}

Our proof of the Tate conjecture for Gushel--Mukai varieties closely follows Madapusi Pera's argument for K3 surfaces. As part of this, we want to consider families of motives. We interpret these as systems of cohomological realizations, similar to the approach in Jannsen's book~\cite{Jannsen} or \S1 of Deligne's paper~\cite{Deligne-pi1}. We will not in any way attempt to develop a full theory. In fact, we will just work with systems of realizations and we will not attempt to define which of these systems deserve to be called ``motives''.

As base space for such systems of realizations we allow algebraic stacks~$S$ that are globally the quotient of a smooth scheme by a finite group, say $S = [T/\Gamma]$. The reader who prefers to work only with schemes can, instead of working over~$S$, work with $\Gamma$-equivariant objects on~$T$.

\subsection{}
\label{subsec:SystReal}
Let $K$ be a number field, and, as above, let $S$ be the quotient stack of a smooth scheme of finite type over~$K$ by a finite group. By a system of realizations on~$S/K$ we then mean a package of data
\[
\motH = \Bigl(\motH_\tau,\motH_{\dR} = (\motH_{\dR},\nabla,\Fil^\bullet),\motH_{\et},i_{\dR,\tau},i_{\et,\tau}\Bigr)_{\tau \colon K \to \CC}
\]
of the following kind. The first three components are:
\begin{itemize}
\item local systems $\motH_\tau$ of $\QQ$-vector spaces on $(S\otimes_{K,\tau} \CC)_{\an}$, one for each complex embedding~$\tau$;
\item a flat vector bundle $(\motH_{\dR},\nabla)$ on~$S$, equipped with a finite exhaustive filtration~$\Fil^\bullet$ by subbundles;
\item a smooth $\AA_{\fin}$-sheaf $\motH_{\et}$ on $S_{\proet}$.
\end{itemize}
(In what follows we often omit $\nabla$ and $\Fil^\bullet$ from the notation.) In addition we have comparison isomorphisms. With the convention that for $F$ a sheaf on~$S$ we denote its pullback to~$S_\tau = S \otimes_{K,\tau} \CC$ by~$\tau^*F$, these are:
\begin{itemize}
\item for each embedding~$\tau$, an isomorphism of flat vector bundles 
\[
i_{\dR,\tau} \colon \bigl(\motH_\tau \otimes \scrO_{S_\tau},\id \otimes d\bigr) \isomarrow \tau^*(\motH_{\dR},\nabla)
\] 
on~$S_\tau$;
\item for each~$\tau$, an isomorphism $i_{\et,\tau} \colon \motH_\tau \otimes \AA_{\fin} \isomarrow \tau^* \motH_{\et}$ on~$S_\tau$.
\end{itemize}
To interpret these comparison isomorphisms correctly, some changes of topology are required; for instance, $\tau^*(\motH_{\dR},\nabla)$ is viewed as a flat vector bundle on~$S_\tau$ for the analytic topology. A way to give meaning to~$i_{\et,\tau}$ is to choose a base point $s \in S(\CC)$ and regard source and target of~$i_{\et,\tau}$ as $\AA_{\fin}$-representations of the topological fundamental group $\pi_1(S_{\tau},s)$.

Finally, we make the requirement that, for each~$\tau$, the local system~$\motH_\tau$ equipped with the filtration $\Fil^\bullet$ on~$\motH_\tau \otimes \scrO_{S_\tau}$ (via~$i_{\dR,\tau}$) is a pure Variation of Hodge Structure (VHS) with $\QQ$-coefficients over~$S_\tau$. (Here ``pure'' does not mean the VHS is pure of some weight, but that it is a direct sum of such, i.e., the weight filtration is split.)

In what follows, if $\motH$ is a system of realizations as above and $p$ is a prime number, $\motH_p$ denotes the $p$-adic component of~$\motH_\et$, and we may write $\motH_{\et} = \motH_p \times \motH^{(p)}$, where the prime-to-$p$ part $\motH^{(p)}$ is a sheaf of $\AA_{\fin}^p$-modules on~$S_{\proet}$.

With the obvious definitions of morphisms and of the tensor product and duals, the category $\Real(S;\QQ)$ of systems of realizations is a $\QQ$-linear Tannakian category. We have Tate objects~$\unitmot(n)$ in~$\Real(S;\QQ)$, and as usual we denote $\motH \otimes \unitmot(n)$ by~$\motH(n)$.

\begin{example}
\label{exa:HiMotive}
The basic example is that for a smooth projective morphism $f\colon X \to S$, we have an associated system of realizations $\motH^\bullet(X/S) = \oplus\; \motH^i(X/S)$. For $\tau$ a complex embedding, $\motH^i(X/S)_\tau$ is the variation of Hodge structure with underlying local system $R^if_{\tau,*}\QQ_{X\otimes_{K,\tau} \CC}$. The de Rham component is the relative de Rham bundle $R^if_* \Omega^\bullet_{X/S}$ with its Hodge filtration and Gauss--Manin connection. The \'etale component of $\motH^i(X/S)$ is the pro-\'etale sheaf $R^if_*\AA_{\fin}$ on~$S$.
\end{example}

\begin{remarks}
\begin{enumerate}
\item A further ingredient that may be included as part of the data is the ``Frobenius at infinity'', as in \cite{Deligne-pi1}, Section~1.4. As we shall make no use of this, we omit it.

\item If $S = \Spec(K)$, our notion agrees with a pure mixed realization (i.e., a pure object in the category of mixed realizations) in the sense of~\cite{Jannsen}. The main difference in presentation is that Jannsen fixes an algebraic closure $K \subset \Kbar$ and interprets~$\motH_{\et}$ as an $\AA_{\fin}$-module with continuous action of $\Gal(\Kbar/K)$ (or rather, he does this for all $\ell$-adic components separately). In Jannsen's approach one then needs to fix an algebraic closure~$\Kbar$ and one has a comparison isomorphism~$i_{\et,\bar\tau}$ for each $\bar\tau \colon \Kbar \to \CC$.

\item For us, the simplest way to define an absolute Hodge class in~$\motH$ is as a morphism $\unitmot \to \motH$. This is of course equivalent to other definitions found in the literature. Note that in our approach the \'etale component of~$\motH$ is a sheaf on~$S_{\proet}$, not just a geometric local system. Therefore, if we want to interpret an absolute Hodge class as a collection of elements in vector spaces, we should first rewrite~$\motH_{\et}$ as an $\AA_{\fin}$-module with Galois action.
\end{enumerate}
\end{remarks}

\subsection{}
We will also need to consider integral structures, which we approach as in \cite{Deligne-pi1}, 1.23--1.26; see also \cite{MP-K3},~2.4. There are two cases of interest for us: $\ZZ$-coefficients and $\ZZ_{(p)}$-coefficients. By a system of realizations with $\ZZ$-coefficients over~$S$ we mean data $\motH$ as above together with a subsheaf of $\Zhat$-modules
\[
\motH_{\Zhat,\et} \subset \motH_{\et}
\]
such that $\motH_{\Zhat,\et} \otimes_{\Zhat} \AA_{\fin} \isomarrow \motH_{\et}$. For each complex embedding~$\tau$ this then defines a $\ZZ$-submodule $\motH_{\ZZ,\tau} \subset \motH_\tau$ by taking
\[
\motH_{\ZZ,\tau} = \bigl\{x \in \motH_\tau \bigm| i_{\et,\tau}(x\otimes 1) \in \tau^*\motH_{\Zhat,\et} \bigr\}\, .
\]
We let $\Real(S;\ZZ)$ be the category of systems of realizations with $\ZZ$-coefficients.

Similarly we can define the category $\Real(S;\ZZ_{(p)})$ of systems of realizations over~$S$ with $\ZZ_{(p)}$-coefficients; in this case the extra datum is a $\ZZ_p$-subsheaf
\[
\motH_{\ZZ_p} \subset \motH_p
\]
such that $\motH_{\ZZ_p} \otimes \QQ_p \isomarrow \motH_p$. With a definition analogous to the one above, the choice of such a $\ZZ_p$-subsheaf gives rise to $\ZZ_{(p)}$-submodules $\motH_{\ZZ_{(p)},\tau} \subset \motH_\tau$, which are part of a $\ZZ_{(p)}$-VHS over~$S_\tau$.

\begin{example}
In the situation considered in Example~\ref{exa:HiMotive}, we can refine $\motH^i(X/S)$ to a system of realizations $\motH^i(X/S;\ZZ)$ with integral coefficients, taking as \'etale component the pro-\'etale sheaf $R^if_*\Zhat$ on~$S$, modulo torsion.
\end{example}

\subsection{}
\label{subsec:Discr}
In the discussion that follows, we will consider discriminant groups of integral quadratic forms. Let $L$ be a lattice, i.e., a finite free $\ZZ$-module equipped with a nondegenerate symmetric bilinear form $b\colon L \times L \to \ZZ$. We define the discriminant group of~$L$ to be
\[
\Discr(L) = L^\vee/L\, ,
\]
where $L^\vee = \bigl\{x \in L_\QQ \bigm| b(x,y) \in \ZZ\ \text{for all $y \in L$}\bigr\}$. This group~$\Discr(L)$ comes equipped with a nondegenerate symmetric bilinear pairing with values in~$\QQ/\ZZ$. We refer to~\cite{Nikulin} for basic results on this. We can extend these notions, in an obvious manner, to local systems of $\ZZ$-modules.

If the form~$b$ is nondegenerate (as will be the case in what follows), the set~$\Phi\Discr(L)$ is a finite group, and we define $\discr(L)$ to be its order.

\subsection{}
We next introduce some Shimura varieties. Let us first make some remarks on the notation we shall use. We fix a prime number~$p$.

Suppose $(G_\QQ,\scrD)$ is a Shimura datum with reflex field~$E$ such that $G_\QQ$ is the generic fibre of a reductive group~$G$ over~$\ZZ_{(p)}$. Let $K \subset G(\AA_{\fin})$ be a compact open subgroup, and let $S_K = \Sh_K(G,\scrD)$ be the corresponding Shimura variety over~$E$. If $K$ is not neat then we will interpret~$S_K$ as an algebraic stack: we can find a normal subgroup $K^\prime \subset K$ of finite index which is neat; then $S_{K^\prime}$ is a scheme over~$E$ (constructed in the usual way), and we define $S_K$ to be the quotient stack $[S_{K^\prime}/(K/K^\prime)]$.

The Shimura data that are relevant for us all have reflex field~$\QQ$, so we now assume $E= \QQ$. Let $K_p = G(\ZZ_p)$, which by our assumption on~$G$ is a hyperspecial subgroup of~$G(\QQ_p)$. If $K = K^pK_p$ for some compact open subgroup $K^p \subset G(\AA_{\fin}^p)$, we denote by $\scrS_K$ the integral canonical model of~$S_K$ over~$\ZZ_{(p)}$. (See~\cite{Kisin-IntCanMod}.) Again this is in general to be interpreted as a stack.

\subsection{}
\label{subsec:ShVars}
Let $(L,b)$ be a lattice of rank~$r\geq 3$ such that the corresponding quadratic form $Q(x) = b(x,x)$ has signature $(r-2,2)$. We have reductive group schemes
\[
G = \SO(L,Q)\qquad\text{and}\qquad
\tilde{G} = \GSpin(L,Q)
\]
over $\Spec\bigl(\ZZ[(2\cdot \discr(L))^{-1}]\bigr)$. As in \cite{Deligne-WeilK3}, Section~3, we have a homomorphism
\[
\tilde{G} \xrightarrow{~q~} G\, ,
\]
whose kernel is~$\GG_{\mult}$. 

We denote by $C(L) = C^+(L) \oplus C^-(L)$ the Clifford algebra. We write $H$ for $C(L)$ viewed as a representation of~$\tilde{G}$ via left multiplication in the Clifford algebra. There exists a symplectic form~$\psi$ on~$H$ such that the corresponding homomorphism $i\colon \tilde{G} \to \GL(H)$ factors through the group $\GSp(H,\psi)$ of symplectic similitudes. (Cf.\ \cite{MP-IntCanMod}, Lemma~1.7.) In what follows we abbreviate $\GSp(H,\psi)$ to~$\GSp$.

Let $\widetilde\scrD$ be the space of oriented negative definite $2$-planes in~$L_\RR$. As explained in \cite{MP-IntCanMod}, Section~3.1, $\widetilde\scrD$ can be viewed as a space of homomorphisms $\SS \to \tilde{G}_\RR$ (where, as usual, $\SS$ is $\CC^*$ viewed as a real algebraic group). Write $\scrD$ for the image of~$\widetilde\scrD$ in $\Hom(\SS,G_\RR)$, and let $\scrH \subset \Hom(\SS,\GSp_\RR)$ be the $\GSp(\RR)$-conjugacy class that contains $i(\widetilde{\scrD})$. We then have a diagram of Shimura data
\[
\begin{tikzcd}
(\tilde{G}_\QQ,\widetilde\scrD) \ar[d,"q"] \ar[r,"i"] & (\GSp_\QQ,\scrH) \\
(G_\QQ,\scrD)
\end{tikzcd}
\]
each of which has reflex field~$\QQ$.

{}From now on we assume that $p>2$ and that $p$ does not divide $\discr(L) = \#\Discr(L)$, so that the bilinear form~$b$ on~$L$ induces a perfect pairing on $L_{(p)} = L \otimes \ZZ_{(p)}$. (This is the only case we will need.) Define
\begin{equation}
\label{eq:KLDef}
K(L) = \Biggl\{ g \in G(\AA_{\fin}) \Biggm| 
\vcenter{
\setbox0=\hbox{$g$ acts trivially on $\Discr(L \otimes \Zhat)$}
\hbox to\wd0{\hfill $g(L\otimes \Zhat) = L\otimes \Zhat$ and\hfill}
\copy0
}\Biggr\}\, ,
\end{equation}
and
\[
\tilde{K}(L) = \tilde{G}(\AA_{\fin}) \cap \bigl(C(L) \otimes \Zhat\bigr)^*\, .
\]
It follows from \cite{MP-IntCanMod}, Lemma~2.6, that the homomorphism~$q$ restricts to a surjective homomorphism $\tilde{K}(L) \to K(L)$. We define
\[
\scrS = \text{the integral canonical model (over~$\ZZ_{(p)}$) of $\Sh_{K(L)}(G_\QQ,\scrD)$,}
\]
and
\[
\tilde{\scrS} = \text{the integral canonical model (over~$\ZZ_{(p)}$) of $\Sh_{\tilde{K}(L)}(\tilde{G}_\QQ,\widetilde\scrD)$,}
\]
Note that in this section, only the generic fibres of these Shimura stacks play a role; in the next two sections the $p$-adic aspects will become important.

For the third Shimura variety in the above picture, let $\scrK(L) \subset \GSp(\AA_{\fin})$ denote the stabilizer of the lattice $C(L) \otimes \Zhat \subset C(L) \otimes \AA_{\fin}$, and let $\scrK_p = \GSp(\ZZ_p)$ be its $p$-adic component. We define~$\scrA_L$ to be the integral canonical model (again over~$\ZZ_{(p)}$) of $\mathrm{Sh}_{\scrK(L)}(\GSp,\scrH)$.

The above diagram of Shimura data gives rise to a diagram
\begin{equation}
\label{eq:ShimVars/Q}
\begin{tikzcd}
\tilde{\scrS} \ar[d,"\pi"] \ar[r,"i"] & \scrA_L \\
\scrS
\end{tikzcd}
\end{equation}
of stacks over~$\ZZ_{(p)}$.

\begin{remark}
\label{rem:mu2gerbe}
The morphism $\pi \colon \tilde{\scrS} \to \scrS$ is a $\mu_2$-gerbe. To see this, choose an integer~$N$ that is not divisible by $2p\cdot \discr(L)$, and consider the subgroups $K_N(L) \subset K(L)$ and $\tilde{K}_N(L) \subset \tilde{K}(L)$ of elements that are the identity modulo~$N$ in $\GL(L \otimes \Zhat)$, respectively $(C(L) \otimes \Zhat)^*$. Let $\scrS_N \to \scrS$ and $\tilde{\scrS}_N \to \tilde{\scrS}$ be the corresponding covers, and let $\Delta = \Zhat^*/\bigl\{z\in \Zhat^* \bigm| z \equiv 1 \bmod N\bigr\} \cong (\ZZ/N\ZZ)^*$. (The~$\Zhat^*$ in this definition is to be viewed as the group of $\Zhat$-points of the centre of~$\tilde{G}$.) Then $\scrS_N \times_{\scrS} \tilde{\scrS} \cong [\tilde{\scrS}_N/\Delta]$, whereas $\scrS_N$ is the quotient of~$\tilde{\scrS}_N$ by the action of~$\Delta/\{\pm 1\}$. 
\end{remark}

\subsection{}
Via the interpretation of~$\scrA_L$ as a moduli stack for abelian varieties, the morphism $\tilde{\scrS} \to \scrA_L$ corresponds to an abelian scheme $A \to \tilde{\scrS}$, of dimension~$2^{r-1}$. Let $\motH = \motH^1(A/\tilde{\scrS}_{\QQ};\ZZ)$ be the family of realizations with $\ZZ$-coefficients over the characteristic~$0$ fibre of~$\tilde{\scrS}$ that is given by the cohomology in degree~$1$ of the fibres of~$A$.

In the proof of the Tate conjecture an important role is played by a sub-object $\tilde{\motL} \subset \ul{\End}(\motH)$ in $\Real(\tilde{\scrS}_\QQ;\ZZ)$ that is constructed by Madapusi Pera. We will briefly discuss the construction below. In fact, $\tilde{\motL} \subset \ul{\End}(\motH)$ descends to a pair $\motL \subset \motE$ in $\Real(\scrS_\QQ;\ZZ)$ and in the present section the system of realizations~$\motL$ plays a key role. Note that our notation is slightly different from the notation used in~\cite{MP-K3}: Madapusi Pera uses~$\motL$ both for~$\tilde{\motL}$ (which lives over~$\tilde{\scrS}_\QQ$) and for~$\motL$ (which lives over~$\scrS_\QQ$). The discussion is complicated further by the fact that $\tilde{\motL}$ is constructed in \cite{MP-IntCanMod} using $\ZZ_{(p)}$-coefficients, whereas it is used in~\cite{MP-K3} with $\ZZ$-coefficients. As furthermore the two papers \cite{MP-IntCanMod} and~\cite{MP-K3} do not use the same notation (e.g., the Shimura variety~$\Sh_K$ of~\cite{MP-IntCanMod} is the $\widetilde{\Sh}(L)$ of~\cite{MP-K3}), we hope our summary can help the reader to understand the construction.

A reason why we want to define~$\motL$ as an object with $\ZZ$-coefficients is that this makes results such as Proposition~\ref{prop:MPK3-4.3} (which then leads to Proposition~\ref{prop:descendQ}) easier to state. It should be noted that once Proposition~\ref{prop:descendQ} is proved (i.e., in the next sections), only the $\ZZ_{(p)}$-structure will be needed.

\subsection{}
We now give a brief explanation of how the system of realizations~$\motL$ is constructed. The easiest, and conceptually cleanest, way to understand the construction is to use that for a Shimura datum $(\scrG,X)$ of abelian type and a compact open subgroup $K \subset \scrG(\AA_{\fin})$, there exists a natural $\otimes$-functor $r_{(\scrG,X,K)}\colon \Rep(\scrG) \to \Real(\Sh_K(\scrG,X);\QQ)$, where $\Rep(\scrG)$ denotes the category of finite dimensional representations of~$\scrG$ (over~$\QQ$). Moreover, if $f \colon (\scrG_1,X_1) \to (\scrG_2,X_2)$ is a morphism of Shimura data of abelian type and $K_i \subset \scrG_i(\AA_{\fin})$ ($i=1,2$) are compact open subgroups with $f(K_1) \subset K_2$, we have an inclusion of reflex fields $E_2 = E(\scrG_2,X_2) \subset E_1 = E(\scrG_1,X_1)$ and a commutative diagram
\[
\begin{tikzcd}[column sep=huge]
\Rep(\scrG_2) \ar[d,"\rho \mapsto \rho \circ f"'] \ar[r,"r_{(\scrG_2,X_2,K_2)}"] & \Real(\Sh_{K_2}(\scrG_2,X_2)_{E_1};\QQ) \ar[d,"\Sh(f)^*"]\\
\Rep(\scrG_1) \ar[r,"r_{(\scrG_1,X_1,K_1)}"] & \Real(\Sh_{K_1}(\scrG_1,X_1);\QQ)
\end{tikzcd}
\]
where of course $\Sh(f) \colon \Sh_{K_1}(\scrG_1,X_1) \to \Sh_{K_2}(\scrG_2,X_2)_{E_1}$ is the morphism induced by~$f$. Using this mechanism, and ignoring integral coefficients for the moment, the system of realizations $\motH = \motH^1(A/\tilde{\scrS}_{\QQ};\QQ)$ over~$\tilde{\scrS}_\QQ$ comes from the natural representation $\tilde{G}_\QQ \to \GL(H_\QQ)$ with $H = C(L)$, on which $\tilde{G}$ acts through left multiplications. The fact that $\ul{\End}(\motH)$ descends to a system of realizations~$\motE$ over~$\scrS_\QQ$ then corresponds to the fact that the induced representation of~$\tilde{G}$ on $\ul{\End}(H_\QQ)$ factors through $q \colon \tilde{G}_\QQ \to G_\QQ$. We have an inclusion $L \hookrightarrow \ul{\End}(H)$, letting $L$ (viewed as a subspace of~$C(L)$) act on~$H$ by left multiplication. Then $L_\QQ \subset \ul{\End}(H_\QQ)$ is stable under the action of~$G_\QQ$, and the associated system of realizations is the system~$\motL$ (with $\QQ$-coefficients).

Unfortunately, the existence of the functors~$r_{(\scrG,X,K)}$ does not seem to be documented in the literature in the generality we need. For Shimura varieties of Hodge type, the existence readily follows from \cite{Kisin-IntCanMod}, Lemma~2.2.1. To extend the construction to Shimura varieties of abelian type requires some work, and since we only need it in one concrete example, this does not seem the right place to develop the theory in full. We therefore take a more direct approach, which in fact is closer to~\cite{MP-IntCanMod}.

We can use the mechanism just described (for the Shimura variety~$\tilde{\scrS}_\QQ$, which is of Hodge type) to construct a system of realizations $\tilde{\motL} \subset \ul{\End}(\motH)$ and a nondegenerate symmetric pairing $b_{\tilde{\motL}} \colon \tilde{\motL} \otimes \tilde{\motL} \to \unitmot$. (We will deal with the integral structure later.) Here is a more explicit version of this, following~\cite{MP-IntCanMod}. We write $\infty \colon \QQ \to \CC$ for the complex embedding of~$\QQ$. Considering the Betti realization~$\motH_\infty$ of $\motH = \motH^1(A/\tilde{\scrS}_{\QQ};\ZZ)$, the first step is the construction of an explicit projector~$\bm{\pi}_\infty$ that cuts out $\tilde{\motL}_\infty \subset \ul{\End}(\motH_\infty)$ over~$\tilde{\scrS}_{\CC}$. This is explained in \cite{MP-IntCanMod}, Section~1.3. (In fact, $\bm{\pi}_\infty$ is the unique projector with image~$L$ which is zero on~$L^\perp$.) Via the Betti--de Rham and Betti--\'etale comparison isomorphisms, this then gives de Rham and \'etale projectors $\bm{\pi}_{\dR,\CC}$ and~$\bm{\pi}_{\et,\CC}$ over~$\tilde{\scrS}_\CC$. It is shown in \cite{MP-IntCanMod}, Proposition~3.11, that these descend to projectors defined over~$\tilde{\scrS}$. This gives us the de Rham realization $\tilde{\motL}_{\dR}$ over~$\tilde{\scrS}$, with Hodge filtration induced from the Hodge filtration on~$\ul{\End}(\motH_{\dR})$, and the \'etale realization $\tilde{\motL}_{\dR}$, as a pro-\'etale sheaf of $\AA_{\fin}$-modules on~$\tilde{\scrS}$. By construction, these projectors (Betti, de Rham, \'etale) are compatible under the comparison isomorphisms for the system of realizations~$\ul{\End}(\motH)$, so we obtain a system of realizations $\tilde{\motL}$ over~$\tilde{\scrS}$.

The pairing $b_{\tilde{\motL}} \colon \tilde{\motL} \otimes \tilde{\motL} \to \unitmot$ can be defined as the restriction of the pairing $\ul{\End}(\motH) \otimes \ul{\End}(\motH) \to \unitmot$ given by $f_1 \otimes f_2 \mapsto 2^{1-r}\cdot \trace(f_1 \circ f_2)$, where we recall that $r = \rk(L)$. (The trace maps involved here of course have to be interpreted in each realization separately but we have the obvious compatibilities that make this definition meaningful.)

As $\ZZ$-structure on~$\tilde{\motL}$ we take the structure induced by the $\ZZ$-structure on~$\ul{\End}(\motH)$. This means that
\[
\tilde{\motL}_{\ZZ,\infty} = \tilde{\motL}_\infty \cap \ul{\End}\bigl(\motH_{\ZZ,\infty}\bigr)
\qquad\text{and}\qquad
\tilde{\motL}_{\Zhat,\et} = \tilde{\motL}_{\et} \cap \ul{\End}\bigl(\motH_{\Zhat,\et}\bigr)\, ,
\]
where the intersections are taken inside $\ul{\End}\bigl(\motH_\infty\bigr)$, respectively $\ul{\End}\bigl(\motH_{\et}\bigr)$. Note that the restriction of~$b_{\tilde{\motL}}$ to $\tilde{\motL}_{\ZZ,\infty}$ does not give a perfect pairing.

Finally, to see that $\tilde{\motL} \subset \ul{\End}(\motH)$ descend to systems of realizations $\motL \subset \motE$ over~$\scrS_\QQ$, we can use that $\tilde{\scrS} \to \scrS$ is a $\mu_2$-gerbe. With notation as in Remark~\ref{rem:mu2gerbe}, we can view~$\motH$ and~$\tilde{L}$ as $\tilde{G}(\ZZ/N \ZZ)$-equivariant objects over~$\tilde{\scrS}_{N,\QQ}$. In particular, we already have an action of the group~$\Delta$. What we need in order to descend $\tilde{\motL} \subset \ul{\End}(\motH)$ to $G(\ZZ/N \ZZ)$-equivariant objects over~$\scrS_{N,\QQ}$ (i.e., systems of realizations over~$\scrS_\QQ$), is an action of~$\Delta/\{\pm 1\}$. Hence it suffices to show that $-1 \in \Delta$ acts trivially, which can be checked on the Betti realizations. Now use that $-1 \in \Delta$ acts as $-\id$ on~$\motH$ and therefore acts trivially on~$\ul{\End}(\motH)$.

\subsection{}
\label{subsec:ConnectToGM}
So far in this section we have only discussed certain Shimura varieties and families of motives (or at least their realizations) over them. We now make the connection to Gushel--Mukai varieties.

In the general discussion above we now take as lattice
\begin{equation}
\label{eq:GM6Lattice}
L = E_8(-1)^{\oplus 2} \oplus U^{\oplus 2} \oplus I(-2)^{\oplus 2}\, ,
\end{equation}
where $I(-2) = \ZZ$ with form given by $b(1,1) = -2$. Recall that we write $b_L \colon L \otimes L \to \ZZ$ for the symmetric bilinear form on~$L$. If $Y$ is a complex $6$-dimensional Gushel--Mukai variety then $H^6(Y,\ZZ)_{00}$ with its intersection pairing is isomorphic to~$L$; see \cite{DK-Kyoto}, Proposition~3.9. As $\Discr(L) = (\ZZ/2\ZZ)^2$, the induced pairing on $L_{(p)} = L \otimes \ZZ_{(p)}$ is perfect for any prime $p>2$.

Let $V_5$ be a free $\ZZ_{(p)}$-module of rank~$5$. As before, we write $\CGr(2,V_5) \subset \PP(\ZZ_{(p)} \oplus \wedge^2 V_5)$ for the cone over the Grassmannian $\Grass = \Grass(2,V_5)$. Consider the scheme~$\scrQ$ over~$\ZZ_{(p)}$ of quadrics in $\PP(\ZZ_{(p)} \oplus \wedge^2 V_5)$. Let $\scrM^\flat \subset \scrQ$ denote the open subset of all quadrics~$Q$ such that $\CGr(2,V_5) \cap Q$ is smooth of dimension~$6$, and let $f^\flat \colon X^\flat \to \scrM^\flat$ be the tautological family of $6$-dimensional Gushel--Mukai varieties.

We have a morphism $\gamma \colon X^\flat \to \Grass_{\scrM^\flat} = (\Grass \times_{\Spec(\ZZ_{(p)})} \scrM^\flat)$. Over the characteristic~$0$ fibre~$\scrM^\flat_\QQ$ we have a system of realizations with $\ZZ$-coefficients $\motH^6\bigl(X^\flat/\scrM^\flat_\QQ;\ZZ\bigr)(3)$, and we define
\[
\motP^\flat = \motH^6\bigl(X^\flat/\scrM^\flat_\QQ;\ZZ\bigr)(3)_{00} \subset \motH^6\bigl(X^\flat/\scrM^\flat_\QQ;\ZZ\bigr)(3)
\]
to be the subobject that in each cohomological realization is given by taking the orthogonal complement of the image of $\motH^6\bigl(\Grass_{\scrM^\flat}/\scrM^\flat_\QQ;\ZZ\bigr)(3)$ under~$\gamma^*$. (For the \'etale component this has to be interpreted on the level of sheaves.) The intersection pairing induces a morphism
\[
b_{\motP^\flat} \colon \motP^\flat \otimes \motP^\flat \to \unitmot
\]
in $\Real(\scrM^\flat_\QQ;\ZZ)$.

We want to pass to a double cover $\scrM \to \scrM^\flat$ by choosing orientations. As we want to do this over~$\ZZ_{(p)}$, it is convenient to use the \'etale component of~$\motP^\flat$ for this. For concreteness we shall use the $2$-adic component. We then define an orientation of~$\motP^\flat$ over an $\scrM^\flat$-scheme~$T$ to be an isomorphism of \'etale sheaves $\theta \colon \det(\motP_2^\flat|_T) \isomarrow \det(L_2)_T$ (where $L_2 = L \otimes \ZZ_2$) such that the diagram
\[
\begin{tikzcd}[column sep=large]
\det(\motP_2^\flat|_T) \otimes \det(\motP_2^\flat|_T) \ar[d,"\theta \otimes \theta"] \ar[r,"\det(b_{\motP^\flat})"]& \ZZ_{2,T} \arrow[equal,d]\\
\det(L_2)_T \otimes \det(L_2)_T \ar[r,"\det(b_L)"]& \ZZ_{2,T}
\end{tikzcd}
\]
is commutative. (As we are jumping ahead a little bit, let us note that $\motP_2^\flat$ exists as an \'etale sheaf on the whole of~$\scrM^\flat$, so that this definition is meaningful in mixed characteristic. We shall return to this in Proposition~\ref{prop:ladicExtend}.) If $\theta$ is such an orientation then over~$T_\CC$ we have a trivialization $\det(\motP^\flat_{\ZZ,\infty}|_{T_\CC}) \isomarrow \det(L)_{T_\CC}$ (where $\infty \colon \QQ \to \CC$ is the unique complex embedding of~$\QQ$) that is compatible with~$\theta$ under~$i_{\et,\infty}$. We shall again denote this trivialization by~$\theta$.

The choice of an orientation defines a $(\ZZ/2\ZZ)$-torsor $\scrM \to \scrM^\flat$. Let
\[
f \colon X \to \scrM
\]
be the pullback of the family~$f^\flat$, let $\motP = \motH^6\bigl(X/\scrM_\QQ;\ZZ\bigr)(3)_{00}$ be the pullback of the family of realizations~$\motP^\flat$, let $b_{\motP} \colon \motP \otimes \motP \to \unitmot$ be the pairing induced by~$b_{\motP^\flat}$, and let $\theta_{\motP} \colon \det(\motP_2) \isomarrow \det(L_2)_\scrM$ be the tautological orientation. The same notation~$\theta_{\motP}$ will be used for the corresponding isometry $\det(\motP_{\ZZ,\infty}) \isomarrow \det(L)_{\scrM_\CC}$ over~$\scrM_\CC$.

\begin{remark}
The notation ``$\motP$'' is inspired by the fact that this family of motives is the analogue of what in the setting of K3 surfaces is the primitive cohomology in degree~$2$. It therefore plays the same role as the family of motives that in~\cite{MP-K3} is denoted by the letter~$\bm{P}$. (This part is also sometimes referred to as the ``variable part''.) 
\end{remark}

\subsection{}
Still with~$L$ as in~\eqref{eq:GM6Lattice}, we have the system of realizations~$\motL$ over~$\scrS_\QQ$ (with $\ZZ$-coefficients) and the symmetric bilinear form $b_{\motL} \colon \motL \otimes \motL \to \unitmot$. We write $\infty \colon \QQ \to \CC$ for the complex embedding of~$\QQ$. Over~$\scrS_\CC$ we then have a polarized $\ZZ$-VHS $(\motL_{\ZZ,\infty},\infty^*\Fil^\bullet,b_{\motL})$. The pullback of $(\motL_{\ZZ,\infty},b_{\motL})$ to the universal cover of~$\scrS_\CC$ is isomorphic to the constant sheaf~$L$ with its pairing~$b_L$. Over~$\scrS_\CC$ we inherit from this an orientation $\theta_{\motL} \colon \det(\motL_{\ZZ,\infty}) \isomarrow \det(L)_{\scrS_\CC}$ such that $\det(b_L) \circ (\theta_{\motL} \otimes \theta_{\motL}) = \det(b_\motL)$. The next lemma implies that this orientation in fact comes from an isomorphism of $2$-adic \'etale sheaves
\[
\theta_{\motL} \colon \det(\motL_{\ZZ_2}) \isomarrow \det(L_2)_{\scrS_\QQ}
\]
over~$\scrS_\QQ$, so we have an orientation of~$\motL$ over~$\scrS_\QQ$ in the sense discussed above. Further we have an isometry $\eta_{\motL} \colon \Discr(\motL_{\ZZ,\infty}) \isomarrow \Discr(L)_{\scrS_\CC}$ (notation as in~\ref{subsec:Discr}), which by the next lemma comes from an isometry of sheaves
\[
\eta_{\motL} \colon \Discr(\motL_{\ZZ_2}) \isomarrow \Discr(L)_{\scrS_\QQ}\, .
\]
(Note that $\Discr(\motL_{\ZZ_2}) \cong \Discr(\motL_{\ZZ,\infty})$, as the latter is a $2$-group.)

\begin{lemma}
The \'etale sheaves $\det(\motL_{\ZZ_2})$ and $\Discr(\motL_{\ZZ_2})$ over~$\scrS_\QQ$ are constant.
\end{lemma}

\begin{proof}
For each connected component $Y \subset \tilde{\scrS}_\QQ$, choose a geometric base point~$\bar{y}$ and consider the action of the arithmetic fundamental group~$\pi_1(Y,\bar{y})$ on $\tilde{\motL}_{\bar{y}} \subset \ul{\End}(\motH_{\bar{y}})$. The action on $\ul{\End}(\motH_{\bar{y}})$ preserves the integral structure. Also, it follows from \cite{MP-IntCanMod}, Lemma~1.4(iii), that $\pi_1(Y,\bar{y})$ acts via the group~$\tilde{G}(\AA_{\fin})$. Hence $\pi_1(Y,\bar{y})$ acts through the group $\tilde{K}(L) = \tilde{G}(\AA_{\fin}) \cap \bigl(C(L) \otimes \Zhat\bigr)^*$, and because the image of this group in~$G$ is the group~$K(L)$ of~\eqref{eq:KLDef}, the assertion follows.
\end{proof}

The following result, which is Proposition~4.3 in~\cite{MP-K3} (and which goes back to at least~\cite{Milne-ShVarsMot}), expresses that the complex Shimura variety~$\scrS_\CC$ is a moduli space for such Variations of Hodge Structure.

\begin{proposition}
\label{prop:MPK3-4.3}
Let $(V,\Fil^\bullet,b_V)$ be a polarized $\ZZ$-VHS over a complex manifold~$T$ which is of type $(-1,1) + (0,0) + (1,-1)$ with Hodge numbers
\[
h^{-1,1} = 1\, ,\quad h^{0,0} = 20\, ,\quad h^{1,-1} = 1\, .
\]
Let further be given isometries 
\[
\theta_V \colon \det(V) \isomarrow \det(L)_T\quad \text{and}\quad
\eta_V \colon \Discr(V) \isomarrow \Discr(L)_T\, ,
\] 
such that for every point $t \in T(\CC)$ there exists an isometry $V_t \isomarrow L$ which induces the values at~$t$ of $\theta_V$ and~$\eta_V$. Then there exists a unique morphism $T \to \scrS_\CC$ such that $(V,\Fil^\bullet,b_V,\theta_V,\eta_V)$ is isomorphic to the pullback of $(\motL_{\ZZ,\infty},\infty^*\Fil^\bullet,b_{\motL},\theta_{\motL},\eta_{\motL})$.
\end{proposition}

\subsection{}
We want to apply this proposition to define a morphism $\iota_\CC \colon \scrM_\CC \to \scrS_\CC$, which in the next proposition will then be shown to be defined over~$\QQ$. In order to do this, we first need to define an identification~$\eta_{\motP}$ of the discriminant group of~$\motP$ with the constant sheaf~$\Discr(L)$ over~$\scrM_\CC$.

Write $\motH(X) = \motH^6\bigl(X/\scrM_\QQ;\ZZ\bigr)(3)$ and $\motH(\Grass) = \motH^6\bigl((\Grass \times \scrM_\QQ)/\scrM_\QQ;\ZZ\bigr)(3)$. The morphism $\gamma \colon X \to (\Grass \times \scrM_\QQ)$ over~$\scrM_\QQ$ induces $\gamma^* \colon \motH(\Grass) \to \motH(X)$, which in all realizations is injective with torsion-free cokernel. Of course, $\motH(\Grass)$ is a constant system of realizations, in the sense that it is the pullback of~$\motH^6\bigl(\Grass/\QQ;\ZZ\bigr)(3)$, and in fact (still with $\infty \colon \QQ\to \CC$)
\[
\motH_\infty(\Grass) \cong I(2)^{\oplus 2}_{\scrM_\CC}
\]
where we recall that $I(2) = \ZZ$ with $b(1,1) = 2$.

As explained in \cite{DK-Kyoto}, Proposition~3.9, the fibres of $\motH_\infty(X)$ are isomorphic to $E_8(-1)^{\oplus 2} \oplus U^{\oplus 4}$, and the orthogonal complement of a primitively embedded copy of~$I(2)^{\oplus 2}$ is the lattice~$L$. As the system of realizations~$\motP$ is defined as the orthogonal complement of~$\gamma^*\motH(\Grass)$ inside~$\motH(X)$, it follows that we have isometries (with respect to the $\QQ/\ZZ$-valued discriminant pairings)
\begin{equation}
\label{eq:DiDiDi}
\Discr(\motP_{\ZZ,\infty}) \isomarrow \Discr\bigl(I(2)^{\oplus 2}\bigr)_{\scrM_\CC} \isomarrow \Discr(L)_{\scrM_\CC}\, ,
\end{equation}
and we define $\eta_{\motP} \colon \Discr(\motP_{\ZZ,\infty}) \isomarrow \Discr(L)_{\scrM_\CC}$ to be the composition. (In general, the maps in~\eqref{eq:DiDiDi} are isometries only if we change the sign of the discriminant pairing on the middle term; see \cite{Nikulin}, \S 1.6. In the case at hand, however, the discriminant groups are killed by~$2$ so changing the sign has no effect.)

By Proposition~\ref{prop:MPK3-4.3}, there is a unique morphism
\[
\iota_\CC \colon \scrM_\CC \to \scrS_\CC
\]
such that we have an isomorphism of $\ZZ$-VHS
\begin{equation}
\label{eq:psiC}
\psi_\CC \colon (\motP_{\ZZ,\infty},\Fil^\bullet) \isomarrow \iota_\CC^*(\motL_{\ZZ,\infty},\Fil^\bullet)\, ,
\end{equation}
which is compatible with the pairings~$Q$, the orientations~$\theta$, and the trivializations~$\eta$. (In order to simplify notation we here write $\Fil^\bullet$ instead of $\infty^*\Fil^\bullet$.)

\begin{proposition}
\label{prop:descendQ}
\begin{enumerate}
\item\label{iota/Q} The morphism $\iota_\CC$ is defined over~$\QQ$, i.e., there exists a (unique) morphism $\iota \colon \scrM_\QQ \to \scrS_\QQ$ that after extension of scalars to~$\CC$ gives back~$\iota_\CC$.

\item\label{psi/E} There exists a field extension $\QQ \subset E$ of degree at most~$2$ such that there is an isomorphism
\[
\psi \colon \motP|_{\scrM_E} \isomarrow \iota_E^*\bigl(\motL|_{\scrS_E}\bigr)\, ,
\]
\end{enumerate}
in $\Real(\scrM_E;\ZZ)$ which after extension of scalars to~$\CC$ gives back~\eqref{eq:psiC}.
\end{proposition}

The proof of the proposition will take up the rest of this section. This result is the analogue of \cite{MP-K3}, Corollary~5.4. As mentioned there, in the case of K3 surfaces the fact that $\iota_\CC$ is defined over~$\QQ$ was first proven by Rizov~\cite{Rizov}. His proof uses arguments for which we have no analogue in the setting of Gushel--Mukai varieties. As indicated by Madapusi Pera, there is also a proof which is based on the theory of absolute Hodge classes, following the ideas in~\cite{Milne-ShVarsMot}. This is the argument we shall use here. Unfortunately, \cite{MP-K3} only gives a brief hint of how the argument works, and all verifications are omitted. As this is really one of the key steps in the proof of the Tate conjecture, and as the proof can be a little confusing, we give full details.

\subsection{}
For $\alpha \colon \CC \to \CC$ an embedding of $\CC$ into~$\CC$, write $\scrM_\alpha = \scrM_\CC \otimes_{\CC,\alpha} \CC$ and $\scrS_\alpha = \scrS_\CC \otimes_{\CC,\alpha} \CC$, and let $\iota_\alpha\colon \scrM_\alpha \to \scrS_\alpha$ be the morphism obtained from $\iota_\CC \colon \scrM_\CC \to \scrS_\CC$ by base change via~$\alpha$. Because $\scrM_\CC$ and $\scrS_\CC$ are obtained as complexifications of schemes~$\scrM$ and~$\scrS$ over~$\QQ$ we have canonical (descent) isomorphisms
\[
d_\scrM(\alpha) \colon \scrM_\CC \isomarrow \scrM_\alpha\, ,\quad
d_\scrS(\alpha) \colon \scrS_\CC \isomarrow \scrS_\alpha\, .
\]
These are isomorphisms of schemes over~$\CC$. Note that, by definition of~$\scrM_\alpha$ and~$\scrS_\alpha$ as pullbacks, we also have projection morphisms
\[
\pr_\scrM(\alpha) \colon \scrM_\alpha \to \scrM_\CC\, ,\quad \pr_\scrS(\alpha) \colon \scrS_\alpha \to \scrS_\CC\, ,
\]
but these are \emph{not} the inverses of $d_\scrM(\alpha)$ and~$d_\scrS(\alpha)$. (In fact, they are not morphisms over~$\CC$.) If $F$ is any sheaf on~$\scrM_\CC$, we denote by $\alpha^*F$ its pullback under~$\pr_\scrM(\alpha)$; analogously for~$\scrS$.

For the duration of the proof we shall simplify notation, writing
\[
\motM = \motP|_{\scrM_\CC}\quad \text{and}\quad \motM^\prime = \motL|_{\scrS_\CC}\, .
\]
The data involved in these systems are the following.
\begin{itemize}
\item For every $\alpha \colon \CC \to \CC$ we have local systems with $\ZZ$-coefficients $\motM_{\ZZ,\alpha}$ on~$\scrM_\alpha$ and $\motM^\prime_{\ZZ,\alpha}$ on~$\scrS_\alpha$. These are given by
\[
\motM_{\ZZ,\alpha} = d_\scrM(\alpha)^{-1,*}\motP_{\ZZ,\infty}\, ,\quad
\motM^\prime_{\ZZ,\alpha} = d_\scrS(\alpha)^{-1,*}\motL_{\ZZ,\infty}\, .
\]

\item We have flat vector bundles equipped with a filtration by subbundles $(\motM_{\dR},\nabla,\Fil^\bullet)$ on~$\scrM_\CC$ and $(\motM^\prime_{\dR},\nabla,\Fil^\bullet)$ on~$\scrS_\CC$. These are given by
\[
(\motM_{\dR},\nabla,\Fil^\bullet) = \infty^*(\motP_{\dR},\nabla,\Fil^\bullet)\, ,\quad
(\motM^\prime_{\dR},\nabla,\Fil^\bullet) = \infty^*(\motL_{\dR},\nabla,\Fil^\bullet)\, .
\]
(Here $\infty^*$ means base change via $\scrM_\CC \to \scrM$, resp.\ $\scrS_\CC \to \scrS$.)

\item We have $\Zhat$-local systems in the pro-\'etale topology
\[
\motM_{\et} = \infty^*\motP_{\et}\, ,\quad
\motM^\prime_{\et} = \infty^*\motL_{\et}
\]
on $\scrM_\CC$, resp.~$\scrS_\CC$.

\item For every $\alpha$ we have isomorphisms of flat bundles
\[
i_{\dR,\alpha} \colon \motM_{\ZZ,\alpha} \otimes \scrO_{\scrM_\alpha} \isomarrow \alpha^*\motM_{\dR}\, ,\quad
i^\prime_{\dR,\alpha} \colon \motM^\prime_{\ZZ,\alpha} \otimes \scrO_{\scrS_\alpha} \isomarrow \alpha^*\motM^\prime_{\dR}
\]
as well as isomorphisms
\[
i_{\et,\alpha} \colon \motM_{\ZZ,\alpha} \otimes \Zhat \isomarrow \alpha^* \motM_{\et}\, ,\quad
i^\prime_{\et,\alpha} \colon \motM^\prime_{\ZZ,\alpha} \otimes \Zhat \isomarrow \alpha^* \motM^\prime_{\et}\, .
\]
These are obtained from the comparison isomorphisms $i_{\dR,\infty}$ and~$i_{\et,\infty}$ of~$\motP$ (resp.\ the comparison isomorphisms $i^\prime_{\dR,\infty}$ and~$i^\prime_{\et,\infty}$ of~$\motL$) by
\begin{multline*}
\qquad
i_{\dR,\alpha} = d_\scrM(\alpha)^{-1,*}(i_{\dR,\infty}) \colon \motM_{\ZZ,\alpha} \otimes \scrO_{\scrM_\alpha} \isomarrow\\
d_\scrM(\alpha)^{-1,*}\infty^* \motP_{\dR} = \pr_\scrM(\alpha)^* \infty^* \motP_{\dR} = \alpha^* \motM_{\dR}\qquad
\end{multline*}
and
\begin{multline*}
i_{\et,\alpha} = d_\scrM(\alpha)^{-1,*}(i_{\et,\infty}) \colon \motM_{\ZZ,\alpha} \otimes \Zhat \isomarrow \\
d_\scrM(\alpha)^{-1,*}\infty^* \motP_{\et} = \pr_\scrM(\alpha)^* \infty^* \motP_{\et} = \alpha^* \motM_{\et}\, ,
\end{multline*}
and similarly $i^\prime_{\dR,\alpha} = d_\scrS(\alpha)^{-1,*}(i^\prime_{\dR,\infty})$ and $i^\prime_{\et,\alpha} = d_\scrS(\alpha)^{-1,*}(i^\prime_{\et,\infty})$. In this construction we use that the two compositions
\[
\scrM_\alpha \xrightarrow{~d_\scrM(\alpha)^{-1}~} \scrM_\CC \tto \scrM_\QQ
\qquad\text{and}\qquad
\scrM_\alpha \xrightarrow{~\pr_\scrM(\alpha)~} \scrM_\CC \tto \scrM_\QQ
\]
are equal; analogously with $\scrS$ instead of~$\scrM$.
\end{itemize}

\subsection{}
We now use that $\motM$ and~$\motM^\prime$ are families of abelian motives, which means they are cut out by absolute Hodge (or even motivated) projectors in systems of realizations coming from abelian schemes. For $\motM^\prime = \motL|_{\scrS_\CC}$ this is clear from its construction, for~$\motM$ this is proven in our paper~\cite{FuMoonen-GM1}, Section~6. This allows us to use Deligne's result, documented in~\cite{DMOS}, which says that classes (in our setting: sections) in the de Rham or \'etale realizations that are Hodge with respect to one complex embedding, are Hodge with respect to every embedding. The generalisation of this result to families of abelian motives presents no difficulties, as we can just apply Deligne's result fibrewise.

The isomorphism $\psi$ of \eqref{eq:psiC} has a de Rham component $\psi_{\dR} \colon \motM_{\dR} \isomarrow \iota_\CC^* \motM^\prime_{\dR}$. If we take $\alpha = \id \colon \CC \to \CC$ then we also have the isomorphism $\psi_{\Betti} \colon \motM_{\ZZ,\id} \isomarrow \iota_\CC^*\motM^\prime_{\ZZ,\id}$ ($\Betti$ for ``Betti''), and we can define $\psi_{\et} \colon \motM_{\et} \isomarrow \iota_\CC^*\motM^\prime_{\et}$ as the $\Zhat$-linear extension of~$\psi_{\Betti}$ via the comparison isomorphisms $i_{\et,\id}$ and~$i^\prime_{\et,\id}$. If we view these components as sections of the system of realizations $\motM^\vee \otimes \iota_\CC^*\motM^\prime$ then, by construction, $(\psi_{\dR},\psi_{\et})$ is Hodge at the complex embedding $\alpha = \id$.

The fact that Hodge sections are absolute Hodge sections therefore implies that for every $\alpha \colon \CC \to \CC$ there exists an isomorphism
\[
\psi_\alpha \colon \motM_{\ZZ,\alpha} \isomarrow \iota_\alpha^* \motM^\prime_{\ZZ,\alpha}
\]
such that the diagrams
\begin{equation}
\label{diag:dRalpha}
\begin{tikzcd}[column sep=huge]
\motM_{\ZZ,\alpha} \otimes \scrO_{\scrM_\alpha} \arrow{d}[swap]{\psi_\alpha \otimes \id} \ar[r,"i_{\dR,\alpha}"] & \alpha^* \motM_{\dR} \ar[d,"\pr_\scrM(\alpha)^*(\psi_{\dR})"]\\
\iota_\alpha^*\bigl(\motM^\prime_{\ZZ,\alpha} \otimes \scrO_{\scrS_\alpha}\bigr) \ar[r,"\iota_\alpha^*(i^\prime_{\dR,\alpha})"]& \iota_\alpha^*\bigl(\alpha^*\motM^\prime_{\dR}\bigr) = \alpha^*\bigl(\iota_\CC^*\motM^\prime_{\dR} \bigr)
\end{tikzcd}
\end{equation}
and
\begin{equation}
\label{diag:etalpha}
\begin{tikzcd}[column sep=huge]
\motM_{\ZZ,\alpha} \otimes \Zhat \arrow{d}[swap]{\psi_\alpha \otimes \id} \ar[r,"i_{\et,\alpha}"] & \alpha^*\motM_{\et} \ar[d,"\pr_\scrM(\alpha)^*(\psi_{\et})"]\\
\iota_\alpha^*\bigl(\motM^\prime_{\ZZ,\alpha} \otimes \Zhat\bigr) \ar[r,"\iota_\alpha^*(i^\prime_{\et,\alpha})"] & \iota_\alpha^*\bigl(\alpha^*\motM^\prime_{\et}\bigr) = \alpha^*\bigl(\iota_\CC^*\motM^\prime_{\et} \bigr)
\end{tikzcd}
\end{equation}
are commutative. (Recall that we use $\alpha^*F$ as a shorthand for $\pr_\scrM(\alpha)^*F$ or $\pr_\scrS(\alpha)^*F$.)

\subsection{}
To prove part~\ref{iota/Q} of the proposition we take an arbitrary $\alpha \colon \CC \to \CC$ and we consider the composition
\[
j \colon \scrM_\CC \xrightarrow{~d_\scrM(\alpha)~} \scrM_\alpha \xrightarrow{~\iota_\alpha~} \scrS_\alpha \xrightarrow{~d_\scrS(\alpha)^{-1}~} \scrS_\CC\, .
\]
(So $j = {}^\alpha\iota_\CC$.) We are done if we can show that there exists an isomorphism of polarized $\ZZ$-VHS
\[
\phi \colon (\motP_{\ZZ,\infty},\Fil^\bullet,b_{\motP}) \isomarrow j^*(\motL_{\ZZ,\infty},\Fil^\bullet,b_{\motL})\, ,
\]
which is compatible with the orientations~$\theta$ and the discriminant rigidifications~$\eta$ on both sides. Indeed, if such an isomorphism~$\phi$ exists then $j = \iota_\CC$ by the uniqueness in Proposition~\ref{prop:MPK3-4.3}, and since $\alpha$ is arbitrary it follows that $\iota_\CC$ is defined over~$\QQ$.

To construct the desired isomorphism~$\phi$, we take as Betti component 
\[
\phi_{\Betti} \colon \motP_{\ZZ,\infty} \isomarrow j^* \motL_{\ZZ,\infty}
\] 
the isomorphism
\[
\phi_{\Betti} = d_\scrM(\alpha)^*(\psi_\alpha) \colon \motP_{\ZZ,\infty} = d_\scrM(\alpha)^*\motM_{\ZZ,\alpha} \longrightarrow j^* \motL_{\ZZ,\infty} = d_\scrM(\alpha)^* \iota_\alpha^* \motM^\prime_{\ZZ,\alpha} \, .
\]
As de Rham component we take
\begin{multline*}
\qquad \phi_{\dR} = d_\scrM(\alpha)^* \pr_\scrM(\alpha)^*(\psi_{\dR}) \colon \infty^*\motP_{\dR} = d_\scrM(\alpha)^* \pr_\scrM(\alpha)^* \motM_{\dR} \\
\longrightarrow j^*\infty^* \motL_{\dR} = d_\scrM(\alpha)^* \pr_\scrM(\alpha)^* \iota_\CC^* \motM^\prime_{\dR}\, .\qquad
\end{multline*}
To see that $\phi_{\Betti}$ and $\phi_{\dR}$ are compatible, in the sense that $\phi_{\dR} \circ i_{\dR,\infty} = j^*(i^\prime_{\dR,\infty}) \circ (\phi_{\Betti} \otimes \id)$, apply $d_\scrM(\alpha)^*$ to diagram~\eqref{diag:dRalpha}. As it is clear that $\phi_{\dR}$ is compatible with the Hodge filtrations and polarizations, we have found an isomorphism~$\phi$ of $\ZZ$-VHS as desired. It only remains to be verified that $\phi$ is compatible with the isomorphims $\theta$ and~$\eta$ on both sides.

\subsection{}
\label{subsec:thetacompat}
To see that $\phi$ is compatible with the isomorphisms~$\theta$, first note that~$\motM_2$, the $2$-adic component of~$\motM_\et$, is the pull-back of~$\motP_2$ to~$\scrM_\CC$, and that by construction we have an orientation~$\theta_{\motP}$ of~$\motP$ over~$\scrM$. We therefore have a canonical isomorphism $d_\scrM(\alpha)^* \pr_\scrM(\alpha)^* \det(\motM_2) \cong \det(\motM_2)$, and under this isomorphism we have $d_\scrM(\alpha)^* \pr_\scrM(\alpha)^*(\theta_\motP) = \theta_\motP$. Likewise, we have an isomorphism $d_\scrS(\alpha)^* \pr_\scrS(\alpha)^* \det(\motL_2) \cong \det(\motL_2)$, and $d_\scrS(\alpha)^* \pr_\scrS(\alpha)^*(\theta_\motL) = \theta_\motL$ under this isomorphism.

Because $\psi$ is compatible with orientations, the diagram
\[
\begin{tikzcd}
\det(\motM_2) \ar[dd,"\det(\psi_\et)"'] \ar[dr,"\theta_\motP"] & \\
& \det(L_2)_{\scrM_\CC}\\
\iota_\CC^* \det(\motM^\prime_2) \ar[ur,"\iota_\CC^* \theta_{\motL}"']
\end{tikzcd}
\]
is commutative. Applying $d_\scrM(\alpha)^* \pr_\scrM(\alpha)^*$ and using diagram~\eqref{diag:etalpha} then gives a commutative diagram
\[
\begin{tikzcd}[column sep=small]
d_\scrM(\alpha)^* \det(\motM_{\ZZ,\alpha} \otimes \ZZ_2) \ar[r,"\sim"] \ar[dd,"\det(\phi_{\Betti})"']
& d_\scrM(\alpha)^* \pr_\scrM(\alpha)^* \det(\motM_2)  \ar[dd,"d_\scrM(\alpha)^* \pr_\scrM(\alpha)^*\det(\psi_\et)"'] \ar[dr,"d_\scrM(\alpha)^* \pr_\scrM(\alpha)^* \theta_{\motP}"]  \\
&& \det(L_2)_{\scrM_\CC}\\
d_\scrM(\alpha)^*  \iota_\alpha^* \det(\motM^\prime_{\ZZ,\alpha} \otimes \ZZ_2) \ar[r,"\sim"]
& d_\scrM(\alpha)^* \pr_\scrM(\alpha)^*\iota_\CC^* \det(\motM^\prime_2) \ar[ur,"d_\scrM(\alpha)^* \pr_\scrM(\alpha)^*\iota_\CC^*\theta_{\motL}"']
\end{tikzcd}
\]
(Of course, $d_\scrM(\alpha)^* \pr_\scrM(\alpha)^* \det(L_2)_{\scrM_\CC} = \det(L_2)_{\scrM_\CC}$.) Because $\det(\motP_\ZZ) = d_\scrM(\alpha)^* \det(\motM_{\ZZ,\alpha})$ and $j^* \det(\motL_{\ZZ,\alpha}) = d_\scrM(\alpha)^* \iota_\alpha^* \det(\motM^\prime_{\ZZ,\alpha})$, the remarks made at the beginning of this subsection then give that the diagram
\[
\begin{tikzcd}
\det(\motP_{\ZZ,\infty} \otimes \ZZ_2) \ar[dd,"\det(\phi_{\Betti} \otimes \id)"'] \ar[dr,"\theta_\motP"] & \\
& \det(L_2)_{\scrM_\CC}\\
j^* \det(\motL_{\ZZ,\infty} \otimes \ZZ_2) \ar[ur,"j^* \theta_{\motL}"']
\end{tikzcd}
\]
is commutative, which means that $\phi$ is compatible with the isomorphisms~$\theta$.

\subsection{}
To show that $\phi$ is compatible with the isomorphisms~$\eta$, note that the discriminant groups involved are all isomorphic to $(\ZZ/2\ZZ)^2$ and that we have natural isomorphisms $\Discr(\motM_{\ZZ,\alpha}) \cong \Discr(\motM_2)$ and $\Discr(\motM^\prime_{\ZZ,\alpha}) \cong \Discr(\motM^\prime_2)$, for all~$\alpha$. With this remark, the argument is the same as in~\ref{subsec:thetacompat}, everywhere replacing ``$\det$'' by~``$\Discr$'' and ``$\theta$'' by~``$\eta$''. We leave it to the reader to write out the details.

\subsection{}
Finally, part~\ref{psi/E} of Proposition~\ref{prop:descendQ} readily follows from the fact that the group of automorphisms of the system of realizations~$\motP|_{\scrM_\CC}$ over~$\scrM_\CC$ that respect~$b_{\motP}$ and are compatible with~$\theta_{\motP}$ is $\{\pm \id\}$.

This completes the proof of Proposition~\ref{prop:descendQ}. \qed

\begin{remark}
Part~\ref{psi/E} of the proposition is one of the places where the situation for Gushel--Mukai varieties differs from the case of K3 surfaces. The point is that the ``variable'' part~$H^6(X)_{00}$ of the middle cohomology has even rank, whereas the primitive cohomology of a K3 surface has rank~$21$. In the setting of Gushel--Mukai varieties, working with $\ZZ$-coefficients and then rigidifying the determinant (through the choice of an orientation) does not suffice to kill all non-trivial automorphisms. A similar problem occurs in the case of cubic fourfolds, though this is not discussed in the brief sketch given in~\cite{MP-K3}, Section~6.
\end{remark}

\section{Matching the crystalline realizations}
\label{sec:MatchCryst}

\subsection{}
\label{subsec:PAdicNot}
Now that we have established Proposition~\ref{prop:descendQ}, it is time to study the $p$-adic extension of all objects involved. While it is tempting to develop a notion of ``families of motives'' over $p$-adic schemes, we will restrict ourselves to the ingredients that are essential for the proof of the Tate conjecture, and treat them in a somewhat ad hoc manner. At this point it will be convenient for us to pass to suitable level covers, so we shall no longer need to consider Shimura stacks.

The following notation will be in force throughout this section. We retain the notation introduced in Section~\ref{subsec:ConnectToGM}; in particular, $L$ is the lattice given in~\eqref{eq:GM6Lattice} and $\scrM$ and~$\motP$ are as defined there. We fix a prime number~$p$ with $p\geq 5$. We choose an odd integer $N > 4$ such that $p \nmid \varphi(N)$ (Euler totient function), and we consider the Shimura varieties $\tilde{\scrS}_N \to \scrS_N$ over $\Spec(\ZZ_{(p)})$ as in Remark~\ref{rem:mu2gerbe}. Define $\tilde{\scrM} = \scrM \times_{\scrS} \tilde{\scrS}_N$, and define~$\tilde{\motP}$ as the pullback of the system of realizations~$\motP$ under $\tilde{\scrM}_\QQ \to \scrM_\QQ$, which from now will be regarded as a system of realizations with $\ZZ_{(p)}$-coefficients. The notation~$\tilde{\motL}$ and~$\motH$ from now on refers to the pullbacks under the morphism $\tilde{\scrS}_{N,\QQ} \to \tilde{\scrS}_\QQ$ of the systems of realizations previously called $\tilde{\motL}$ and~$\motH$. These, too, are now viewed as systems of realizations with $\ZZ_{(p)}$-coefficients. Let $\tilde{\iota}_\QQ \colon \tilde{\scrM}_\QQ \to \tilde{\scrS}_{N,\QQ}$ be the pullback of the morphism~$\iota$ of Proposition~\ref{prop:descendQ}.

Fixing a field extension $\QQ \subset E$ of degree at most~$2$ as in part~\ref{psi/E} of Proposition~\ref{prop:descendQ}, choose a prime ideal $\frp \subset \scrO_E$ above~$p$, let $\Lambda = \scrO_{E,\frp} \twoheadrightarrow \kappa = \Lambda/\frp$ be the local ring at~$\frp$ and its residue field, and let
\[
E_\frp \supset \hat\Lambda = \hat{\scrO}_{E,\frp}
\]
denote the completions of $E\supset \scrO_E$ at~$\frp$.

With this notation we have an isomorphism $\psi \colon \tilde{\motP} \isomarrow \tilde{\iota}_\QQ^* \tilde{\motL}$ in $\Real(\tilde{\scrM}_E;\ZZ_{(p)})$.

\begin{lemma}
The scheme $\tilde{\scrM}$ is smooth over~$\ZZ_{(p)}$ and the morphism $\tilde{\iota}_\QQ$ uniquely extends to a morphism $\tilde{\iota} \colon \tilde{\scrM} \to \tilde{\scrS}_N$ over~$\ZZ_{(p)}$.
\end{lemma}

\begin{proof}
For the first assertion it suffices to show that the morphism $\tilde{\scrS}_N \to \scrS_N$ is smooth, as $\scrM$ is smooth over~$\ZZ_{(p)}$ and $\scrM \times_{\scrS} \scrS_N \to \scrM$ is a $G(\ZZ/n\ZZ)$-torsor. As remarked in~\ref{rem:mu2gerbe}, $\scrS_{N,\QQ}$ is the quotient of~$\tilde{\scrS}_{N,\QQ}$ by the (free) action of the group $\Delta/\{\pm 1\}$. This action extends to an action on the integral canonical model~$\tilde{\scrS}_N$ and because $p$ does not divide the order of~$\Delta$, this extended action is again free (see Section~3.21 of~\cite{BM-Durham}), with quotient~$\scrS_N$.

The second assertion follows because~$\tilde{\scrS}_N$ is the integral canonical model of~$\tilde{\scrS}_{N,\QQ}$.
\end{proof}

\begin{proposition}
\label{prop:ladicExtend}
Let $\ell$ be a prime number with $\ell \neq p$. Then $\tilde{\motP}_\ell$ (resp.\ $\tilde{\motL}_\ell \subset \ul{\End}(\motH_\ell)$) uniquely extends to a smooth $\QQ_\ell$-sheaf on~$\tilde{\scrM}$ (resp.~$\tilde{\scrS}_N$), and $\psi_\ell$ extends to an isomorphism of $\QQ_\ell$-sheaves  $\psi_\ell \colon \tilde{\motP}_\ell \isomarrow \tilde{\iota}^* \tilde{\motL}_\ell$ over~$\tilde{\scrM}$.
\end{proposition}

Note that, once the first assertion is proven, the notation $\tilde{\motP}_\ell$ and~$\tilde{\motL}_\ell$ will be used for the extended $\QQ_\ell$-sheaves.

\begin{proof}
For $\tilde{\motP}_\ell$ the assertion is clear, as we have a universal family $f\colon \scrX \to \tilde{\scrM}$ of Gushel--Mukai varieties together with a morphism $\gamma \colon \scrX \to \Grass(2,V_5)$ over~$\tilde{\scrM}$. Similarly, $\motH_\ell$ extends to a smooth $\QQ_\ell$-sheaf over~$\tilde{\scrS}_N$, and then it is automatic that $\tilde{\motL}_\ell \subset \ul{\End}(\motH_\ell)$ extends as well. The last assertion is clear from the fact that these extensions are unique up to isomorphism.
\end{proof}

\begin{corollary}
\label{cor:HlGMss}
Let $k$ be a field of characteristic $p \geq 5$ which is finitely generated over~$\FF_p$. Let $\kbar$ be an algebraic closure of~$k$ and $k^{\perf} \subset \kbar$ the perfect closure of~$k$ inside~$\kbar$. Let $X$ be a Gushel--Mukai variety of dimension~$6$ over~$k$, and let $\ell$ be a prime number different from~$p$. Then the $\ell$-adic Galois representation $\Gal(\kbar/k^{\perf}) \to \GL\bigl(H^\bullet(X_{\kbar},\QQ_\ell)\bigr)$ is completely reducible.
\end{corollary}

\begin{proof}
It suffices to prove that the Galois representation on $H^6(X_{\kbar},\QQ_\ell)$ is semisimple, as $X$ has zero cohomology in odd degree and the cohomology groups $H^{2i}\bigl(X_{\kbar},\QQ_\ell(i)\bigr)$ for $i\neq 3$ are spanned by algebraic cycles and are therefore of Tate type.

Choose a point $m \in \tilde{\scrM}(k)$ such that $X$ is isomorphic to the fibre at~$m$ of the universal family of Gushel--Mukai varieties over~$\tilde{\scrM}$. We view the fibre~$\tilde{\motP}_{\ell,m}$ of the $\ell$-adic sheaf~$\tilde{\motP}_\ell$ at~$m$ as a representation of~$\Aut(\kbar/k) = \Gal(\kbar/k^{\perf})$.

We have a morphism $\gamma \colon X \to \Grass = \Grass(2,V_5)$ over~$k$ and the Galois representation~$\tilde\motP_{\ell,m}$ is, by definition, the orthogonal complement of the image of $\gamma^* \colon H^6(\Grass_{\kbar},\QQ_\ell)(3) \hookrightarrow H^6(X_{\kbar},\QQ_\ell)(3)$ with respect to the intersection pairing. As the morphism~$\gamma$ is finite and $H^*(\Grass_{\kbar},\QQ_\ell)$ is spanned by the classes of algebraic cycles, it follows that $H^6(X_{\kbar},\QQ_\ell)(3) \cong H^6(\Grass_{\kbar},\QQ_\ell)(3) \oplus \tilde\motP_{\ell,m}$.

Let $s = \tilde{\iota}(m) \in \tilde{\scrS}_N(k)$. By Proposition~\ref{prop:ladicExtend}, it now suffices to prove complete reducibility of~$\tilde{\motL}_{\ell,s}$ as a Galois representation. This follows from the results of Zarhin, see Corollary~1 in~\cite{Zarhin-AVCharp}.
\end{proof}

\subsection{}
We next turn to the de Rham and crystalline realizations. By the same argument as in the proof of Proposition~\ref{prop:ladicExtend}, the de Rham realization~$\tilde{\motP}_{\dR}$ over~$\tilde{\scrM}_\QQ$ naturally extends to a filtered flat vector bundle $\bigl(\tilde{\motP}_{\dR},\nabla,\Fil^\bullet \bigr)$ over~$\tilde{\scrM}$, and $\tilde{\motH}_{\dR}$ over~$\tilde{\scrS}_{N,\QQ}$ naturally extends to a filtered flat vector bundle $\bigl(\tilde{\motH}_{\dR},\nabla,\Fil^\bullet \bigr)$ over~$\tilde{\scrS}_N$. Rather than introducing new notation, the old notation ($\tilde{\motP}_{\dR}$, etc.) will from now on refer to these extended objects, just as we did for the $\ell$-adic realizations.

Next we claim that the de Rham realization~$\tilde{\motL}_{\dR}$ over~$\tilde{\scrS}_{N,\QQ}$ extends to a flat sub-bundle $(\tilde{\motL}_{\dR},\nabla) \subset \bigl(\ul{\End}(\tilde{\motH}_{\dR}),\nabla \bigr)$ over~$\tilde{\scrS}_N$. This is proven in \cite{MP-IntCanMod}, Corollary~4.13(i).

On the special fibre of~$\tilde{\scrM}$ we have a crystalline realization $\tilde{\motP}_{\crys}$, which is an $F$-crystal on~$\tilde{\scrM}_{\FF_p}$ relative to~$\ZZ_p$. (Just as with the other incarnations of the family of motives~$\tilde\motP$, we define~$\tilde{\motP}_{\crys}$ by taking the orthogonal complement of the image of~$\gamma^*$ inside the $F$-crystal $R^6f_{\crys,*}\scrO$, where $f$ and~$\gamma$ are as in the proof of Proposition~\ref{prop:ladicExtend} and $\scrO$ denotes the structure sheaf of the big crystalline site of~$\scrX$ over~$\ZZ_p$.) Similarly, we have an $F$-crystal $\tilde{\motH}_{\crys}$ on~$\tilde{\scrS}_{N,\FF_p}$ relative to~$\ZZ_p$, obtained by taking the relative crystalline cohomology in degree~$1$ of the Kuga--Satake abelian scheme over~$\tilde{\scrS}_N$.

At this point we will not attempt to construct a sub-$F$-crystal~$\tilde{\motL}_{\crys}$ inside $\ul{\End}(\tilde{\motH}_{\crys})$. We will only need it ``pointwise'' and for this we can rely on the following result of Madapusi Pera.

\begin{proposition}[Madapusi Pera~\cite{MP-IntCanMod}, Proposition~4.7]
Let $s_0 \in \tilde{\scrS}_N(k)$ be a $k$-valued point, where $k$ is a perfect field of characteristic~$p$ such that $W(k)$ admits an embedding into~$\CC$, and write $\tilde\motH_{\crys,s_0} = s_0^* \tilde{\motH}_{\crys}$. Then there exist an idempotent $\bm{\pi}_{\crys,s_0} \in \End\bigl(\ul{\End}(\tilde\motH_{\crys,s_0})\bigr)^{\phi=1}$ (with $\phi=$ Frobenius) which has the following properties.
\begin{enumerate}
\item\label{picrys:a} Let $K_0(k) \subset K$ be a finite field extension, where $K_0(k)$ denotes the fraction field of~$W(k)$. Let $s \in \tilde{\scrS}_N(\scrO_K)$ be a lift of~$s_0$, and let $s_K  \in \tilde{\scrS}_N(K)$ be the corresponding $K$-valued point. Write $\tilde{\motH}_{\dR,s_K} = s_K^*\tilde{\motH}_{\dR}$, and let $\bm{\pi}_{\dR,s_K}$ be the idempotent that cuts out $\tilde{\motL}_{\dR,s_K} \subset \ul{\End}(\tilde{\motH}_{\dR,s_K})$. Then $\bm{\pi}_{\crys,s_0}$ corresponds to~$\bm{\pi}_{\dR,s_K}$ under the Berthelot--Ogus comparison isomorphism $\tilde{\motH}_{\dR,s_K}  \isomarrow \tilde\motH_{\crys,s_0} \otimes_{W(k)} K$.

\item\label{picrys:b} With notation and assumptions as in~\emph{\ref{picrys:a}}, let $K \subset \Kbar$ be an algebraic closure, and view $\tilde\motH_{p,s_K} = s_K^*\tilde\motH_p$ as a $\ZZ_p$-module with $\Gal(\Kbar/K)$-action. Then $\bm{\pi}_{\crys,s_0}$ corresponds to~$\bm{\pi}_{p,s_K}$  under the comparison isomorphism $\tilde\motH_{p,s_K} \otimes_{\ZZ_p} \mathrm{B}_{\crys} \isomarrow \tilde\motH_{\crys,s_0} \otimes_{W(k)} \mathrm{B}_{\crys}$.
\end{enumerate}
\end{proposition}

Let $\tilde\motL_{\crys,s_0} \subset \ul{\End}(\tilde\motH_{\crys,s_0})$ be the image of~$\bm{\pi}_{\crys,s_0}$. It follows from~\ref{picrys:b} that $\tilde\motL_{\crys,s_0} \otimes K_0(k)$ is determined by the projector~$\bm{\pi}_{p,s_K}$ that cuts out the $p$-adic realization on the generic fibre. As we will see in the proof of Lemma~\ref{lem:psidRext}, this is in fact true integrally.

\subsection{}
\label{subsec:Wpoint}
Let $k$ be a perfect field of characteristic $p$ that contains the residue field~$\kappa$ of~$\scrO_{E,\frp}$. Let $W = W(k)$ be its ring of Witt vectors, and let $K_0 = K_0(k)$ be the fraction field of~$W(k)$. Choose an algebraic closure $K_0 \subset \Kbar_0$ and an embedding $E_\frp \subset \Kbar_0$, let $K = E_\frp \cdot K_0(k)$ be the compositum, and let $\scrO_K \subset K$ be the ring of integers of~$K$. (Note that in our case either $K = K_0(k)$, or $K$ is a ramified quadratic extension of~$K_0(k)$.)

Consider a $W(k)$-valued point $m \in \tilde{\scrM}(W)$, and let $s = \tilde{\iota}(m) \in \tilde{\scrS}_N(W)$. The isomorphism $\psi \colon \tilde{\motP} \isomarrow \tilde{\iota}^* \tilde{\motL}$ over~$\tilde{\scrM}_E$ induces an isomorphism
\[
\psi_{\dR,m} \colon m^*\tilde{\motP}_{\dR} \otimes_W K \isomarrow s^*\tilde{\motL}_{\dR} \otimes_W K\, .
\]

\begin{lemma}
\label{lem:psidRext}
In the above situation, $\psi_{\dR,m}$ extends to an isomorphism of $\scrO_K$-modules $m^*\tilde{\motP}_{\dR} \otimes_W \scrO_K \isomarrow s^*\tilde{\motL}_{\dR} \otimes_W \scrO_K$.
\end{lemma}

\begin{proof}
The proof is an application of integral $p$-adic Hodge theory. A reference for what we need is~\cite{BMS}.

Denote by $m_K$ and~$s_K$ the $K$-valued points induced by~$m$ and~$s$, respectively, and by $m_0$ and~$s_0$ the induced $k$-valued points. We write $\tilde\motP_{p,m_K}$ for $m_K^* \tilde{\motP}_p$, viewed as a $\ZZ_p$-module with $\Gal(\Kbar/K)$-action; analogously we have $\tilde{\motL}_{p,s_K} \subset \ul\End(\tilde{\motH}_{p,s_K})$. Note that $\tilde{\motL}_{p,s_K}$ is a direct summand of $\ul\End(\tilde{\motH}_{p,s_K})$.

As in \cite{BMS} we define $\frS = W[\![T]\!]$, and we write $\varphi$ for the endomorphism of this ring that sends~$T$ to~$T^p$ and that on~$W$ is the usual Frobenius automorphism. The homomorphism $\frS \to W$ that appears in what follows is the composition $\frS \xrightarrow{~\varphi~} \frS \xrightarrow{~T\mapsto 0~} W$. Breuil--Kisin modules are defined as in~\cite{BMS}, Definition~4.1. (Here we follow the setup as in~\cite{BMS}; in particular, we fix a uniformizer and an Eisenstein polynomial as in ibid., beginning of Section~4.) Let $\Rep^{\crys}(\Gal(\Kbar/K),\ZZ_p)$ be the category of finite free $\ZZ_p$-modules~$T$ with continuous action of~$\Gal(\Kbar/K)$ such that $T \otimes \QQ_p$ is crystalline.

We use the tensor functor
\[
M \colon \Rep^{\crys}(\Gal(\Kbar/K),\ZZ_p) \to
\left(\vcenter{\setbox0=\hbox{Breuil--Kisin modules}\hbox to\wd0{\hfill finite free\hfill}\copy0} \right)
\]
that is constructed in \cite{Kisin-IntCanMod}, Theorem~(1.2.1). (In loc.\ cit.\ this functor is called~$\mathfrak{M}$; in~\cite{BMS} the notation~$M$ is used but in the key geometric example the notation $\BK$ is used; see ibid.\ Remark~5.2.) For a discussion of some basic properties of this functor, see~\cite{BMS}, Section~4. In particular, if $V$ is a crystalline representation of~$\Gal(\Kbar/K)$ and $T \subset V$ is a Galois-stable $\ZZ_p$-lattice, there is a natural isomorphism
\[
M(T) \otimes_{\frS} W[1/p] \cong D_{\crys}(V)\, .
\]

We now apply this to the study of the families of realizations~$\tilde\motP$ and~$\tilde\motL$. The $p$-adic component of the isomorphism~$\psi$ is an isomorphism of Galois representations $\psi_{p,m} \colon \tilde{\motP}_{p,m_K} \isomarrow \tilde{\motL}_{p,s_K}$. This induces an isomorphism $M\bigl(\tilde{\motP}_{p,m_K}\bigr) \otimes_{\frS} W \isomarrow M\bigl(\tilde{\motL}_{p,s_K}\bigr) \otimes_{\frS} W$. The idea of the proof, then, is that
\begin{equation}
\label{eq:MPp=Pcrys}
M\bigl(\tilde{\motP}_{p,m_K}\bigr) \otimes_{\frS} W = \tilde{\motP}_{\crys,m_0}
\end{equation}
as submodules of $M\bigl(\tilde{\motP}_{p,m_K}\bigr) \otimes_{\frS} W[1/p] = D_{\crys}(\tilde{\motP}_{p,m_K} \otimes \QQ_p) = \tilde{\motP}_{\crys,m_0}[1/p]$, and, similarly,
\begin{equation}
\label{eq:MLp=Lcrys}
M\bigl(\tilde{\motL}_{p,s_K}\bigr) \otimes_{\frS} W = \tilde{\motL}_{\crys,s_0}
\end{equation}
as submodules of $M\bigl(\tilde{\motL}_{p,s_K}\bigr) \otimes_{\frS} W[1/p] = D_{\crys}(\tilde{\motL}_{p,s_K} \otimes \QQ_p) = \tilde{\motL}_{\crys,s_0}[1/p]$. This implies the assertion of the lemma because, since $p \geq 5$ and the ramification index of~$K/K_0$ is at most~$2$, we have natural isomorphisms $m^*\tilde{\motP}_{\dR} \otimes_W \scrO_K \cong \tilde{\motP}_{\crys,m_0} \otimes_W \scrO_K$ and $s^*\tilde{\motL}_{\dR} \otimes_W \scrO_K \cong \tilde{\motL}_{\crys,s_0} \otimes_W \scrO_K$. (In this situation, $p\scrO_K \subset \scrO_K$ has a natural PD-structure. Now use Berthelot's comparison results between crystalline and de Rham cohomology; see \cite{Berth}, Chap.~V, in particular Corollaire~2.3.7 and Th\'eor\`eme~3.5.1.)

To prove \eqref{eq:MPp=Pcrys} and \eqref{eq:MLp=Lcrys} we use \cite{BMS}, Theorem~1.1(iii). (See also the discussion in loc.\ cit.\ after Remark~1.3.) Let $X$ be the Gushel--Mukai sixfold over~$W$ obtained by taking the fibre at~$m$ of the universal family over~$\tilde{\scrM}$, and let $\gamma \colon X \to \Grass = \Grass(2,V_5)$ be its morphism to the Grassmannian. Writing~$X_0$ for the special fibre of~$X$, it follows from Corollary~\ref{cor:dRCohGM} that $H_{\crys}^\bullet(X_0/W)$ is torsion-free. Similarly, it follows from Proposition~\ref{prop:CohomGr} that the analogue of Corollary~\ref{cor:dRCohGM} is true for the Grassmannian; hence also $H_{\crys}^\bullet(\Grass/W)$ is torsion-free. Because $p\geq 5$ ($p\neq 2$ here suffices) we have
\[
H^6(X_{\Kbar},\ZZ_p) = \gamma^* H^6(\Grass_{\Kbar},\ZZ_p) \oplus \tilde{\motP}_{p,m_K}\, .
\]
Applying \cite{BMS}, Theorem~1.1(iii) twice (once to~$X$, once to~$\Grass$), we obtain~\eqref{eq:MPp=Pcrys}.

Because $\tilde\motL_{p,s_K}$ is a direct summand of $\ul\End(\tilde\motH_{p,s_K})$ and $\tilde\motL_{\crys,s_0}$ is a direct summand of $\ul\End(\tilde\motH_{\crys,s_0})$, we have
\[
M\bigl(\tilde{\motL}_{p,s_K}\bigr) \otimes_{\frS} W = \Bigl[M\bigl(\tilde{\motL}_{p,s_K}\bigr) \otimes_{\frS} W[1/p]\Bigr] \cap \Bigl[M\bigl(\ul\End(\tilde{\motH}_{p,s_K})\bigr) \otimes_{\frS} W \Bigr]
\]
and
\[
\tilde\motL_{\crys,s_0} = \tilde\motL_{\crys,s_0}[1/p] \cap \ul\End\bigl(\tilde\motH_{\crys,s_0}\bigr)\, .
\]
Because $M$ is a tensor functor, it therefore suffices to show that $M\bigl(\tilde{\motH}_{p,s_K}\bigr) \otimes_{\frS} W = \tilde{\motH}_{\crys,s_0}$ as submodules of $M\bigl(\tilde{\motH}_{p,s_K}\bigr) \otimes_{\frS} W[1/p] = D_{\crys}(\tilde{\motH}_{p,s_K} \otimes \QQ_p) = \tilde{\motH}_{\crys,s_0}[1/p]$. Because $\tilde{H}$ is obtained as the cohomology in degree~$1$ of a Kuga--Satake abelian variety (which has torsion-free cohomology), this again follows from \cite{BMS}, Theorem~1.1(iii).
\end{proof}

\begin{corollary}
\label{cor:dRMatch}
The isomorphism $\psi_{\dR}$ over~$\tilde{\scrM}_E$ extends to an isomorphism of filtered flat vector bundles
\[
\psi_{\dR} \colon (\tilde{\motP}_{\dR} \otimes \scrO_{\tilde{\scrM}_\Lambda},\nabla,\Fil^\bullet) \isomarrow \tilde{\iota}^*(\tilde{\motL}_{\dR}  \otimes \scrO_{\tilde{\scrS}_{N,\Lambda}},\nabla,\Fil^\bullet)
\]
over~$\tilde{\scrM}_\Lambda$.
\end{corollary}

(Recall that $\Lambda = \scrO_{E,\frp}$.)

\begin{proof}
The argument is similar to the proof of \cite{MP-K3}, Proposition~5.11. It suffices to show that $\psi_{\dR}$ extends to an isomorphism on the underlying vector bundles; the compatibility with connections and filtrations then follows from the compatibility over~$\tilde{\scrM}_E$. Further, we may work Zariski-locally on~$\tilde\scrM$, so it suffices to consider an affine open part $U = \Spec(A) \subset \tilde\scrM$ which is faithfully flat over~$\ZZ_{(p)}$ such that the vector bundles $\tilde\motP_{\dR}$ and $\tilde\iota^* \tilde\motL_{\dR}$ are free when restricted to~$U$. Once we choose bases, the morphism~$\psi_{\dR}$ over~$U_E$ is given by an invertible matrix~$\Psi$ of size $22 \times 22$ with coefficients in $A\otimes_\QQ E$. We have to show that both~$\Psi$ and~$\Psi^{-1}$ have coefficients in $A_\Lambda = A \otimes_{\ZZ_{(p)}} \Lambda$.

Let $\varpi \in \Lambda$ be a uniformizer. Suppose some matrix coefficient $a = \Psi_{ij}$ does not lie in~$A_\Lambda$. Let $\nu >0$ be the minimal integer such that $\varpi^\nu \cdot a \in A_\Lambda$. Now choose a finite field~$k$ that contains the residue field~$\kappa$ of~$\Lambda$, and a $k$-valued point $m_0 \in U_\Lambda(k)$ that does not lie in the zero locus of~$\varpi^\nu \cdot a$. Because $\tilde\scrM$ is smooth over~$\ZZ_{(p)}$, we can lift~$m_0$ to a point $m \in U_\Lambda(\scrO_\scrE)$, where $K \supset \scrO_K$ is as in~\ref{subsec:Wpoint}. This gives a contradiction with Lemma~\ref{lem:psidRext}. The same argument shows that $\Psi^{-1}$ has coefficients in~$A_\Lambda$.
\end{proof}

\begin{corollary}
\label{cor:iotasmooth}
Assume $p \geq 5$. Then the morphism $\tilde{\iota} \colon \tilde{\scrM} \to \tilde{\scrS}_N$ is smooth.
\end{corollary}

\begin{proof}
Because $\tilde{\scrM}$ and $\tilde{\scrS}_N$ are smooth over~$\ZZ_{(p)}$, it suffices to show that if $m \in \tilde{\scrM}(k)$ for a field~$k$, the induced map on tangent spaces
\begin{equation}
\label{eq:TangentMap}
T_{\tilde{\scrM},m} \to T_{\tilde{\scrS}_N,\tilde\iota(m)}
\end{equation}
is surjective. (Use EGA~IV, Theorem~(17.11.1).) By \cite{MP-IntCanMod}, Proposition~4.16, the tangent sheaf of $\tilde{\scrS}_N$ over~$\ZZ_{(p)}$ can be identified with $\ShHom\bigl(\gr^0 \tilde{\motL}_{\dR}, \gr^{-1} \tilde{\motL}_{\dR}\bigr)$. (The graded quotients are of course taken with respect to the Hodge filtration~$\Fil^\bullet$). Let $X$ be the fibre of the universal family $\scrX/\tilde{\scrM}$ at~$m$. The tangent map~\eqref{eq:TangentMap} factors as
\[
T_{\tilde{\scrM},m} \to H^1(X,\scrT_X) \to \Hom\bigl(\gr^0 \tilde{\motP}_{\dR,m}, \gr^{-1} \tilde{\motP}_{\dR,m}\bigr)\, ;
\]
here we use Corollary~\ref{cor:dRMatch}. The first map is surjective by our choice of~$\scrM$ (if $X = \CGr(2,V_5) \cap Q$ then every infinitesimal deformation of~$X$ is obtained by varying the quadric~$Q$), and the second map is surjective by Proposition~\ref{prop:KodSpencX}.
\end{proof}

\section{Special endomorphisms and their liftings to characteristic~$0$}
\label{Sec:SpecEndoms}

The goal of this section is to review the notion of a ``special endomorphism'' that is introduced by Madapusi Pera in \cite{MP-IntCanMod}, Section~5. In what follows, $T$ is a $\ZZ_{(p)}$-scheme and $A/T$ is the abelian scheme corresponding to a $T$-valued point $t \colon T \to \tilde\scrS_N$. Our goal is to define a subspace $L(A) \subset \End(A)$ of special endomorphisms. The idea is that an endomorphism~$f$ of~$A$ is special if the induced endomorphism of~$\motH$ is a section of $\tilde\motL \subset \ul{\End}(\motH)$.

We follow \cite{MP-IntCanMod}, Section~5 rather closely, and we only summarise the main results proven there. There are three topics: (\romannumeral1)~the definition of the subspace $L(A) \subset \End(A)$, (\romannumeral2)~lifting special endomorphisms to characteristic~$0$, and (\romannumeral3)~realising Tate classes by special endomorphisms. Note that this entire section only deals with the family of realizations $\tilde\motL \subset \ul\End(\motH)$ over the Shimura variety~$\tilde\scrS_N$. In the next section we will see how these results lead to a proof of the Tate conjecture for Gushel--Mukai sixfolds.

\subsection{}
With notation as above, first assume $T$ has characteristic~$0$, i.e., $t$ factors through~$\tilde\scrS_{N,\QQ}$. Let $A/T$ be the abelian scheme corresponding to~$t$, and let $f \in\End(A)$. For $? \in \{\ell,\dR\}$, let $\motH_{?,t}$ be the corresponding realization of~$t^*\motH$, and let $f_? \colon \motH_{?,t} \to \motH_{?,t}$ be the endomorphism induced by~$f$, which gives a global section of $\ul\End(\motH_{?,t})$. Define $f$ to be ?-special if $f_?$ is a global section of $\tilde\motL_{?,t} \subset \ul\End(\motH_{?,t})$. Further, if $T = \Spec(\CC)$, we say that $f$ is Hodge-special if its Betti realization is a global section of $\tilde\motL_{\infty,t} \subset \ul\End(\motH_{\infty,t})$.

One easily shows (\cite{MP-IntCanMod}, Lemma~5.4) that being $?$-special is independent of which $? \in \{\ell,\dR\}$ we choose and is also equivalent to the condition that the restriction of~$f$ to any $\CC$-valued point of~$T$ is Hodge-special. We say that $f$ is special if these equivalent conditions are satisfied.

Specialness is invariant under deformations: if $T$ is connected and if $f$ is an endomorphism which is special when restricted to some fibre of~$A/T$, then $f$ is special.

\subsection{}
Next we consider the case when $p$ is locally nilpotent on~$T$. If $k$ is a perfect field of characteristic~$p$ and $s \in T(k)$, the restriction~$f_s$ of~$f$ to the fibre of~$A$ at~$s$ induces a global section of $\ul\End(\motH_{\crys,s})$, and we say that $f$ is special at~$s$ if this section lies in $\tilde\motL_{\crys,s} \subset \ul\End(\motH_{\crys,s})$. One then shows (\cite{MP-IntCanMod}, Lemma~5.9) that, for connected~$T$, if $f$ is special at some~$s$, it is special at every~$s$. Now define $f$ to be special if on every connected component of~$T$ it is pointwise special.

With this definition, it is not a priori clear if specialness implies $\ell$-specialness for primes $\ell \neq p$ (which is defined in the same way as above). This is true (see Corollary~\ref{cor:Uf0flat}\ref{spec=>lspec} below), but it requires some work to prove this.

As a bridge between characteristic~$0$ and characteristic~$p$, we have the following result.

\begin{proposition}
\label{prop:spec/dvr}
Let $R$ be a dvr of mixed characteristic $(0,p)$ with perfect residue field, and let $A/R$ be the abelian scheme corresponding to an $R$-valued point of~$\tilde\scrS_N$. If $f \in \End(A)$ then the induced endomorphism~$f_0$ of the special fibre of~$A$ is special if and only if the endomorphism~$f_\eta$ of the generic fibre of~$A$ is special.
\end{proposition}

\begin{proof}
Without loss of generality we may assume $R$ to be complete. The assertion then readily follows by considering the Berthelot--Ogus isomorphism between the de Rham cohomology of the generic fibre and the crystalline cohomology of the special fibre tensored with the fraction field of~$R$. Alternatively, it follows from the comparison isomorphism of $p$-adic Hodge theory; cf.\ the proof of \cite{MP-IntCanMod}, Lemma~5.13.
\end{proof}

\subsection{}
\label{subsec:A0f0}
Crucial for the proof of the Tate conjecture is a result about deformations of special endomorphisms. The situation we consider is that $k$ is a perfect field of characteristic~$p$, and $x_0 \in \tilde\scrS_N(k)$. Let $A_0$ be the corresponding abelian variety, and let $\frU$ be the formal completion of~$\tilde\scrS_N$ at~$x_0$, which is a formal scheme over~$W(k)$, non-canonically isomorphic to the formal spectrum of $W(k)[\![t_1,\ldots,t_{20}]\!]$. If $f_0$ is an endomorphism of~$A_0$, we denote by $\frU_{f_0} \subset \frU$ the formal subscheme of deformations to which the endomorphism~$f_0$ lifts. More precisely, this means that if $R$ is an artinian $W(k)$-algebra with residue field~$k$, the $R$-valued points of~$\frU_{f_0}$ are the pairs $(x,f)$ consisting of a point $x\in \tilde\scrS_N(R)$ lifting~$x_0$ and an endomorphism~$f$ of the corresponding abelian variety~$A/R$, such that $f$ reduces to~$f_0$ modulo the maximal ideal of~$R$. Note that, for a given lift~$x$, if a lifting~$f$ of~$f_0$ exists then it is unique. The forgetful morphism $(x,f) \mapsto x$ realizes~$\frU_{f_0}$ as a closed formal subscheme of~$\frU$.

\begin{theorem}[Madapusi Pera, \cite{MP-IntCanMod}, Proposition~5.21.]
In the above situation, assume $f_0$ is a special endomorphism of~$A_0$. Then the ideal defining $\frU_{f_0} \subset \frU$ is principal and $\frU_{f_0}$ is flat over~$W(k)$.
\end{theorem}

\begin{corollary}
\label{cor:Uf0flat}
Let $x_0 \in \tilde\scrS_N(k)$ and $(A_0,f_0)$ be as in~\emph{\ref{subsec:A0f0}}.
\begin{enumerate}
\item\label{lift/dvr} Assume $f_0$ is special. Then there exist a discrete valuation ring~$R$ which is finite over~$W(k)$, an abelian variety $A/R$ and an endomorphism~$f$ of~$A$, such that $(A,f) \otimes_R R/\frm_R \cong (A_0,f_0) \otimes_k R/\frm_R$.

\item\label{spec=>lspec} If $f_0$ is a special endomorphism of~$A_0$ then it is $\ell$-special for every prime number $\ell \neq p$, i.e., the induced endomorphism $f_\ell \colon \motH_{\ell,x_0} \to \motH_{\ell,x_0}$ corresponds to a global section of $\tilde\motL_{\ell,x_0} \subset \ul\End(\motH_{\ell,x_0})$.
\end{enumerate}
\end{corollary}

\begin{proof}
For \ref{lift/dvr}, the same as in the proof of \cite{Deligne-RelK3}, Corollary~1.7, gives a finite extension $W(k) \subset R$ and a formal lifting $(\hat{A},\hat{f})$ over~$\Spf(R)$. Because the Kuga--Satake abelian variety over~$\tilde\scrS_N$ admits a polarization, this formal lifting can be algebraized. Part~\ref{spec=>lspec} follows from \ref{lift/dvr} together with Proposition~\ref{prop:spec/dvr}.
\end{proof}

\subsection{}
\label{subsec:TateDef}
If $k$ is a field and $k \subset \kbar$ is an algebraic closure, we write $k^{\perf} \subset \kbar$ for the perfect closure of~$k$. If $V_\ell$ is an $\ell$-adic representation of $\Aut(\kbar/k) = \Gal(\kbar/k^{\perf})$, we write
\[
\Tate(V_\ell) \subset V_\ell
\]
for the subspace of Tate classes; by definition this means that $v \in \Tate(V_\ell)$ if and only if the stabilizer of~$v$ in~$\Gal(\kbar/k^{\perf})$ is an open subgroup.

Next suppose $k$ is perfect of characteristic~$p$ with algebraic closure~$\kbar$. We write $K_0(k)$ for the fraction field of~$W(k)$ and $\sigma$ for its Frobenius automorphism. Let $(V_p,\phi \colon V_p^{(\sigma)} \isomarrow V_p)$ be an $F$-isocrystal over~$k$; here $V_p^{(\sigma)} = V_p \otimes_{K_0(k),\sigma} K_0(k)$. The space $\Hom_\Isoc\bigl(\mathbf{1},(V_p,\phi)\bigr)$ from the unit isocrystal $\mathbf{1} = \bigl(K_0(k),\sigma\bigr)$ to~$(V_p,\phi)$ is in bijection with $V_p^{\phi=1} = \bigl\{v\in V_p\bigm| \phi(v\otimes 1) = v\bigr\}$. We could define Tate classes to be the elements in the latter space; but this would not be a natural analogue of the $\ell$-adic Tate classes because $\Hom_\Isoc\bigl(\mathbf{1},(V_p,\phi)\bigr)$ may get bigger after a finite extension of the base field. For this reason we define the $\QQ_p$-subspace $\Tate(V_p) \subset V_p \otimes_{K_0(k)} K_0(\kbar)$ by
\[
\Tate(V_p) = \bigcup_{k\subset k^\prime}\; \bigl(V_p \otimes_{K_0(k)} K_0(k^\prime)\bigr)^{(\phi\otimes \sigma) = 1}\, ,
\]
where the union is taken over all \emph{finite} field extensions $k\subset k^\prime$ inside~$\kbar$. (Caution: this is definitely not the same as the space $(V_p \otimes_{K_0(k)} K_0(\kbar))^{(\phi\otimes \sigma)=1}$, which in general is much bigger. The point is that $K_0(\kbar)$ is not the union of all~$K_0(k^\prime)$.)

\begin{remark}
\label{rem:TatepEigenval}
If $(V_p,\phi)$ is an $F$-isocrystal over a finite field~$k$ with $p^r$ elements, $\phi^r$ is a $K_0(k)$-linear automorphism of~$V_p$ and the natural map
\[
V_p^{\phi=1} \otimes_{\QQ_p} K_0(k) \to \bigl\{v\in V_p\bigm| \phi^r(v) = v\bigr\}
\]
is an isomorphism. It follows from this that the dimension of $\Tate(V_p)$ is equal to the number of eigenvalues of~$\phi^r$ (counted with multiplicities) that are roots of unities and that there exists a finite extension $k \subset k^\prime$ such that $\Tate(V_p) = \bigl(V_p \otimes_{K_0(k)} K_0(k^\prime)\bigr)^{(\phi\otimes \sigma) = 1}$.
\end{remark}

\subsection{}
\label{subsec:LtoTate}
Let $A_0/k$ be the abelian variety corresponding to a point $x_0 \in \tilde\scrS_N(k)$, where $k$ is a perfect field of characteristic~$p$. For $\ell \neq p$ we have the $\ell$-adic realization~$\tilde\motL_{\ell,x_0}$; further, we have the $F$-isocrystal $\tilde\motL_{\crys,x_0}[1/p]$. In the next results, the following condition plays a role:
\begin{equation}
\label{eq:ConditionI}
\vcenter{
\setbox0=\hbox{and equals $\dim_{\QQ_p}\bigl(\Tate(\tilde\motL_{\crys,x_0}[1/p])\bigr)$.}
\hbox to\wd0{\hfill$\dim_{\QQ_\ell}\bigl(\Tate(\tilde\motL_{\ell,x_0})\bigr)$ is independent of~$\ell$\hfill}
\copy0}
\end{equation}
As we shall see in the proof of Theorem~\ref{thm:TCGMcharp}, this condition is satisfied in the situation where we need it.

Let $L(A_{0,\kbar}) \subset \End(A_{0,\kbar})$ be the subspace of special endomorphisms of $A_0 \otimes_k \kbar$. It follows from Corollary~\ref{cor:Uf0flat}\ref{spec=>lspec} that we have natural maps
\begin{equation}
\label{eq:LtoTatel}
L(A_{0,\kbar}) \otimes \QQ_\ell \to \Tate\bigl(\tilde\motL_{\ell,x_0}\bigr)
\end{equation}
and
\begin{equation}
\label{eq:LtoTatep}
L(A_{0,\kbar}) \otimes \QQ_p \to \Tate\bigl(\tilde\motL_{\crys,x_0}[1/p]\bigr)\, ,
\end{equation}
which are injective.

A priori, the subspace $L(A_{0,\kbar})$ could be very small. One of the beautiful insights in~\cite{MP-K3} is that this does not happen, and in fact $L(A_0)$ is as big as one could hope:

\begin{theorem}[Madapusi Pera, \cite{MP-K3}, Theorem~6.4 and Corollary~6.11]
\label{thm:LtoTate}
Notation and assumptions as above.
\begin{enumerate}
\item\label{LtoTateIsomFq} Assume that $k$ is a finite field and that condition~\eqref{eq:ConditionI} is satsified. Then the maps \eqref{eq:LtoTatel} and \eqref{eq:LtoTatep} are isomorphisms.

\item\label{LtoTateIsomGen} Assume $k$ is a finitely generated field extension of~$\FF_p$ and that there exists a Zariski open $U \subset \tilde\scrS_{N,\FF_p}$ containing~$x_0$ such that condition~\eqref{eq:ConditionI} is satsified at all points of~$U$ with values in a finite field. Then the maps \eqref{eq:LtoTatel} are isomorphisms for all $\ell \neq p$.
\end{enumerate}
\end{theorem}

\subsection{}
In view of the importance of this result, we outline the proof of part~\ref{LtoTateIsomFq}, following \cite{Kisin-Points}, Section~2 and \cite{MP-K3}, Section~6. Since we assume that condition~\eqref{eq:ConditionI} is satisfied, it suffices to show that \eqref{eq:LtoTatel} is surjective for \emph{some} choice of~$\ell$.

Possibly after replacing~$k$ by a finite extension, we may assume that all endomorphisms of~$A_0$ are defined over~$k$, i.e., $\End(A_0) = \End(A_{0,\kbar})$, and that the subgroup of~$\bar{\QQ}^*$ that is generated by the eigenvalues of Frobenius is torsion-free.

The main ingredients for the argument are the following:
\begin{enumerate}[label=---]
\item An algebraic group $I$ over~$\QQ$ that acts on the space $L(A_{0,\kbar}) \otimes \QQ$;

\item an algebraic group $I_\ell$ over~$\QQ_\ell$ that acts on the space $\Tate\bigl(\tilde\motL_{\ell,x_0}\bigr)$;

\item a homomorphism $I \otimes \QQ_\ell \hookrightarrow I_\ell$ such that the map $L(A_{0,\kbar}) \otimes \QQ_\ell \to \Tate\bigl(\tilde\motL_{\ell,x_0}\bigr)$ is $I \otimes \QQ_\ell$-equivariant.
\end{enumerate}
We will give further details below. Suppose, then, that the following conditions are satisfied:
\begin{enumerate}[label=(\alph*)]
\item\label{IlRepIrr} the representation of $I_\ell$ on $\Tate\bigl(\tilde\motL_{\ell,x_0}\bigr)$ is irreducible;

\item\label{Lneq0} $L(A_{0,\kbar}) \neq 0$;

\item\label{IIlIsom} $I \otimes \QQ_\ell \hookrightarrow I_\ell$ is an isomorphism.
\end{enumerate}
Under these assumptions, it is clear that the map~\eqref{eq:LtoTatel} is surjective.

Let us now explain how the groups $I$ and~$I_\ell$ are defined. We will use notation as in Section~\ref{subsec:SystReal}, and as before we denote by~$\infty$ the unique complex embedding of~$\QQ$. Over the complex Shimura variety~$\tilde\scrS_{N,\CC}$ we have a Variation of Hodge Structure~$\motH_\infty$ and a local system of algebraic subgroups $\tilde\scrG_\infty \subset \GL(\motH_\infty)$, the fibres of which are isomorphic to the group $\tilde{G}_\QQ = \GSpin(L\otimes \QQ)$. For $\ell \neq p$ we have the $\ell$-adic sheaf~$\motH_\ell$ over~$\tilde\scrS_N$ (the $\ell$-adic component of~$\motH_\et$) and a local system of algebraic groups $\tilde\scrG_\ell \subset \GL(\motH_\ell)$, whose restriction to~$\tilde\scrS_{N,\CC}$ is isomorphic, under the comparison isomorphism~$i_{\et,\infty}$, to $\tilde\scrG_\infty \otimes \QQ_\ell$. In particular, for $x_0 \in \tilde\scrS_N(k)$ we have $\tilde\scrG_{\ell,x_0} \subset \GL(\motH_{\ell,x_0})$. There is also a $p$-adic analogue for this. A quick way to define it is to pick a lift of~$x_0$ to a point $x \in \tilde\scrS_N(W)$ (with $W = W(k)$), and identify $\motH_{\crys,x_0}[1/p]$ with $\motH_{\dR,x} =  \motH_{\infty,x} \otimes_{\QQ} K_0$ (with $K_0 = W[1/p]$); then define $\tilde\scrG_{\crys,x_0} \subset \GL\bigl(\motH_{\crys,x_0}[1/p]\bigr)$ to be $\tilde\scrG_{\infty,x} \otimes K_0$.

In order to define the group~$I$, view the group of units in the endomorphism algebra $\End^0(A_{0,\kbar}) = \End(A_{0,\kbar}) \otimes \QQ$ as an algebraic group over~$\QQ$. Then define an algebraic subgroup $I \subset \End^0(A_{0,\kbar})^*$ over~$\QQ$ by
\[
I = \Biggl\{f \in \End^0(A_{0,\kbar})^* \Biggm|
\vcenter{
\setbox0=\hbox{$\ell \neq p$, and the induced $f_{\crys} \in \GL\bigl(\motH_{\crys,x_0}[1/p]\bigr)$}
\hbox to\wd0{\hfill the induced $f_\ell \in \GL(\motH_{\ell,x_0})$ lies in $\tilde\scrG_{\ell,x_0}$ for all\hfill}
\copy0
\hbox to\wd0{\hfill lies in $\tilde\scrG_{\crys,x_0}$\hfill}
} \Biggr\}\, .
\]
Because the representation of $\tilde\scrG_{\ell,x_0}$ on $\ul\End(\motH_{\ell,x_0})$ maps the subspace $\tilde\motL_{\ell,x_0}$ into itself, and the analogous assertion is true for the crystalline component, the subspace $L(A_{0,\kbar}) \otimes \QQ \subset \End^0(A_{0,\kbar})$ spanned by the special endomorphisms is mapped into itself under left multiplication by elements in~$I$.

Next we turn to the group~$I_\ell$. Let $q$ be the cardinality of~$k$, and let $\gamma_\ell \in \tilde\scrG_{\ell,x_0}(\QQ_\ell)$ be the image of the $q$-power Frobenius element in $\Gal(\kbar/k)$. By Corollary~\ref{cor:HlGMss}, $\gamma_\ell$ is semisimple. We define $I_\ell \subset  \tilde\scrG_{\ell,x_0}(\QQ_\ell)$ to be the centralizer of~$\gamma_\ell$. Because we have made the base field~$k$ large enough, the subspace $\Tate\bigl(\tilde\motL_{\ell,x_0}\bigr) \subset \tilde\motL_{\ell,x_0}$ is simply the eigenspace of the semisimple transformation~$\gamma_\ell$ for the eigenvalue~$1$. In particular, the action of~$I_\ell$ on~$\tilde\motL_{\ell,x_0}$ maps the subspace $\Tate\bigl(\tilde\motL_{\ell,x_0}\bigr)$ into itself. As the $q$-power Frobenius endomorphism of~$A_0$ is central in~$\End(A_0)$, we have a natural inclusion $I \otimes \QQ_\ell \hookrightarrow I_\ell$.

We now consider the conditions \ref{IlRepIrr}--\ref{IIlIsom}. Condition~\ref{IIlIsom} can be dealt with by a result of Kisin (\cite{Kisin-Points}, Corollary~2.1.7), which says that there exists a prime number $\ell \neq p$ (and in fact, infinitely many such) for which $I \otimes \QQ_\ell \hookrightarrow I_\ell$ is an isomorphism. (In loc.\ cit., this is stated for the identity components, but the argument shows that after making $k$ sufficiently big the groups involved are connected.) Further, an easy lemma (\cite{MP-K3}, Lemma~6.8) shows that the representation of~$I_\ell$ on $\Tate\bigl(\tilde\motL_{\ell,x_0}\bigr)$ is irreducible, provided that $\dim_{\QQ_\ell}\bigl( \Tate(\tilde\motL_{\ell,x_0}) \bigr) \neq 2$. (This last condition is missing in loc.\ cit., but if the space of Tate classes is $2$-dimensional, the assertion of \cite{MP-K3}, Lemma~6.8 is wrong.)

At this point we are confronted with two problems: a priori it is not clear why $L(A_{0,\kbar})$ should be nonzero, and the space of Tate classes could be $2$-dimensional. We can bypass these difficulties by redoing the entire construction of the family of motives $\tilde\motL \subset \ul\End(\motH)$ over~$\tilde\scrS$ starting with a bigger lattice~$L$. (To avoid confusion, note that Theorem~\ref{thm:LtoTate} as such is unrelated to Gushel--Mukai varieties.) Thus, one chooses a positive definite integral lattice~$L^\prime$ of rank~$>2$ with $p \nmid \discr(L^\prime)$, and one redoes the entire construction with $L$ replaced by $L^\sharp = L \perp L^\prime$. Using the notation~$?^\sharp$ for objects constructed from~$L^\sharp$, we obtain a diagram of (integral canonical models of) Shimura varieties
\[
\begin{tikzcd}[column sep={{{{3em,between origins}}}}]
& \tilde{\scrS}_N^\sharp \ar{rr} && \scrA_{L^\sharp,N}\\
\tilde{\scrS}_N \ar{ur}{i}\ar{rr} && \scrA_{L,N} \ar{ur}
\end{tikzcd}
\]
Over $\tilde{\scrS}_N^\sharp$ we have an abelian scheme~$A^\sharp$ which gives rise to a system of realizations~$\motH^\sharp$, and we have a submotive $\motL^\sharp \subset \ul\End(\motH^\sharp)$. The pullback of~$A^\sharp$ to~$\tilde\scrS_N$ (via~$i$) is isomorphic to $A \otimes C(L^\prime)$, where of course $A$ is the original (``Kuga--Satake'') abelian scheme over~$\tilde\scrS_N$, and where $A \otimes C(L^\prime)$ is Serre's tensor construction (if $r^\prime$ is the rank of the lattice~$L^\prime$ this just means that we take the sum of $2^{r^\prime}$ copies of~$A$). Correspondingly, $i^*\motH^\sharp \cong \motH \otimes C(L^\prime)$, so that $i^*\bigl(\ul\End(\motH^\sharp)\bigr) \cong \ul\End(\motH) \otimes \End\bigl(C(L^\prime)\bigr)$. In this description, $i^*(\motL^\sharp)$ becomes the submotive
\[
(\motL \otimes \ZZ) \oplus (\unitmot \otimes L^\prime)
\]
where in the second term the unit motive $\unitmot$ should be interpreted as the submotive $\unitmot \cdot \id_{\motH} \subset \ul\End(\motH)$ spanned by the identity endomorphism, and where $L^\prime$ is identified with a subspace of~$\End\bigl(C(L^\prime)\bigr)$ via left multiplication in the Clifford algebra. We now take sections at our point $x_0 \in \tilde\scrS_N(k)$ (as in~\ref{subsec:LtoTate}), and we note that the summand $(\unitmot \otimes L^\prime)$ consists of Tate classes. As explained, grace to Kisin's result that allows us to deal with condition~\ref{IIlIsom}, this puts us in a situation where conditions \ref{IlRepIrr}--\ref{IIlIsom} are satisfied. The conclusion, then, is that the map $L(A^\sharp_{0,\kbar}) \otimes \QQ_\ell \to \Tate\bigl(\tilde\motL^\sharp_{\ell,x_0}\bigr)$ is surjective, and this implies that \eqref{eq:LtoTatel} is surjective.

\section{The Tate Conjecture for Gushel--Mukai sixfolds in characteristic $p \geq 5$}
\label{sec:TCGM6}

We can now put everything together. Let $Y/k$ be a GM sixfold as in the statement of Theorem~\ref{thm:TCGMcharp}. Note that, by our definition of the space of Tate classes, $\Tate\bigl(H^{2i}_{\crys}(Y/K_0(k))(i)\bigr)$ is a subspace of $H^{2i}_{\crys}\bigl(Y_{\kbar}/K_0(\kbar)(i)\bigr)$; see Section~\ref{subsec:TateDef}.

\begin{proof}[Proof of \emph{Theorem~\ref{thm:TCGMcharp}} for $\dim(Y) =6$.]
The first assertion of~\ref{TateGMl} has already been proven in Corollary~\ref{cor:HlGMss}. For the proof of the remaining assertions, it suffices to consider the case $i=3$.

We start with a general remark. Let $k$ be a field of characteristic~$p$, and suppose we have a $k$-valued point $m_0 \in \tilde\scrM(k)$ which maps to $s_0 = \tilde\iota(m_0) \in \tilde\scrS_N(k)$. Let $X_0/k$ be the Gushel--Mukai sixfold corresponding to~$m_0$ and $A_0/k$ the Kuga--Satake abelian variety. For $\ell \neq p$ we have an isomorphism $\tilde\motP_{\ell,m_0} \isomarrow \tilde\motL_{\ell,s_0}$, and if $k$ is perfect we have $\tilde\motP_{\crys,m_0} \isomarrow \tilde\motL_{\crys,s_0}$. Further, $H^6\bigl(X_{0,\kbar},\QQ_\ell(3)\bigr) \cong \QQ_\ell^{\oplus 2} \oplus \tilde\motP_{\ell,m_0}$ as Galois representations and if $k$ is perfect then $H^6_{\crys}(X_0/K_0(k))\bigl(3\bigr) \cong \QQ_p^{\oplus 2} \oplus \tilde\motP_{\crys,m_0}$ as $F$-isocrystals over~$k$. If $k$ is a finite field then it follows from the Weil conjectures together with Corollary~\ref{cor:HlGMss} that the $\QQ_\ell$-dimension of the space $\Tate\bigl(H^6(X_{0,\kbar},\QQ_\ell(3))\bigr)$ is independent of~$\ell \neq p$. (This dimension can be read from the zeta function, which does not depend on~$\ell$.) Moreover, by \cite{KatzMessing}, Theorem~1, together with Remark~\ref{rem:TatepEigenval}, this dimension equals the $\QQ_p$-dimension of $\Tate\bigl(H^6_{\crys}\bigl(X_0/K_0(k)\bigr)(3)\bigr)$. Because $\tilde\iota \colon \tilde\scrM \to \tilde\scrS_N$ is smooth (Corollary~\ref{cor:iotasmooth}), we find that the condition in Theorem~\ref{thm:LtoTate}\ref{LtoTateIsomGen} is satisfied, and it follows that whenever $k$ is finitely generated over~$\FF_p$, the maps $L(A_{0,\kbar}) \otimes \QQ_\ell \to \Tate(\tilde\motL_{\ell,s_0})$ are isomorphisms.

We now prove the second assertion of~\ref{TateGMl}. In doing so, we may replace~$k$ by a finitely generated extension. Choose $m_0$ as above such that $Y \cong X_0$, and let $\xi_0$ be a Tate class in $\tilde\motP_{\ell,m_0}$. By the above, it suffices to prove that $\xi_0$ is algebraic if under the isomorphism $\tilde\motP_{\ell,m_0} \isomarrow \tilde\motL_{\ell,s_0}$ it maps to the $\ell$-adic realization of a special endomorphism~$f_0$ of~$A_0$. Assume this is the case. By Corollary~\ref{cor:Uf0flat}\ref{lift/dvr} there exists a finite extension of dvr $W(k^{\perf}) \subset R$ and an $R$-valued point $s \in \tilde\scrS_N(R)$ lifting~$s_0$ such that $f_0$ lifts to an endomorphism~$f$ of the corresponding abelian variety~$A/R$. By Corollary~\ref{cor:iotasmooth}, there exists $m \in \tilde\scrM(R)$ lifting~$m_0$ such that $s = \tilde\iota(m)$. Let $X/R$ be the Gushel--Mukai variety corresponding to~$m$. Since we may replace~$R$ by a finite extension, we may assume that $R$ contains the ring~$\hat\scrO_{E,\frp}$ as in Section~\ref{subsec:PAdicNot}. Let $K$ be the fraction field of~$R$, let $X_\eta$ and~$A_\eta$ be the generic fibres of~$X$ and~$A$, and let $f_\eta \in L(A_\eta)$ be the restriction of~$f$ to the generic fibre. By Proposition~\ref{prop:descendQ}\ref{psi/E} we have an isomorphism $\tilde\motP_{m_\eta} \isomarrow \tilde\motL_{s_\eta}$ in $\Real(K;\QQ)$. By construction, the endomorphism~$f_\eta$ defines an absolute Hodge class in~$\tilde\motL_{s_\eta}$. Let $\xi_\eta$ be the corresponding absolute Hodge class in~$\tilde\motP_{m_\eta}$. Because the Hodge conjecture is true for Gushel--Mukai sixfolds (see \cite{FuMoonen-GM1}, Section~5), $\xi_\eta$ is algebraic, i.e., it is the image of a class in $\CH^3(X_\eta) \otimes \QQ$. By specialization it follows that $\xi_0$ is algebraic, and we are done.

If $k$ is finite, the same argument (now starting with a Tate class in~$\tilde\motP_{\crys,m_0}$) gives part~\ref{TateGMp} of the theorem.
\end{proof}

\section{The Tate conjecture for Gushel--Mukai fourfolds in characteristic $p\geq 5$}
\label{sec:TCGM4}

In this section we prove the Tate conjecture for Gushel--Mukai fourfolds in characteristic $p\geq 5$, by reducing it to the case of Gushel--Mukai sixfolds. To this end, we show that for a GM fourfold~$X$ there exists a ``partner''~$X^\prime$, which is a GM sixfold, such that $X$ and~$X^\prime$ have isomorphic Chow motives in middle degree. The construction uses liftings to characteristic~$0$ together with a result from our paper~\cite{FuMoonen-GM1} (see Theorem~\ref{thm:genpartners}).

\subsection{}
Let $R$ be a ring which is either a field or a dvr. If $V$ is a free $R$-module of finite rank and $V^\vee = \Hom(V,R)$ is its dual, $(\wedge^i V)^\vee$ is canonically isomorphic to $\wedge^i (V^\vee)$ (see for instance~\cite{Bourbaki-Algebre}, Chap.~3, \S~11, no.~5). The notation~$\wedge^i V^\vee$ will therefore not cause any ambiguity. 

For symmetric powers one has to be more careful. We shall assume that $2$ is invertible in~$R$, and we only need to consider symmetric squares. In this situation, we can use that $\Sym^2(V)^\vee$ is canonically isomorphic to the subspace $\ST^2(V^\vee) \subset (V^\vee)^{\otimes 2}$ of symmetrical tensors (i.e., the subspace of $\frS_2$-invariants in~$(V^\vee)^{\otimes}$) and that the composition $\ST^2(V^\vee) \hookrightarrow (V^\vee)^{\otimes 2} \twoheadrightarrow \Sym^2(V^\vee)$ is an isomorphism because $2 \in R^*$. We use this to identify $\Sym^2(V)^\vee$ with~$\Sym^2(V^\vee)$.

\subsection{}
\label{subsec:Data}
An important tool in the study of GM varieties is a correspondence between GM data sets and Lagrangian data sets, which first appeared in work of O'Grady \cite{OGrady-EPW-Duke} and Iliev--Manivel  \cite{IlievManivel11}, and which was refined by Debarre--Kuznetsov in~\cite{DK-GMClassif}. We shall work with a version of this correspondence that is adapted to our needs. 

Throughout the discussion, $R$ is a ring which is either a field or a dvr, such that $2 \in R^*$.
\begin{enumerate}
\item A \emph{GM datum} of dimension $n\in \{3, 4, 5\}$ over~$R$ is defined to be a tuple $(V_6,V_5,W,q,\epsilon)$, where 
  \begin{itemize}
  \item $V_6$ is a free $R$-module of rank~$6$,
  \item $V_5\subset V_6$ is a direct summand of rank~$5$,
  \item $W\subset \wedge^2 V_5$ is a direct summand of rank~$(n+5)$,
  \item $q \colon V_6 \to \Sym^2(W^\vee)$ is an $R$-linear map, 
  \item $\epsilon\colon \det(V_5) \isomarrow R$ is an isomorphism
  \end{itemize}
such that
\begin{equation}
\label{eq:CompatibilityGMdata}
q(v)\bigl(w \cdot w^\prime\bigr) = \epsilon(v\wedge w\wedge w'), \text{ for all $v\in V_5$ and $w, w^\prime \in W$.}
\end{equation}
(Here we use the isomorphism $\Sym^2(W^\vee) \cong \Sym^2(W)^\vee$.)
	
\item A \emph{Lagrangian datum} of dimension $n\in \{3, 4, 5\}$ over~$R$ is a tuple $(V_6,V_5,A,\epsilon)$, where
  \begin{itemize}
  \item $V_6$ is a free $R$-module of rank~$6$,
  \item $V_5\subset V_6$ is a direct summand of rank~$5$,
  \item $\epsilon\colon \det(V_5) \isomarrow R$ is an isomorphism, 
  \item $A\subset \wedge^3 V_6$ is a Lagrangian submodule with respect to the $\det(V_6)$-valued symplectic form $\wedge\colon \wedge^3 V_6\otimes \wedge^3 V_6 \to \wedge^6 V_6 = \det(V_6)$, such that $A\cap \wedge^3 V_5$ is free of rank~$5-n$.
  \end{itemize}
Note that $A$ is assumed to be a direct summand of~$\wedge^3 V_6$ (necessarily of rank~$10$) and that $A\cap \wedge^3 V_5$ is then automatically a direct summand of both~$A$ and~$\wedge^3 V_5$.
\end{enumerate}

We trust it is clear what it means for two GM data (resp.\ Lagrangian data) to be isomorphic.

\begin{proposition}[GM--Lagrangian correspondence]
\label{prop:GMLagCorrespondence}
Let $R$ be a ring which is either a field or a dvr, with $2 \in R^*$, and let $n\in \{3,4,5\}$. There is a natural bijection between the set of isomorphism classes of $n$-dimensional GM data and the set of isomorphism classes of $n$-dimensional Lagrangian data.
\end{proposition}

When $R$ is a field of characteristic~$0$, this is a special case of the more elaborate version in \cite{DK-GMClassif}, Theorem~3.6. See \cite{Debarre-GMSurvey}, Theorem~2.2, for a proof in a simplified setting that is closer to the version we consider here. The proofs in loc.\ cit.\ also work in characteristic $p \neq 2$. For the reader's convenience we recall how the correspondence works.

\begin{proof}
Given a GM datum $(V_6,V_5,W,q,\epsilon)$, we want to construct a Lagrangian submodule $A \subset \wedge^3V_6$. Choose an element $v_0 \in V_6$ whose class generates~$V_6/V_5$ and define
\begin{equation}
\label{eq:ADef}
A := \Ker\Bigl(\wedge^3 V_5 \oplus (v_0\wedge W) \to W^\vee\Bigr)\, ,	
\end{equation}
where $\wedge^3 V_5\oplus (v_0\wedge W)$ is viewed as a submodule of~$\wedge^3 V_6$ and the map sends $(\xi, v_0\wedge w)$ to the linear functional $w^\prime \mapsto \epsilon(\xi\wedge w^\prime) + q(v_0)\bigl(w\cdot w^\prime\bigr)$. Using the relation \eqref{eq:CompatibilityGMdata} one shows that the submodule $A \subset \wedge^3 V_6$ thus obtained is independent of the choice of~$v_0$, and that $A$ is isotropic. As the map $\wedge^3 V_5 \to W^\vee$ that is used in~\eqref{eq:ADef} is surjective, $A$ is a direct summand of~$\wedge^3 V_6$ of rank equal to $\rk_R(\wedge^3 V_5) + \rk_R(W) - \rk_R(W) = 10 = \frac{1}{2} \rk_R(\wedge^3 V_6)$, so $A$ is  Lagrangian. Further, $\rk_R(A\cap \wedge^3 V_5) = \rk_R(\wedge^3 V_5) - \rk_R(W) = 10 - (n+5) = 5-n$.

In the opposite direction, given a Lagrangian datum $(V_6, V_5, A, \epsilon)$, we want to construct~$W$ and~$q$. Choose a primitive element $0\neq \lambda \in V_6^\vee$ with $\Ker(\lambda) = V_5$. Let $x\mapsto \lambda \iprod x \colon \wedge^\bullet V_6 \to \wedge^{\bullet-1} V_6$ be the contraction (=interior product) defined by~$\lambda$. The kernel of this map is $\wedge^\bullet V_5$ and the image is $\wedge^{\bullet -1} V_5 \subset \wedge^{\bullet-1} V_6$; in particular we have a short exact sequence
\begin{equation}
\label{eq:w3w3w2seq}
0 \tto \wedge^3 V_5 \tto \wedge^3 V_6 \xrightarrow{~ \lambda\, \iprod~} \wedge^2 V_5 \tto 0\, .
\end{equation}
Define 
\[
W := \Image\bigl(A \hookrightarrow \wedge^3 V_6  \xrightarrow{~  \lambda\, \iprod~} \wedge^2 V_5\bigr)\, ;
\]
then $W$ is a direct summand of~$\wedge^2 V_5$ and because $W \cong A/(A\cap \wedge^3 V_5)$ it is free of rank $10 - \rk_R(A\cap \wedge^3 V_5) = n+5$. For $\xi$, $\xi^\prime \in A$ we have $\xi \wedge \xi^\prime = 0$, and because $\lambda\, \iprod$ is a derivation in the graded sense (see \cite{Bourbaki-Algebre}, Chap.~3, \S~11, nos.~6--9) this gives 
\[
0 = \lambda \iprod (\xi \wedge \xi^\prime) = (\lambda \iprod \xi) \wedge \xi^\prime - \xi \wedge (\lambda \iprod \xi^\prime)\, ,
\]
so $\xi \wedge (\lambda \iprod \xi^\prime) = (\lambda \iprod \xi) \wedge \xi^\prime = \xi^\prime \wedge (\lambda \iprod \xi)$. Therefore, we obtain a well-defined homomorphism $\tilde{q} \colon V_6 \otimes \Sym^2(A) \to R$ given by
\[
\tilde{q}\bigl(v \otimes (\xi\cdot \xi^\prime)\bigr) = \epsilon\Bigl(v \wedge (\lambda \iprod \xi) \wedge (\lambda \iprod \xi^\prime) - \lambda(v) \cdot \xi \wedge (\lambda \iprod \xi^\prime)\Bigr)\, .
\]
As $\lambda \, \iprod$ is zero on~$\wedge^3 V_5$, this descends to a homomorphism $V_6 \otimes \Sym^2(W) \to R$. Let $q \colon V_6 \to \Sym^2(W^\vee)$ be the induced map. If $v \in V_5$ then $\lambda(v) = 0$, so it follows that $q(v)\bigl(w\cdot w^\prime\bigr) = \epsilon(v\wedge w \wedge w^\prime)$ (write $w = \lambda \iprod \xi$ and $w^\prime = \lambda \iprod \xi^\prime$). This shows that $(V_6, V_5, W, q, \epsilon)$ is a  GM datum. Note that $\lambda$ is unique only up to a scalar in~$R^\times$, but a routine calculation shows that $W$ and~$q$ do not change under such a rescaling. 

One checks without difficulty that the above two constructions are inverse to each other.
\end{proof}

\subsection{}
Let $(V_6,V_5,W,q,\epsilon)$ be an $n$-dimensional GM datum over a ring~$R$ as in~\ref{subsec:Data}, with $n \in \{3,4,5\}$. We may then define $X \subset \PP(W)$ over~$\Spec(R)$ by
\begin{equation}
\label{eq:GMIntersection}
X := \bigcap_{v\in V_6} Q(v)\, ,
\end{equation}
where $Q(v)$ denotes the quadric in~$\PP(W)$ defined by~$q(v)$. We refer to~$X$ as the GM intersection defined by the GM datum.

It follows from \eqref{eq:CompatibilityGMdata} that $\bigl(\PP(W) \cap \Grass(2,V_5)\bigr) \subset \PP(W)$ is cut out by the quadrics~$Q(v)$ for $v \in V_5$. Therefore, if $v_0 \in V_6$ is an element whose class generates $V_6/V_5$, we have
\[
X = \PP(W) \cap \Grass(2,V_5) \cap Q(v_0)\, .
\]

If $k$ is a field of characteristic $\neq 2$, every GM variety of Mukai type (as in Definition~\ref{def:GMVariety}) can be realised as a GM intersection~\eqref{eq:GMIntersection} over a finite extension of~$k$.

A GM intersection as above is not necessarily smooth over~$R$, but when it is so, $X$ is an $n$-dimensional GM variety of Mukai type over~$R$. Moreover, if $R=k$ is a field, the following proposition, due to Debarre and Kuznetsov, gives a useful criterion for when $X$ is smooth over~$k$. In characteristic~$0$, this is a special case of \cite{DK-GMClassif}, Theorem~3.16.

\begin{proposition}
\label{prop:smoothcrit}
Let $(V_6,V_5,W,q,\epsilon)$ be an $n$-dimensional GM datum over a field~$k$ with $\charact(k) \neq 2$. Let $(V_6,V_5,A,\epsilon)$ be the corresponding Lagrangian datum over~$k$ and $X/k$ the GM intersection as in \eqref{eq:GMIntersection}. Then $X$ is smooth of dimension~$n$ over~$k$ if and only if $A_{\kbar} \subset \wedge^3 V_{6,\kbar}$ contains no decomposable vectors (i.e., $\PP(A_{\kbar}) \cap \Grass(3,V_{6,\kbar}) = \emptyset$).
\end{proposition}

\begin{proof}
Of course, $X/k$ is smooth if and only if $X_{\kbar}/\kbar$ is smooth; so we may assume $k = \kbar$. In this situation the same proof as in \cite{Debarre-GMSurvey}, Section~2, gives the result. (This proof uses \cite{PvdV}, Corollary~1.6, which is valid over an arbitrary field of characteristic~$\neq 2$.)
\end{proof}

\subsection{}
\label{subsec:GMLagchar0}
Let $K$ be a field of characteristic~$0$. In \cite{DK-GMClassif} it is explained that a GM variety~$X$ over~$K$ gives rise to a ``GM data set'' 
\[
\bigl(V_6(X),V_5(X),W(X),L(X),\mu_X,q_X,\epsilon_X \bigr)\, ,
\] 
and a corresponding ``Lagrangian data set'' $\bigl(V_6(X),V_5(X),A(X)\bigr)$. We shall only use this construction over fields of characteristic~$0$. We do use that if $X/K$ (still with $\charact(K) =0$) is obtained as a GM intersection~\eqref{eq:GMIntersection} given by a GM datum $(V_6,V_5,W,q,\epsilon)$, and if $(V_6,V_5,A,\epsilon)$ is the corresponding Lagrangian datum obtained as described in (the proof of) Proposition~\ref{prop:GMLagCorrespondence}, we have $\bigl(V_6(X),V_5(X),A(X)\bigr) \cong (V_6,V_5,A)$.

A configuration that arises in the study of GM varieties is that an $n$-dimensional GM variety~$X$ of Gushel type can be described as a double cover of another nonsingular variety $M_X^\prime \subset \Grass(2,V_5)$ (a linear section of the Grassmannian), with branch locus a GM variety $Y \subset M_X^\prime$ of Mukai type of dimension $n-1$; see~\cite{DK-Kyoto}, Section~2.1. In this situation, $X$ and~$Y$ are said to be ``opposite'' GM varieties. It follows from the results in Section~3 of~\cite{DK-GMClassif} that $\bigl(V_6(X),V_5(X),A(X)\bigr) \cong \bigl(V_6(Y),V_5(Y),A(Y)\bigr)$.
\bigskip

We shall make use of the following result of O'Grady.

\begin{proposition}
\label{prop:V5prime}
Let $k$ be an algebraically closed field of characteristic~$\neq 2$. Let $(V_6,V_5,A,\epsilon)$ be a Lagrangian datum over~$k$, and assume $A$ contains no decomposable vectors. Then there exists a $5$-dimensional subspace $V_5^\prime \subset V_6$ such that $A \cap \wedge^3 V_5^\prime = 0$.
\end{proposition}

\begin{proof}
Define $A^\perp \subset \wedge^3 V_6^\vee$ by $A^\perp = \bigl\{\xi \in \wedge^3 V_6^\vee \bigm| \xi(A) = 0\bigr\}$; this is a Lagrangian subspace of~$\wedge^3 V_6^\vee$. First we note that $A^\perp$ contains no decomposable vectors. Indeed if $A^\perp$ contains a decomposable vector, we can choose a basis $e_1,\ldots,e_6$ for~$V_6$, with dual basis $\check{e}_1,\ldots,\check{e}_6$ for~$V_6^\vee$, such that $\check{e}_{123} = \check{e}_1 \wedge \check{e}_2 \wedge \check{e}_3 \in A^\perp$; because $A^\perp$ is Lagrangian this implies that the $\check{e}_{456}$-coefficient of every vector $\alpha \in A^\perp$ is zero; hence the decomposable vector $e_{456} = e_4 \wedge e_5 \wedge e_6$ lies in~$A$, contradicting the assumptions.

Let $U \subset V_6$ be a $5$-dimensional subspace and let $[u] \in \PP(V_6^\vee)$ be the corresponding line in~$V_6^\vee$. Then $\wedge^3 U \subset \wedge^3 V_6$ is a Lagrangian subspace and
\[
\bigl(\wedge^3 U \bigr)^\perp = u \wedge \bigl(\wedge^2 V_6^\vee\bigr) = F_u\, ,\qquad \text{where}\quad 
F_u := \bigl\{\xi \in \wedge^3 V_6^\vee \bigm| \xi \wedge u = 0\bigr\}\, .
\]
Because 
\[
A \cap \wedge^3 U = 0 \iff A^\perp + F_u = \wedge^3 V_6^\vee \iff A^\perp \cap F_u = 0\, ,
\]
the proposition is equivalent to the assertion that there exists a $u \in \PP(V_6^\vee)$ such that $A^\perp \cap F_u = 0$. This is proven by O'Grady in \cite{OGrady-EPW-taxonomy}, Corollary~2.5; the proof that is given there works over an arbitrary algebraically closed field~$k$ with $\charact(k) \neq 2$. (The assumption that $k = \kbar$ is needed because the proof uses a Bertini argument.)
\end{proof}

\subsection{}
If $k$ is a field, we denote by $\CHM(k)$ the category of Chow motives over~$k$ with rational coefficients. In what follows, the notation $\CH$ is used for Chow groups with $\QQ$-coefficients.
 
If $X/k$ is smooth projective we write $\frh(X)$ for its Chow motive. (Let us specify that we use the cohomological version of the theory, i.e., the functor~$\frh$ that we use is the contravariant one.) Further, if $M$ is a Chow motive we have Tate twists $M(n) = M \otimes \unitmot(n)$, for $n \in \ZZ$.

Let now $X$ be a GM variety of dimension~$n$ over~$k$. Then $X$ admits a natural Chow--K\"unneth decomposition, that is, a decomposition $[\Delta_X] = \sum_{i=0}^{2n}\, \pi^i_X$ of the diagonal class as a sum of mutually orthogonal projectors $\pi^i_X \in \CH^n(X \times_k X)$, such that for any classical Weil cohomology theory~$H^*$, the endomorphism of~$H^*(X)$ induced by~$\pi^i_X$ is the projector onto the summand $H^i(X)$ in degree~$i$. These projectors can be given explictly; we here recall the formulas in the two cases that are most relevant for us, namely GM varieties of dimension~$4$ or~$6$; see \cite{FuMoonen-GM1}, Section~7, for more details. In both cases, all odd projectors are zero: $\pi_X^1 = \pi_X^3 = \cdots = \pi_X^{2n-1} = 0$; so it only remains to describe the even projectors.

\begin{itemize}
\item Let $X/k$ be a GM fourfold. Let $H = -\frac{1}{2} \cdot [K_X] \in \CH^1(X)$ be the class of an ample generator. Then $\pt = \frac{1}{10} \cdot H^4 \in \CH^4$ is the class of a point. (The given formulas show that these classes make sense over any field, even if for instance $X$ has no $k$-rational point.) The even Chow--K\"unneth projectors are given by
\begin{align*}
\pi_X^0 &= \pt \times X  & \pi_X^8 &= X \times \pt\\
\pi_X^2 &= \tfrac{1}{10} \cdot H^3 \times H  & \pi_X^6 &= \tfrac{1}{10} \cdot H \times H^3
\end{align*}
and
\[
\pi_X^4 = [\Delta_X] - \pi_X^0 - \pi_X^2 - \pi_X^6 - \pi_X^8\, .
\]

\item Next consider a GM sixfold~$X$. For simplicity of exposition, assume the Gushel morphism $\gamma \colon X \to \Grass(2,V_5)$ is defined over~$k$. (The projectors we shall describe can in fact be defined over~$k$ without this assumption, but we shall not need this.) Let $\scrU_X$ be the Gushel bundle on~$X$, i.e., the pullback under~$\gamma$ of the tautological rank~$2$ bundle on~$\Grass(2,V_5)$. Let $\sigma_{i,j} \in \CH^{i+j}\bigl(\Grass(2,V_5)\bigr)$ (for $3\geq i\geq j\geq 0$) be the Schubert classes. Similar to the previous case we have the class $H = -\frac{1}{4} \cdot [K_X] = \gamma^*(\sigma_{1,0}) \in \CH^1(X)$ of an ample generator and the class $\pt = \frac{1}{10} \cdot H^6 = \frac{1}{2}\cdot \gamma^*(\sigma_{3,3}) \in \CH^6(X)$ of a point. Then
\[
e_1 = H^2 \, ,\qquad e_2 = \gamma^*(\sigma_{1,1}) = c_2(\scrU_X)
\]
are in $\CH^2(X)$. The classes $f_1$, $f_2 \in \CH^4(X)$ given by
\[
f_1 = \tfrac{1}{2}\cdot H^4 - H^2 e_2\, ,\qquad f_2 = -H^4 + \tfrac{5}{2}\cdot H^2 e_2
\]
have the property that $\int_X e_i \cdot f_j = \delta_{i,j}$. The even Chow--K\"unneth projectors of~$X$ are then given by
\begin{align*}
\pi_X^0 &= \pt \times X  & \pi_X^{12} &= X \times \pt\\
\pi_X^2 &= \tfrac{1}{10} \cdot H^5 \times H  & \pi_X^{10} &= \tfrac{1}{10} \cdot H \times H^5\\
\pi_X^4 &= f_1 \times e_1 + f_2 \times e_2 & \pi_X^8 &= e_1 \times f_1 + e_2 \times f_2
\end{align*}
and
\[
\pi_X^6= [\Delta_X] - \pi_X^0 - \pi_X^2 - \pi_X^4  - \pi_X^8 - \pi_X^{10}-\pi_X^{12}\, .
\]
\end{itemize}
\bigskip

Theorem~\ref{thm:TCGMcharp} for GM fourfolds follows from the case of sixfolds together with the following result. (Note that it suffices to prove Theorem~\ref{thm:TCGMcharp} after replacing the base field by a finite extension.)

\begin{theorem}
\label{thm:MotivePartner}
Let $k$ be a field of characteristic $p>2$. Let $X/k$ be a GM fourfold. Then there exists a finite field extension $k \subset k^\prime$ and a GM sixfold $X^\prime/k^\prime$ such that
\begin{equation}
\label{eq:h4h6isom}
\frh^4(X_{k^\prime})\bigl(2\bigr) \cong \frh^6(X^\prime)\bigl(3\bigr)
\end{equation}
in $\CHM(k^\prime)$.
\end{theorem}

The proof of this result will take up the rest of this section.

\subsection{}
In our proof of Theorem~\ref{thm:MotivePartner}, the GM sixfold~$X^\prime$ will be constructed in terms of Lagrangian data. To obtain the desired relation~\eqref{eq:h4h6isom}, we will construct lifts of $X$ and~$X^\prime$ to characteristic~$0$ and use a result that was proven in our paper~\cite{FuMoonen-GM1}, which gives the analogue of~\eqref{eq:h4h6isom} in characteristic~$0$. The result in characteristic~$p$ will then be deduced by a specialisation argument. 

Before stating the result from~\cite{FuMoonen-GM1} that we need, let us first recall the notion of generalised partners, following~\cite{KuzPerry}, Definition~3.5. For this, consider GM varieties $X$ and~$X^\prime$ over a field~$K$ of characteristic~$0$, of dimensions $n$ and~$n^\prime$, respectively, such that $n$, $n^\prime \in \{3,4,5,6\}$ have the same parity. Then we say that $X$ and~$X^\prime$ are \emph{generalised partners} if there exists an isomorphism $f \colon V_6(X) \isomarrow V_6(X^\prime)$ such that $\wedge^3 f \colon \wedge^3 V_6(X) \isomarrow \wedge^3 V_6(X^\prime)$ restricts to an isomorphism $A(X) \isomarrow A(X^\prime)$.

\begin{theorem}
\label{thm:genpartners}
Let $X$ and~$X^\prime$ be generalised partners over an algebraically closed field~$K$ of characteristic~$0$, of dimensions $n$ and~$n^\prime$, respectively. Then there exists an isomorphism
\[
\frh^n(X) \cong \frh^{n^\prime}(X^\prime)\bigl(\tfrac{n^\prime-n}{2}\bigr)
\]
in $\CHM(K)$.
\end{theorem}

See~\cite{FuMoonen-GM1}, Section~8. (There is also a version of this result for so-called generalised dual GM varieties, but we shall not need it here.)

\begin{lemma}
\label{lem:LagrLift}
Let $k$ be a field of characteristic $p>2$. Let $(V_6,V_5,A,\epsilon)$ be a Lagrangian datum of dimension $n \in \{3,4,5\}$ over~$k$. Then there exists a dvr~$R$ with residue field~$k$ and fraction field~$K$ of characteristic~$0$, and a Lagrangian datum $(\scrV_6,\scrV_5,\scrA,\epsilon)$ over~$R$ whose reduction modulo the maximal ideal $\frm \subset R$ is isomorphic to $(V_6,V_5,A,\epsilon)$.
\end{lemma}

\begin{proof}
Let $R$ be a Cohen ring for~$k$. Let $\scrV_6 = R^6$ and fix an isomorphism $\scrV_6 \otimes_R k \isomarrow V_6$, via which we identify $\scrV_6 \otimes_R k$ and~$V_6$. Choose any $\epsilon \colon \det(\scrV_6) \isomarrow R$ that lifts the given isomorphism $\det(V_6) \isomarrow k$. Choose a lifting of~$V_5$ to a direct summand $\scrV_5 \subset \scrV_6$. By assumption, $L = A\cap \wedge^3 V_5 \subset \wedge^3 V_6$ is $(5-n)$-dimensional. Choose a direct summand $\scrL \subset \wedge^3 \scrV_5$ lifting~$L$. Let $\scrL^\perp \subset \wedge^3 \scrV_6$ be the orthogonal complement of~$\scrL$ with respect to the symplectic form $\psi \colon \wedge^3 \scrV_6 \times \wedge^3 \scrV_6 \to \det(\scrV_6)$. Then $\scrT = \scrL^\perp/\scrL$ is free of rank~$10+2n$ over~$R$ and $\psi$ induces a symplectic form~$\bar\psi$ on~$\scrT$. Writing $T = \scrT \otimes_R k = L^\perp/L$, the subspace $\bar{A} = A/L \subset T$ is Lagrangian with respect to the pairing~$\bar\psi$. Because the scheme of Lagrangian subspaces in~$\scrT$ is smooth over~$R$ we can lift~$\bar{A}$ to a Lagrangian submodule $\overline{\scrA} \subset \scrT$. Define $\scrA \subset \scrL^\perp \subset \wedge^3 \scrV_6$ to be the inverse image of~$\overline{\scrA}$. By construction, the intersection of~$\bar{A}$ and the image of~$\wedge^3 V_5$ in~$T$ is zero; this implies that $\scrA \cap \wedge^3 \scrV_5 = \scrL$, and therefore $(\scrV_6,\scrV_5,\scrA,\epsilon)$ is a Lagrangian datum of dimension~$n$ over~$R$.
\end{proof}

\subsection{}
We are now ready to prove Theorem~\ref{thm:MotivePartner}, thereby completing the proof of Theorem~\ref{thm:TCGMcharp}. There are several steps in the proof where we may need to replace the base field by a finite extension. In each such step, the notation that was already in place will be retained and will from then on refer to the varieties that are obtained by extension of scalars to the bigger field. 

We first consider a GM fourfold~$X$ of Mukai type. Possibly after replacing~$k$ by a finite extension, there exists a $4$-dimensional GM datum over~$k$ such that $X$ is isomorphic to the associated GM intersection as in~\eqref{eq:GMIntersection}. Let $(V_6,V_5,A,\epsilon)$ be the $4$-dimensional Lagrangian datum over~$k$ that corresponds to this GM datum. By Proposition~\ref{prop:V5prime}, there exists a $5$-dimensional subspace $V_5^\prime \subset V_6$ such that $A \cap \wedge^3 V_5^\prime = 0$. As in Lemma~\ref{lem:LagrLift}, let $(\scrV_6,\scrV_5,\scrA,\epsilon)$ be a lift of $(V_6,V_5,A,\epsilon)$ over a dvr~$R$ of mixed characteristic~$(0,p)$. Choose an arbitrary lift $\scrV_5^\prime \subset \scrV_6$ of~$V_5^\prime$ to a direct summand of~$\scrV_6$, and note that $\scrA \cap \wedge^3 \scrV_5^\prime = 0$. By Proposition~\ref{prop:smoothcrit}, $\PP(A) \cap \Grass(3,V_6) = \emptyset$. If $K$ is the fraction field of~$R$ this implies that $\PP(\scrA_K) \cap \Grass(3,\scrV_{6,K}) = \emptyset$; so neither $\scrA_{\kbar} = A_{\kbar}$ nor~$\scrA_{\Kbar}$ contains decomposable vectors.  

Let $\scrX/R$ be the GM intersection defined by the GM datum that corresponds, via the bijection of Proposition~\ref{prop:GMLagCorrespondence}, with the Lagrangian datum $(\scrV_6,\scrV_5,\scrA,\epsilon)$. By the fiberwise criterion for smoothness together with Proposition~\ref{prop:smoothcrit} we see that $\scrX$ is smooth over~$R$, and of course its special fibre is~$X$. We can do exactly the same with $\scrV_5^\prime$ instead of~$\scrV_5$. If $(\scrV_6,\scrV_5^\prime,\scrW^\prime,q^\prime,\epsilon)$ is the GM datum over~$R$ corresponding with $(\scrV_6,\scrV_5^\prime,\scrA,\epsilon)$, we obtain a GM fivefold 
\begin{equation}
\label{eq:scrYprime}
\scrY^\prime = \bigcap_{v \in \scrV_6}\; \scrQ(v) \quad \subset \PP(\scrW)
\end{equation}
over~$R$. Let $\scrM^\prime =  \cap_{v \in \scrV_5^\prime}\; \scrQ(v) = \Grass(2,\scrV_5^\prime) \cap \PP(\scrW)$, and let $\scrX^\prime \to \scrM^\prime$ be a double cover branched along~$\scrY^\prime$. (There is, in general, not a unique double cover over~$R$, though any two such become isomorphic over an extension of~$R$. For the argument we may just choose any double cover of~$\scrM^\prime$ branched along~$\scrY^\prime$.)

By construction, the generic fibres $\scrX_K$ and~$\scrX^\prime_K$ are generalised partners. (To pass from~$\scrY^\prime_K$ to~$\scrX^\prime_K$, use what was explained in Section~\ref{subsec:GMLagchar0}.) It follows from Theorem~\ref{thm:genpartners} that there is a finite extension $K \subset K^\prime$ and an isomorphism $\alpha \colon \frh^4(\scrX_{K^\prime})\bigl(2\bigr) \isomarrow \frh^6(\scrX^\prime_{K^\prime})\bigl(3\bigr)$ in~$\CHM(K^\prime)$. Let $R^\prime$ be the normalisation of~$R$ in~$K^\prime$ (which is again a dvr), and let $k^\prime$ be the residue field of~$R^\prime$. By specialisation of the cycles that give the isomorphism~$\alpha$ and its inverse we obtain the desired isomorphism $\frh^4(X_{k^\prime})\bigl(2\bigr) \isomarrow \frh^6(\scrX^\prime_{k^\prime})\bigl(3\bigr)$ in~$\CHM(k^\prime)$.

\subsection{}
The argument in case of a GM fourfold~$X/k$ of Gushel type is similar but involves one extra step. In this case there exists, possibly after replacing~$k$ by a finite extension, a $3$-dimensional GM datum $(V_6,V_5,W,q,\epsilon)$ over~$k$ with associated GM intersection
\[
Y = \bigcap_{v \in V_6}\; Q(v) \quad \subset \PP(W)\, ,
\]
such that $X$ can be realised as a double cover of $\cap_{v \in V_5}\; Q(v) = \PP(W) \cap \Grass(2,V_5)$ branched along~$Y$. Let $(V_6,V_5,A,\epsilon)$ be the corresponding $3$-dimensional Lagrangian datum over~$k$. Now we proceed as in the previous case: by Proposition~\ref{prop:V5prime} there exists a $5$-dimensional subspace $V_5^\prime \subset V_6$ such that $A \cap \wedge^3 V_5^\prime = 0$, using Lemma~\ref{lem:LagrLift} we can lift $(\scrV_6,\scrV_5,\scrA,\epsilon)$ to a Lagrangian datum $(V_6,V_5,A,\epsilon)$ over a dvr~$R$ of mixed characteristic~$(0,p)$, and we choose an arbitrary lift $\scrV_5^\prime \subset \scrV_6$ of~$V_5^\prime$ to a direct summand of~$\scrV_6$. 

Let $(\scrV_6,\scrV_5,\scrW,q,\epsilon)$ be the GM datum over~$R$ that corresponds to the Lagrangian datum $(\scrV_6,\scrV_5,\scrA,\epsilon)$. Let $\scrY/R$ be the lift of~$Y$ given by this GM datum, and let~$\scrX$ be the lift of~$X$ to a GM fourfold over~$R$ by taking a double cover of $\PP(\scrW) \cap \Grass(2,\scrV_5)$ with branch divisor~$\scrY$. 

Similarly, let $(\scrV_6,\scrV_5^\prime,\scrW^\prime,q^\prime,\epsilon)$ be the GM datum over~$R$ that corresponds to $(\scrV_6,\scrV_5^\prime,\scrA,\epsilon)$, and define the GM fivefold~$\scrY^\prime$ over~$R$ as in~\eqref{eq:scrYprime}. Let~$\scrX^\prime$ be a double cover of~$\Grass(2,\scrV_5)$ branched along~$\scrY^\prime$. 

The generic fibres $\scrX_K$ and~$\scrX^\prime_K$ are again generalised partners, because $\scrY_K$ is a generalised partner of~$\scrY_K^\prime$ and $\scrX_K$ (resp.\ $\scrX^\prime_K$) is opposite to~$\scrY_K$ (resp.\ $\scrY^\prime_K$). (See Section~\ref{subsec:GMLagchar0} for this notion.) After passing to a finite extension $K \subset K^\prime$, Theorem~\ref{thm:genpartners} gives us an isomorphism $\frh^4(\scrX_{K^\prime})\bigl(2\bigr) \cong \frh^6(\scrX^\prime_{K^\prime})\bigl(3\bigr)$, and by specialisation we then obtain the desired isomorphism $\frh^4(X_{k^\prime})\bigl(2\bigr) \cong \frh^6(\scrX^\prime_{k^\prime})\bigl(3\bigr)$.

This completes the proof of Theorem~\ref{thm:MotivePartner} and thereby also the proof of Theorem~\ref{thm:TCGMcharp}.

\end{document}